\documentclass[a4paper,oneside,12pt]{amsart}
\usepackage{placeins} 

\usepackage{amsmath,amssymb,amsthm}
\usepackage{epsf}
\usepackage{epsfig}
\usepackage[english]{babel}

\addto\captionsenglish{}

\usepackage{mathrsfs}   

\usepackage{fancyhdr}           
\newcommand{\RR}{\mathbb R}
\newcommand{\CC}{\mathbb C}

\renewcommand{\Re}{\mathop{\rm Re}\nolimits}
\renewcommand{\Im}{\mathop{\rm Im}\nolimits}
\newcommand{\trace}{\mathop{\mathrm{trace}}\nolimits}
\newcommand{\arctanh}{\mathop{\mathrm{arctanh}}\nolimits}

\newcommand{\ii}{\mathrm{i}}

\newcommand{\manifold}[1]{\mathcal{#1}}
\newcommand{\M}{\manifold{M}}
\newcommand{\D}{\manifold{D}}

\newcommand{\vect}[1]{\mathrm{#1}} 
\newcommand{\x}{\vect{x}}
\newcommand{\y}{\vect{y}}
\newcommand{\va}{\vect{a}}
\newcommand{\vb}{\vect{b}}
\newcommand{\vc}{\vect{c}}
\newcommand{\vd}{\vect{d}}
\newcommand{\vX}{\vect{X}}
\newcommand{\vY}{\vect{Y}}
\newcommand{\vZ}{\vect{Z}}
\newcommand{\vW}{\vect{W}}
\newcommand{\vH}{\vect{H}}
\newcommand{\n}{\vect{n}}
\newcommand{\m}{\vect{m}}

\newcommand{\e}{\mathrm{e}}

\newtheorem{theorem}{Theorem}[section]
\newtheorem{prop}[theorem]{Proposition}

\theoremstyle{remark}
\newtheorem{remark}{Remark}[section]
\theoremstyle{definition}
\newtheorem{dfn}{Definition}[section]

\newcommand{\ds}{\displaystyle}

\makeatletter
\@addtoreset{equation}{section}
\makeatother

\begin{document}

\title
{Canonical Weierstrass Representations for Maximal Space-like Surfaces in $\RR^4_2$}%

\author{Georgi Ganchev and Krasimir Kanchev}

\address{Bulgarian Academy of Sciences, Institute of Mathematics and Informatics,
Acad. G. Bonchev Str. bl. 8, 1113 Sofia, Bulgaria}
\email{ganchev@math.bas.bg}%

\address {Department of Mathematics and Informatics, Todor Kableshkov University of Transport,
158 Geo Milev Str., 1574 Sofia, Bulgaria}%
\email{kbkanchev@yahoo.com}%


\begin{abstract}
It is known that any maximal space-like surface without isotropic points
in the four-dimensional pseudo-Euclidean space with 
neutral metric admits locally geometric parameters which are special case of isothermal parameters.  
With respect to such parameters the surface is determined uniquely up to a motion by the Gauss curvature 
and the curvature of the normal connection, which satisfy a system of two PDE's (the system of natural PDE's).

For any maximal space-like surface parametrized by canonical parameters we obtain a special
Weierstrass representation -- canonical Weierstrass representation. These Weierstrass formulas
allow us to solve explicitly the system of natural PDE's by virtue of
two holomorphic functions in the Gauss plane. We find the relation between two pairs of holomorphic
functions generating one and the same solution to the system of natural PDE's.

We establish a geometric correspondence between the maximal space-like surfaces of general type in $\RR^4_2$, 
the solutions to the system of natural PDE's and the pairs of holomorphic functions in the Gauss plane.
We prove that any maximal space-like surface in the four-dimensional pseudo-Euclidean space with neutral metric 
generates two maximal space-like surfaces in the three-dimensional Minkowski space and vice versa. 
\end{abstract}

\maketitle

\thispagestyle{empty}

\tableofcontents


\section{Introduction}

Let $\RR^4_2$ be the standard flat four-dimensional space, whose metric is with signature (2,2), also known as
a pseudo-Euclidean space with neutral metric. A two-dimensional surface $\M$ in $\RR^4_2$ is said to be 
\emph{space-like} if the induced metric on the tangential space at any point of $\M$ is positive definite.
A surface $\M$ in $\RR^4_2$ is said to be \textit{maximal} if the mean curvature vector field $H$ is identically 
zero. 

Maximal space-like surfaces in $\RR^4_2$ were studied in \cite{S}. Introducing locally special isothermal parameters
(which we call canonical parameters) 
on a maximal space-like surface $\M$ in $\RR^4_2$, it was proved in \cite{S} a theorem of Bonnet type: the surface 
$\M$ is determined uniquely up to a motion through the Gauss curvature $K$ and the curvature of the normal connection 
$\varkappa$, which satisfy the background system of PDE's. 

This paper is a model of our approach to the study of maximal space-like surfaces in $\RR^4_2$.
Briefly, the basic tools and results in the present work include the following:

  Let $\M=(\D,\x(u,v))$ be a space-like surface in $\RR^4_2$, parametrized by isothermal parameters $(u,v) \in \D \subset \RR^2$.
In Section \ref{sect_Phi} we describe the complex function $\Phi(t)=2\frac{\partial\x}{\partial t}$, where $t=u+\ii v$.
In Section \ref{sect_Phi_Psi} we characterize the maximal space-like 
surfaces in $\mathbb R^4_2$ trough the function $\Phi$. Further in Section \ref{sect_K_kappa-Phi} we express the invariants of 
the surface $\M$ in terms of this function $\Phi$. 
The main feature of our approach is the use of \emph{canonical parameters} on the maximal space-like surfaces,
which are described in Section \ref{sect_can-def}\,, 
and are characterized by the equalities ${\Phi}^{2}=0$ and ${\Phi'}^{2}=\pm 1$\,.
We find in Section \ref{sect_W_can} a canonical Weierstrass representation \eqref{W_Can1_polinom_R42} for the maximal space-like surfaces, 
which gives a maximal space-like surface parametrized by canonical parameters through two holomorphic functions in the Gauss plane. 
The formula for the canonical Weierstrass representation appears to be central from the point of view of effective applications. 
One of the important applications of this representation are the formulas \eqref{K_kappa_g1g2_Can1_R42},
which give us in an explicit form all the solutions of the system \eqref{Nat_Eq_K_kappa_R42} of natural PDE's 
for the maximal space-like surfaces in $\RR^4_2$.

\smallskip 

The philosophy of the application of the canonical Weierstrass representation includes consideration of the 
following classes of geometric objects:

1. The set of equivalence classes of maximal space-like surfaces of general type in $\RR^4_2$.

2. The set of equivalence classes of solutions to the background system of (system of natural) PDE's.

3. The set of equivalence classes of pairs of holomorphic functions in the Gauss plane.     

In Section \ref{sect_sol_nat_eq} we establish natural maps between the above three sets. Using Theorem \ref{Thm-Nat_Eq_K_kappa_R42} 
and Theorem \ref{Bone_Phi_bar_Phi_n1_n2_K_kappa_R42} we obtain a map from the set of classes of maximal space-like 
surfaces of general type in $\RR^4_2$ into the set of equivalence classes of solutions to the system of natural 
PDE's, which is a bijection. Using Theorem \ref{thm-W_Can1_polinom_R42} and results from Section \ref{sect_W_can} we obtain 
a map from the set of equivalence classes of pairs of holomorphic functions in the Gauss plane into the set of 
classes of maximal space-like surfaces of general type in $\RR^4_2$, which is a bijection.
Finally, using Theorem \ref{Nat_Eq_K_kappa_solv_g1g2_R42} and Theorem \ref{Nat_eq_same_K_kappa_hatg_g_R42}
we obtain a map from the set of equivalence classes of pairs of holomorphic functions in the Gauss plane into 
the set of equivalence classes of solutions to the system of natural PDE's, which is a bijection.

The result, which summarizes our investigations is the following:
\begin{figure}
		\centering
			\includegraphics[width=0.80\textwidth]{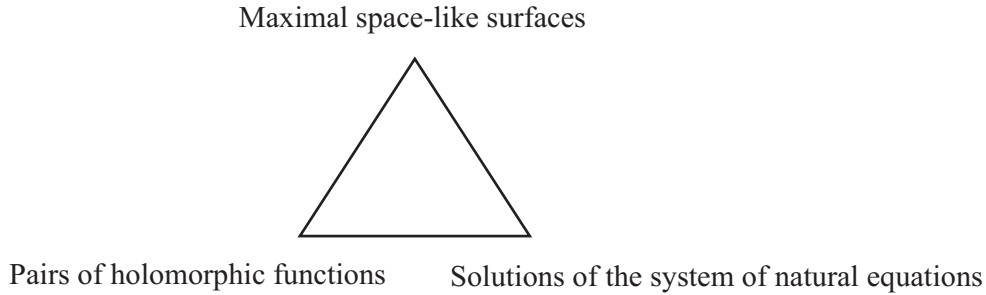}  
		\caption{Natural correspondence between the sets of basic objects}
		\label{Diag_MS_SNE_H_42}
\end{figure}\FloatBarrier

{\it The diagram in the Figure \ref{Diag_MS_SNE_H_42} is commutative and the three maps are bijections}\,. 

In \cite{G-K-1} and \cite{G-K-arXiv-1} we began to apply the above scheme of investigations to the study
of minimal surfaces in the four-dimensional Euclidean space $\RR^4$. Further, in \cite{G-K-2} and 
\cite{G-K-arXiv-2} we applied this scheme to the study of space-like surfaces with zero mean curvature 
vector field in Minkowski space-time. The present work completes the study of \textit{space-like} surfaces with 
$H=0$ in $\RR^4$, $\RR^4_1$ and $\RR^4_2$.

Finally, we note an interesting phenomenon, arising in the case of minimal surfaces in $\RR^4$ and of 
maximal space-like surfaces in $\RR^4_2$, but not in the case of space-like surfaces with zero mean curvature 
in $\RR^4_1$.

Considering the maximal space-like surfaces in the three-dimensional Minkowski space $\RR^3_1$, in Section \ref{sect_corresp_R42_R31} 
we establish the following natural correspondences:

Any maximal space-like surface in $\RR^4_2$ generates two maximal space-like surfaces in $\RR^3_1$ 
and vice versa. What is more, any solution to the system of natural PDE's of maximal space-like surfaces 
in $\RR^4_2$ generates two solutions to the natural PDE of maximal space-like surfaces in $\RR^3_1$ 
and vice versa. These facts give valuable information how the local geometry of maximal space-like surfaces 
in $\RR^3_1$ generates the local geometry of maximal space-like surfaces in $\RR^4_2$.

A similar fact about the relation between the theory of minimal surfaces in $\RR^4$ and the theory of 
minimal surfaces in $\RR^3$ was established in \cite{G-K-3} and \cite{G-K-arXiv-1}.

\section{Preliminaries}\label{sect_preliminaries}

 Let $\RR^4_2$ be the standard four-dimensional pseudo-Euclidean space with neutral metric, i.e.
the indefinite scalar product is given by the formula: 

\begin{equation}\label{R^4_2}
\va\cdot \vb=a_1b_1+a_2b_2-a_3b_3-a_4b_4\,.
\end{equation}

 Let $\M_0$ be a two-dimensional manifold and $\x : \M_0 \to \RR^4_2$ be an immersion of $\M_0$ in $\RR^4_2$.
Then $\M=(\M_0 ,\x)$ is called a regular surface in $\RR^4_2$. The immersion $\x$ is locally given by a vector function
$\x (u,v), \; (u,v) \in \D \subset \RR^2.$

The tangent space and the normal space of $\M$ at the point $p$ is denoted by $T_p(\M)$ and $N_p(\M)$, respectively.
The scalar product in $\RR^4_2$ induces a scalar product in $T_p(\M)$. If the induced scalar product is positive definite,
then the surface $\M$ is said to be \textit{space-like}. The induced scalar product on $N_p(\M)$ is negative definite.
  
We use the standard denotation for the first fundamental form on $\M$: 
\[
\mathbf{I}=E\,du^2+2F\,dudv+G\,dv^2,
\]
where $E=\x_u^2$, $F=\x_u\cdot\x_v$ and $G=\x_v^2$.

 It is well known that $\M$ admits locally isothermal coordinates characterized by the conditions
$E=G$ and $F=0$. Throughout this paper we always suppose that the local coordinates $(u,v)$ are
isothermal. Together with the real coordinates $(u,v)$ we also consider the complex coordinate
$t=u+\ii v$, identifying the coordinate plane $\RR^2$ with the Gauss plane $\CC$. Thus, using the standard 
imbedding of $\RR^4_2$ in $\CC^4$, all real functions locally defined on $\M$ are considered as complex 
functions of the complex variable $t$.

We also consider the complexified tangent space $T_{p,C}(\M)$ and the complexified normal space $N_{p,C}(\M)$ 
as subspaces in $\CC^4$.

For any two vectors $\va$ and $\vb$  in $\CC^4$, we denote by $\va\cdot \vb$ (or $\va\vb$) the bilinear 
scalar product in $\CC^4$, which is the natural extension of the product in $\RR^4_2$ given by \eqref{R^4_2}.
Together with the bilinear product in $\CC^4$ we also consider the indefinite Hermitian product of
$\va$ and $\vb$, which is given by the formula:
\[\va\cdot\bar \vb = a_1\bar b_1+a_2\bar b_2-a_3\bar b_3-a_4\bar b_4\, .\] 
The square of $\va$ with respect to the bilinear product is given by:
\[\va^2=\va\cdot \va = a_1^2+a_2^2-a_3^2-a_4^2\; .\]
The square of the norm of $\va$ with respect to the Hermitian product is: 
\[\|\va\|^2=\va\cdot\bar \va = |a_1|^2+|a_2|^2-|a_3|^2-|a_4|^2\, .\]
The complex spaces $T_{p,C}(\M)$ and $N_{p,C}(\M)$ are closed with respect to the complex conjugation
and are mutually orthogonal with respect to the bilinear product and also with respect to the Hermitian product.
Therefore we have the following orthogonal decomposition:
\[\CC^4 = T_{p,C}(\M) \oplus N_{p,C}(\M) .\]
For any vector $\va$ in $\CC^4$ $\va^\top$ and $\va^\bot$ denote the tangential and the normal component 
of $\va$, respectively. Then 
\[\va=\va^\top + \va^\bot \]
and this decomposition is the same with respect to both: the bilinear or the Hermitian product in $\CC^4$.

Let $\nabla$ be the canonical linear connection in $\RR^4_2$. For any tangent vector fields $\vX$, $\vY$ 
and normal vector field $\n$ the formulas of Gauss and Weingarten are as follows: 
\[\nabla_\vX \vY = \nabla^T_\vX \vY + \sigma(\vX,\vY)\,,\]
\[\nabla_\vX \n = -A_\n(\vX) + \nabla^N_\vX \n \,, \] 
where $\nabla^T$ is the Levi-Civita connection on $\M$; $\sigma(\vX,\vY)$  is the second fundamental form; 
$A_\n$ is the Weingarten map with respect to $\n$ and  $\nabla^N_\vX \n$ is the normal connection of $\M$.
The second fundamental form and the Weingarten map satisfy the equality
\[ A_\n\vX\cdot\vY = \sigma(\vX,\vY)\cdot \n \;. \]

 The curvature tensor $R$ on $\M$ and the curvature tensor $R^N$ of the normal connection are given by
\[ R(\vX,\vY)\vZ = \nabla^T_\vX \nabla^T_\vY \vZ - \nabla^T_\vY \nabla^T_\vX \vZ - \nabla^T_{[\vX,\vY]} \vZ\,,\]
\[ R^N(\vX,\vY)\n = \nabla^N_\vX \nabla^N_\vY \n - \nabla^N_\vY \nabla^N_\vX \n - \nabla^N_{[\vX,\vY]} \n \,,\]
respectively.
The covariant derivative of $\sigma$ is given by:
\[ (\overline\nabla_\vX \sigma)(\vY,\vZ) =
\nabla^N_\vX \sigma(\vY,\vZ)-\sigma(\nabla^T_\vX\vY,\vZ)-\sigma(\vY,\nabla^T_\vX\vZ)\,.\]

 The tensors $R$, $R^N$ and $\overline\nabla\sigma$ satisfy the following basic equations:

\noindent
 \emph{the Gauss equation}: 
\begin{equation}\label{Gauss}
R(\vX,\vY)\vZ\cdot\vW = \sigma(\vX,\vW)\sigma(\vY,\vZ) - \sigma(\vX,\vZ)\sigma(\vY,\vW)\, ;
\end{equation}
 \emph{the Codazzi equation}: 
\begin{equation}\label{Codazzi}
(\overline\nabla_\vX \sigma)(\vY,\vZ) = (\overline\nabla_\vY \sigma)(\vX,\vZ)\;; 
\end{equation}
 \emph{the Ricci equation}:
\begin{equation}\label{Ricci}
R^N(\vX,\vY)\n\cdot\m = (A_\n A_\m - A_\m A_\n)\vX\cdot\vY\ = [A_\n,A_\m]\vX\cdot\vY\, .
\end{equation}

 Further, $\vH$ , $K$ and $\varkappa$ denote the basic invariants of $\M$: the mean curvature, 
the Gauss curvature and the curvature of the normal connection.

Let $\vX_1$, $\vX_2$ be two orthonormal tangent vector fields on $\M$, and 
$\n_1$, $\n_2$ be two orthonormal normal vector fields such that the quadruple $(\vX_1,\vX_2,\n_1,\n_2)$
is a right oriented orthonormal frame field at any point $p\in\M$. Then
\begin{equation}\label{H-def}
\vH = \frac{1}{2}\trace\sigma = \frac{1}{2}(\sigma(\vX_1,\vX_1)+\sigma(\vX_2,\vX_2)) \,,
\end{equation}
\begin{equation}\label{K-def}
K = R(\vX_1,\vX_2)\vX_2\cdot\vX_1 \,,
\end{equation}
\begin{equation}\label{kappa-def}
\varkappa = R^N(\vX_1,\vX_2)\n_2\cdot\n_1 \,.
\end{equation}

The Gauss equation \eqref{Gauss}, implies the following formula for $K$:
\begin{equation}\label{K_sigma}
K = \sigma(\vX_1,\vX_1)\sigma(\vX_2,\vX_2) - \sigma^2(\vX_1,\vX_2)\,,
\end{equation}
 and the Ricci equation \eqref{Ricci}, implies the following expression for $\varkappa$:
\begin{equation}\label{kappa_A}
\varkappa = [A_{\n_2},A_{\n_1}] \vX_1\cdot \vX_2=A_{\n_1}\vX_1\cdot A_{\n_2}\vX_2-A_{\n_2}\vX_1\cdot A_{\n_1}\vX_2 \,.
\end{equation}

 Any space-like surface in $\RR^4_2$ with $\vH=0$ is said to be a \emph{maximal space-like surface}.

 As it is known from the theory of the surfaces in $\RR^4$, some of their geometric properties can be described 
in terms of their \emph{ellipse of the normal curvature}. Now we shall introduce this notion for the space-like surfaces
in $\RR^4_2$. First, let $\M$ be an arbitrary space-like surface in $\RR^4_2$ and $p\in\M$. Consider the
following subspace of the normal space $N_p(\M)$ of $\M$ at point $p$:
\begin{equation}\label{ellipse_curv-def_R42}
\mathscr{E}_p = \{ \sigma(\vX,\vX):\ \  \vX \in T_p(\M), \ \|\vX\|^2=1  \}.   
\end{equation}
Let $(\vX_1,\vX_2)$ be an orthonormal basis of $T_p(\M)$. Then any unit vector in $T_p(\M)$ can be 
represented as $\vX=\vX_1 \cos(\psi) + \vX_2 \sin(\psi)$. Replacing this expression in the definition of 
$\mathscr{E}_p$, we obtain the equality:
\begin{equation}\label{ellipse_curv_R42}
\sigma(\vX,\vX) = \vH +  
\frac{\sigma(\vX_1,\vX_1)-\sigma(\vX_2,\vX_2)}{2}\cos(2\psi) + \sigma(\vX_1,\vX_2)\sin(2\psi)\,.
\end{equation}
The last formula shows that $\mathscr{E}_p$ is an ellipse with center $\vH(p)$. Similarly to the theory of the 
surfaces in $\RR^4$, this ellipse is said to be the \emph{ellipse of the normal curvature} of $\M$ at $p$.
From here, we have:
\begin{prop}
A space-like surface $\M$ in $\RR^4_2$ is maximal if and only if at any point $p \in \M$, the ellipse of 
the normal curvature $\mathscr{E}_p$ is centered at the point $p$. 
\end{prop}


\section{The complex function $\Phi (t)$ }\label{sect_Phi}
Let $\M=(\D,\x(u,v))$ be a space-like surface in $\RR^4_2$, parametrized by arbitrary parameters $(u,v) \in \D$.
Define the complex function $\Phi(t)$ with values in $\CC^4$ by the following equality:
\begin{equation}\label{Phi_def}
\Phi(t)=2\frac{\partial\x}{\partial t}=\x_u-\ii\x_v \, .
\end{equation}

 The above equality implies that:
 \[ \Phi^2=(\x_u - \ii\x_v)^2=\x_u^2-\x_v^2-2\x_u \x_v \ii \,.\]
Then the following equalities are equivalent:
$$\Phi^2=0\,;\  \Leftrightarrow\  \begin{array}{l} \x_u^2-\x_v^2=0\,,\\  \x_u \x_v=0\,; \end{array} \ 
\Leftrightarrow \  \begin{array}{l} E=\x_u^2=\x_v^2=G\,,\\ F=0\,. \end{array}.$$
Hence, the parameters $(u,v)$ are isothermal if and only if 
\begin{equation}\label{Phi2} 
\Phi^2=0\,.
\end{equation}

 In the case of isothermal parameters, the norm of $\Phi$ is given by:
\[
\|\Phi\|^2=\Phi\bar{\Phi}=(\x_u - \ii\x_v)(\x_u + \ii\x_v)=\x_u^2+\x_v^2=E+G=2E=2G.
\]
Therefore the coefficients of the first fundamental form $\mathbf{I}$ are given by:
\begin{equation}\label{EG} 
E=G=\frac{1}{2}\|\Phi\|^2; \qquad F=0 \,,
\end{equation}
and
\begin{equation}\label{Idt}
\mathbf{I}=\frac{1}{2}\|\Phi\|^2 (du^2 + dv^2)=\frac{1}{2}\|\Phi\|^2|dt|^2 .
\end{equation}
It follows that $\Phi$ satisfies the inequality:
\begin{equation}\label{modPhi2} 
\|\Phi\|^2> 0\,.
\end{equation}

Differentiating equality \eqref{Phi_def} we get 
\begin{equation}\label{dPhi_dbt}
\frac{\partial\Phi}{\partial\bar t}=
\frac{\partial}{\partial\bar t}\, \left(2\, \frac{\partial \x}{\partial t} \right)=
\frac{1}{2}\Delta \x \,,
\end{equation}
where $\Delta$ is the Laplace operator in $\RR^2$.

 The last equality implies that $\ds\frac{\partial\Phi}{\partial\bar t}$ is a real valued function, i.e.

\begin{equation}\label{dPhi_dbt=dbPhi_dt}
\frac{\partial\Phi}{\partial\bar t}=\frac{\partial\bar\Phi}{\partial t} \, .
\end{equation} 

 Thus any function $\Phi$ given by \eqref{Phi_def} satisfies properties \eqref{Phi2}, \eqref{modPhi2} 
and \eqref{dPhi_dbt=dbPhi_dt}. These properties are characteristic for the function $\Phi$.

\begin{theorem}\label{x_Phi-thm}
Let the surface $\M=(\D,\x)$ in $\RR^4_2$ be parametrized by isothermal coordinates $(u,v)\in \D$ \, and \, $t=u+\ii v$.
Then the function $\Phi$, defined by \eqref{Phi_def} satisfies the conditions:
\begin{equation}\label{Phi_cond} 
\Phi^2=0; \quad \|\Phi\|^2> 0; \quad \frac{\partial\Phi}{\partial\bar t}=\frac{\partial\bar\Phi}{\partial t} \; .
\end{equation}

 Conversely, if $\Phi(t):\D \to \CC^4$ is a complex-valued  vector function, defined in $\D\subset\CC$
and satisfies \eqref{Phi_cond}, then there exists locally $\D_0\subset\D$ and a function $\x : \D_0 \to \RR^4_2$, 
such that $(\D_0,\x)$ is a space-like surface in $\RR^4_2$ parametrized by isothermal coordinates $(u,v)$ determined by
$t=u+\ii v$ and the function $\Phi(t)$ satisfies \eqref{Phi_def}.

The space-like surface $(\D_0,\x)$ is determined by $\Phi$ via equality \eqref{Phi_def} uniquely 
up to translation in $\RR^4_2$.
\end{theorem}

\begin{proof}
The first part of the statement follows from \eqref{Phi2}, \eqref{modPhi2} and \eqref{dPhi_dbt=dbPhi_dt}.
For the inverse, let the function $\Phi(t)$ satisfy \eqref{Phi_cond}.
Then
\[
2\frac{\partial\Phi}{\partial\bar t}=
\left(\frac{\partial}{\partial u}+\ii \frac{\partial}{\partial v}\right)(\Re(\Phi)+\ii\Im(\Phi))=
(\Re(\Phi)_u-\Im(\Phi)_v)+\ii (\Re(\Phi)_v+\Im(\Phi)_u) .
\]
The third condition in \eqref{Phi_cond} implies that the imaginary part  of $\ds\frac{\partial\Phi}{\partial\bar t}$ is $0$.
Hence, we have:
\[ 
\Re(\Phi)_v=-\Im(\Phi)_u \,.
\]
Therefore, it follows that there exists a sub-domain $\D_0\subset\D$ and a function $\x : \D_0 \to \RR^4_2$,
such that
\[ 
\Re(\Phi)=\x_u ; \quad -\Im(\Phi)=\x_v \,.
\]
The last equalities are equivalent to \eqref{Phi_def}. The first two conditions in \eqref{Phi_cond} imply
$\x_u\cdot\x_v=0$ and $\x_u^2=\x_v^2> 0$. Hence, $(\D_0,\x)$ is a space-like surface in $\RR^4_2$ 
parametrized by isothermal coordinates $(u,v)$. 

Note that $\x_u$ and $\x_v$ are determined uniquely from \eqref{Phi_def}. Consequently the function 
$\x(u,v)$ is determined up to an additive constant. 
\end{proof}

 Next we obtain the transformation formulas for $\Phi$ under a change of the isothermal coordinates 
and under a motion of $\M=(\D,\x)$ in $\RR^4_2$.

 Let us change the isothermal coordinates by $t=t(s)$ and denote by $\tilde\Phi(s)$ the new function. 
Then the function $t(s)$ is either holomorphic or anti-holomorphic. 

1. The case of a holomorphic change.
 
In view of the definition \eqref{Phi_def} we have:
\begin{equation*}
\tilde\Phi(s)=2\frac{\partial\x}{\partial s}=
2\frac{\partial\x}{\partial t}\frac{\partial t}{\partial s}+
2\frac{\partial\x}{\partial \bar t}\frac{\partial \bar t}{\partial s}=
2\frac{\partial\x}{\partial t}\frac{\partial t}{\partial s}+ 2\frac{\partial\x}{\partial \bar t} 0=
2\frac{\partial\x}{\partial t}\frac{\partial t}{\partial s} \; .
\end{equation*}
Hence
\begin{equation}\label{Phi_s-hol}
\tilde\Phi(s)=\Phi(t(s)) \frac{\partial t}{\partial s} \; .
\end{equation}

2. The case of an anti-holomorphic change

In this case we have:
\begin{equation*}
\tilde\Phi(s)=2\frac{\partial\x}{\partial s}=
2\frac{\partial\x}{\partial t}\frac{\partial t}{\partial s}+
2\frac{\partial\x}{\partial \bar t}\frac{\partial \bar t}{\partial s}=
2\frac{\partial\x}{\partial t} 0+ 2\frac{\partial\x}{\partial \bar t} \frac{\partial \bar t}{\partial s}=
2\frac{\partial\x}{\partial \bar t}\frac{\partial \bar t}{\partial s} \; .
\end{equation*}
Hence
\begin{equation}\label{Phi_s-antihol}
\tilde\Phi(s)=\bar\Phi(t(s)) \frac{\partial \bar t}{\partial s} \; .
\end{equation}
Especially, under the change $t=\bar s$, the function $\Phi$ is transformed as follows:
\begin{equation}\label{Phi_s-t_bs}
\tilde\Phi(s)=\bar\Phi(\bar s) \,.
\end{equation}

Now, let $\M=(\D,\x)$ and $\hat\M=(\D,\hat\x)$ be two surfaces in $\RR^4_2$, parametrized by isothermal coordinates 
$t=u+\ii v$ in the same domain $\D \subset \CC$. Suppose that $\hat\M$ is obtained from $\M$ by a motion 
(possibly improper) in $\RR^4_2$:
\begin{equation}\label{hat_M-M-mov}
\hat\x(t)=A\x(t)+\vb; \qquad A \in \mathbf{O}(2,2,\RR), \ \vb \in \RR^4_2 \,.
\end{equation}
Differentiating the above equality, we get the relation between the corresponding functions $\Phi$ and $\hat\Phi$: 
\begin{equation}\label{hat_Phi-Phi-mov}
\hat\Phi(t)=A\Phi(t); \qquad A \in \mathbf{O}(2,2,\RR) \,.
\end{equation}
Conversely, if $\Phi$ and $\hat\Phi$ satisfy \eqref{hat_Phi-Phi-mov}, then
$\hat\x_u=A\x_u$ and $\hat\x_v=A\x_v$ which imply \eqref{hat_M-M-mov}. 
Therefore \eqref{hat_M-M-mov} and \eqref{hat_Phi-Phi-mov} are equivalent.


\section{Characterization of maximal space-like surfaces in $\RR^4_2$ by means of the function $\Phi$ 
and its primitive function $\Psi$}\label{sect_Phi_Psi}

Let $\M=(\D,\x)$ be a surface in $\RR^4_2$ parametrized by isothermal coordinates $t=u+\ii v$,
and  $\Phi$ is the function, defined by \eqref{Phi_def}. Next we find the condition for $\Phi$ so that
the surface $\M$ to be maximal. We consider an orthonormal tangent basis $(\vX_1,\vX_2)$ in a usual way: 
\begin{equation}\label{vX1_vX2-def}
\vX_1=\frac{\x_u}{\|\x_u\|}=\frac{\x_u}{\sqrt E};\quad \vX_2=\frac{\x_v}{\|\x_v\|}=\frac{\x_v}{\sqrt G}=\frac{\x_v}{\sqrt E}\; .
\end{equation}

In view of \eqref{Phi_def} the coordinate vectors $\x_u$ and $\x_v$  are expressed by $\Phi$ as follows:
\begin{equation}\label{xuxv}
\begin{array}{ll}
\x_u=\ \ \,\Re (\Phi)=\ds\frac{1}{2}(\Phi+\bar\Phi),\\[2ex]
\x_v=-\Im (\Phi)=\ds\frac{-1}{2\ii}(\Phi-\bar\Phi)=\ds\frac{\ii}{2}(\Phi-\bar\Phi).
\end{array}
\end{equation}

 Differentiating \eqref{Phi2}, we find:
\begin{equation}\label{Phi.dPhi_dbt} 
\Phi\cdot\frac{\partial\Phi}{\partial\bar t}=0 \,.
\end{equation}
According to \eqref{dPhi_dbt} $\ds\frac{\partial\Phi}{\partial\bar t}$ is real and after a conjugation we have:
\begin{equation}\label{bPhi.dPhi_dbt} 
\bar\Phi\cdot\frac{\partial\Phi}{\partial\bar t}=0 \,.
\end{equation}
Since $\Phi$ and $\bar\Phi$ form a basis of $T_{p,\CC}(M)$ at any point $p\in\M$,
then \eqref{Phi.dPhi_dbt} and \eqref{bPhi.dPhi_dbt} imply that the vector $\ds\frac{\partial\Phi}{\partial\bar t}$
is orthogonal to $T_p(\M)$ and
\begin{equation}\label{dPhi_dbt_inN} 
\frac{\partial\Phi}{\partial\bar t} \in N_p(\M) \,.
\end{equation}
With the aid of \eqref{dPhi_dbt} it follows that
\begin{equation*}
\begin{array}{rl}\ds
\frac{\partial\Phi}{\partial\bar t}\!\! &=
\ds\left(\frac{\partial\Phi}{\partial\bar t}\right)^\bot=
\frac{1}{2}(\Delta \x )^\bot=
\frac{1}{2}(\x_{uu}+\x_{vv} )^\bot=
\frac{1}{2}(\nabla_{\x_u} \x_u + \nabla_{\x_v} \x_v )^\bot\\[2.5ex]
&=\ds\frac{1}{2}(\sigma(\x_u,\x_u) + \sigma(\x_v,\x_v) )=
E\; \frac{1}{2}(\sigma(\vX_1,\vX_1) + \sigma(\vX_2,\vX_2) ) = E\vH\; .
\end{array}
\end{equation*}
Thus we obtained
\begin{equation}\label{dPhi_dbt-Delta_x-EH}
\frac{\partial\Phi}{\partial\bar t}=\frac{1}{2}\Delta \x = E\vH\,.
\end{equation}

The above equalities imply the following statement.
\begin{theorem}\label{Max_x_Phi-thm}
let $\M=(\D,\x)$ be a space-like surface in $\RR^4_2$ parametrized by isothermal coordinates $(u,v)\in \D$,
and $\Phi(t)$ is the function \eqref{Phi_def}\,, defined in $\D$.

Then the following conditions are equivalent:
\begin{enumerate}
	\item The function $\Phi(t)$ is holomorphic: $\left( \ds\frac{\partial\Phi}{\partial\bar t}= 0 \right)$;
	\item The function $\x (u,v)$ is harmonic: $(\Delta \x = 0)$;
	\item $\M=(\D,\x)$ is maximal space-like surface: $(\vH=0)$\,.
\end{enumerate}
\end{theorem}

\smallskip

Equality \eqref{Phi_def} implies further:
\begin{equation}\label{dPhi_dt_bot}
\frac{\partial\Phi}{\partial t}=
\frac{\x_{uu} - \x_{vv}}{2} - \ii \x_{uv}\,;\qquad 
\left(\frac{\partial\Phi}{\partial t}\right)^\bot\!\!=
\frac{\sigma (\x_u,\x_u)-\sigma (\x_v,\x_v)}{2} - \ii \sigma (\x_u,\x_v)\,.
\end{equation}

 In the case of a maximal space-like surface, $\Phi$ is holomorphic, i.e. 
$\ds\frac{\partial\Phi}{\partial\bar t}=0$,
and we can write as usual $\ds\frac{\partial\Phi}{\partial t} = \Phi'$.

\smallskip
The maximality condition: 
$\vH=\frac{1}{2}(\sigma(\vX_1,\vX_1)+\sigma(\vX_2,\vX_2))=0$ gives that: 
\begin{equation}\label{sigma_22}
\sigma(\vX_2,\vX_2)=-\sigma(\vX_1,\vX_1)\,; \qquad  \sigma(\x_v,\x_v)=-\sigma(\x_u,\x_u)\,.
\end{equation}

 Thus, the formulas for $\Phi'$ and its orthogonal projection on $N_{p,\CC}(\M)$ are:
\begin{equation}\label{PhiPr}
\Phi^\prime=\frac{\partial\Phi}{\partial u}=\x_{uu}-\ii\x_{uv};\quad
\Phi^{\prime \bot}=\x_{uu}^\bot -\ii \x_{uv}^\bot =\sigma (\x_u,\x_u)-\ii\sigma (\x_u,\x_v).
\end{equation}

 Then $\sigma (\x_u,\x_u)$, $\sigma (\x_v,\x_v)$ and $\sigma (\x_u,\x_v)$ are expressed by $\Phi$ as follows:
\begin{equation}\label{sigma_uu_uv}
\begin{array}{l}
\sigma(\x_u,\x_u)=\ \ \,\Re (\Phi^{\prime \bot})=\ \ \,\ds\frac{1}{2}(\Phi^{\prime \bot}+\overline{\Phi^{\prime \bot}})=
\ \ \,\ds\frac{1}{2}(\Phi^{\prime \bot}+{\overline{\Phi^\prime}}^\bot),\\[2ex]

\sigma(\x_v,\x_v)=     -\Re (\Phi^{\prime \bot})=-\ds\frac{1}{2}(\Phi^{\prime \bot}+\overline{\Phi^{\prime \bot}})=
-\ds\frac{1}{2}(\Phi^{\prime \bot}+{\overline{\Phi^\prime}}^\bot),\\[2ex]

\sigma(\x_u,\x_v)=-\Im (\Phi^{\prime \bot})=
\ds\frac{-1}{2\ii}(\Phi^{\prime \bot}-\overline{\Phi^{\prime \bot}})=
\ \ \!\,\ds\frac{\ii}{2}(\Phi^{\prime \bot}-{\overline{\Phi^\prime}}^\bot).
\end{array}
\end{equation}

\medskip

According to Theorem \ref{Max_x_Phi-thm} in the case of a maximal surface $\M=(\D,\x)$, the function
$\x$ is harmonic. Then we can introduce its conjugate harmonic function
$\y$ by means of Cauchy-Riemann equations: $\y_u=-\x_v$ and $\y_v=\x_u$. 
Hence the complex-valued vector function $\Psi$:
\begin{equation}\label{Psi-def}
\Psi=\x+\ii\y ,
\end{equation}
is holomorphic and we have the following representation:
\begin{equation}\label{x_Phi_Psi}
\x=\Re\Psi ; \qquad \Phi=\x_u-\ii\x_v=\x_u+\ii\y_u=\frac{\partial\Psi}{\partial u}=\Psi' .
\end{equation}
 The maximality condition for $\M$ in terms of $\Psi$ is given by the next statement.
\begin{theorem}\label{Max_x_Psi-thm}
Let $\M=(\D,\x)$ be a maximal space-like surface in isothermal coordinates $(u,v)\in \D$.
Then $\x$ is locally given by:
\begin{equation}\label{x_Psi}
\x(u,v)=\Re\Psi(t) \,,
\end{equation}
where $\Psi$ is a holomorphic function of $t=u+\ii v$, satisfying the conditions:
\begin{equation}\label{Psi_cond}
\Psi'^{\,2}=0\,; \qquad \|\Psi'\|^2 > 0\,.
\end{equation}

 Conversely, let $\Psi$ be a holomorphic function, defined in the domain $\D\subset \CC$, and satisfying
\eqref{Psi_cond}. Then $(\D,\x)$, where $\x$ is defined by \eqref{x_Psi}, is a maximal space-like surface
and $(u,v)$ are isothermal coordinates.
\end{theorem}

\begin{proof}
If $\M=(\D,\x)$ is a maximal space-like surface, then $\Psi$ defined by \eqref{Psi-def}
satisfies \eqref{x_Phi_Psi}. The conditions \eqref{Psi_cond} are equivalent to the conditions
\eqref{Phi_cond} for $\Phi$.

 Conversely, if $\Psi$ is a holomorphic function satisfying the conditions \eqref{Psi_cond}, then
the equalities $\x=\Re\Psi$ and $\Phi=\Psi'$ imply that $\Phi=\x_u-\ii\x_v$ and according to Theorem
\ref{x_Phi-thm}\,the surface $(\D,\x)$ is a space-like surface in isothermal coordinates $(u,v)$. 
Now the condition $\x$ is harmonic implies that $(\D,\x)$ is maximal. 
\end{proof}

Now we shall find how the function $\Psi$ is  transformed under a change of the isothermal coordinates and 
under a motion of the surface $\M$ in $\RR^4_2$.

If $t=t(s)$ is a holomorphic change of the complex parameter we have $\x(t(s))=\Re\Psi(t(s))$. 
Since $\Psi(t(s))$ is holomorphic, in this case $\Psi(t)$ is transformed into $\Psi(t(s))$. 
Any anti-holomorphic change can be reduced to a holomorphic change after the special change $t=\bar s$. 
Under the latter change we have $\x(\bar s)=\Re\bar\Psi(\bar s)$ and therefore $\Psi (t)$ is transformed into
$\bar\Psi(\bar s)$. 
 
Let $\M=(\D,\x)$ and $\hat\M=(\D,\hat\x)$ be two maximal space-like surfaces in $\RR^4_2$, parametrized 
by isothermal coordinates. Then they are obtained one from the other by a motion (possibly improper) 
according to the formula $\hat\x(t)=A\x(t)+\vb$, where $A \in \mathbf{O}(\RR^4_2)$ and $\vb \in \RR^4_2$, 
if and only if when the corresponding functions $\Psi$ and $\hat\Psi$ are related by the equality
$\hat\Psi (t)=A\Psi (t)+\vb$.

 Theorem \ref{Max_x_Psi-thm} shows that we can obtain (at least locally) from one maximal space-like surface 
$\M=(\D,\x)$ other maximal space-like surfaces changing the function $\Psi$. 
For example, if $k>0$ is a constant, the function $k\Psi$ satisfies \eqref{Psi_cond}
and therefore we can apply \ref{Max_x_Psi-thm} for $(\D,\hat\x=\Re( k\Psi))$.
Thus, we obtained a new maximal space-like surface $\hat\M$ satisfying $\hat\x=k\x$.   
The last equality means that $\hat\M$ is obtained from $\M$ by means of a homothety with coefficient $k$.

Hence we have the statement:
\begin{prop}\label{Max_Surf_Sim}
Let $\M$ be a maximal space-like surface in $\RR^4_2$, parametrized by isothermal coordinates $(u,v)$.
If $\hat\M$ is obtained from $\M$ by means of similarity, then $\hat\M$ is also maximal space-like
surface and the coordinates $(u,v)$ are isothermal. 
\end{prop}
\begin{proof}
Any similarity in $\RR^4_2$ can be considered as a composition of a motion and a homothety in $\RR^4_2$.
It is clear that the notions \textit{maximal space-like surface} and \textit{isothermal coordinates}
are invariant under both: motion and homothety. This implies the assertion. 
\end{proof}

If the vector function $\x$ gives a maximal space-like surface, 
we can use its harmonic conjugate function $\y$ to obtain a new surface.
The equality $\x=\Re\Psi$ implies that $\y=\Re(-\ii\Psi)$. The function $-\ii\Psi$ together with $\Psi$
satisfies the conditions \eqref{Psi_cond} and Theorem \ref{Max_x_Psi-thm} implies that
$\y(u,v)$ determines a maximal space-like surface in $\RR^4_2$.

\begin{dfn}\label{Conj_Min_Surf} 
Let $\M=(\D,\x)$ a maximal space-like surface in $\RR^4_2$ parametrized by isothermal coordinates $(u,v)\in\D$.
If the function $\y$ is harmonically conjugate to the function $\x$, then the maximal space-like surface $\bar\M=(\D,\y)$ 
is said to be the \textbf{conjugate maximal space-like surface} to the given one. 
\end{dfn}

 It is well known that the above construction can be extended replacing the function $\Psi$ with the function 
$\e^{-\ii\theta}\Psi, \, \theta \in [0,2\pi)$. The new functions satisfy the conditions \eqref{Psi_cond} and 
determine a one-parameter family of maximal space-like surfaces:
\begin{equation}\label{1-param_family}
\x_\theta = \Re \e^{-\ii\theta}\Psi = \x\cos\theta + \y\sin\theta \,.
\end{equation}

\begin{dfn}\label{1-param_family_assoc_surf} 
Let $\M=(\D,\x)$ be a maximal space-like surface in $\RR^4_2$ parametrized by isothermal coordinates $(u,v)\in\D$.
The family of surfaces $\M_\theta=(\D,\x_\theta)$ determined by $\x_\theta$ from  \eqref{1-param_family}
is said to be an  
\textbf{associated family of maximal space-like surfaces} with the given one.
\end{dfn}

 Note that, the given surface and its conjugate one are obtained by $\theta = 0$ and $\theta=\frac{\pi}{2}$,
respectively. An essential property of this family is that any surface obtained by $\theta \neq 0$ is isometric to
the given one.
\begin{prop}\label{Isom_M_phi-M}
If $\M=(\D,\x)$ is a maximal space-like surface, parametrized by isothermal coordinates $(u,v)\in\D$,
and $\M_\theta=(\D,\x_\theta)$ are its conjugate maximal space-like surfaces, then the map
$\mathcal{F}_\theta: \x(u,v)\rightarrow \x_\theta(u,v)$ gives an isometry between $\M$ and $\M_\theta$ for any
$\theta >0$.
\end{prop}
\begin{proof}
It is enough to check that the coefficients of the first fundamental forms of $\M$ and $\M_\theta$
at the corresponding points are equal, i.e. $E_\theta = E$. 
Equality \eqref{x_Phi_Psi} implies that
\begin{equation}\label{Phi_1-param_family}
\Phi_\theta = (\e^{-\ii\theta}\Psi)' = \e^{-\ii\theta}\Phi; \qquad \|\Phi_\theta\|=\|\Phi\|\; .
\end{equation}
Taking into account \eqref{EG}, we obtain the assertion. 
\end{proof}


\section{Expressions for $K$ and $\varkappa$ in terms of the Weingarten operators and $\Phi$ }\label{sect_K_kappa-Phi}

 Let $\M=(\D,\x)$ be a maximal space-like surface parametrized by isothermal coordiantes $(u,v)\in\D$.
The unit vectors $\vX_1$ and $\vX_2$ are the normalized tangent vectors $\x_u$ and $\x_v$, respectively and 
the normal unit vectors $\n_1$, $\n_2$ form an orthonormal basis for $N_p(\M)$, such that the quadruple 
$(\vX_1,\vX_2,\n_1,\n_2)$ forms a right oriented frame at any $p\in\M$. Further, $A_{\n}$ denotes the 
Weingarten operator at $T_p(\M)$, corresponding to the normal vector $\n$. The condition $\vH=0$ means that 
$\trace A_{\n}=0 $ at any $p\in\M$. Then the matrices of the operators $A_{\n_1}$ and $A_{\n_2}$ have 
the following form:

\begin{equation}\label{A1A2_nu_lambda_rho_mu}
A_{\n_1}=
\left(
\begin{array}{rr}
\nu      &  \lambda\\
\lambda  & -\nu
\end{array}
\right); \qquad
A_{\n_2}=
\left(
\begin{array}{rr}
\rho &  \mu\\
\mu  & -\rho
\end{array}
\right).
\end{equation}
The second fundamental form $\sigma$ satisfies the conditions:
\begin{equation}\label{sigma_nu_lambda_rho_mu}
\begin{array}{l}
\sigma (\vX_1,\vX_1)=-(\sigma (\vX_1,\vX_1)\cdot \n_1)\n_1-(\sigma (\vX_1,\vX_1)\cdot \n_2)\n_2=
-\nu \n_1 - \rho \n_2 \,;\\
\sigma (\vX_1,\vX_2)=-(\sigma (\vX_1,\vX_2)\cdot \n_1)\n_1-(\sigma (\vX_1,\vX_2)\cdot \n_2)\n_2=
-\lambda \n_1 - \mu \n_2 \,; \\
\sigma (\vX_2,\vX_2)=-\sigma (\vX_1,\vX_1)= \nu \n_1 + \rho \n_2 \,.
\end{array}
\end{equation}

 The Gauss curvature $K$ of $\M$ in view of \eqref{Gauss} and \eqref{sigma_22} satisfies the equality:
\begin{equation}\label{K}
\begin{array}{rl}
K &=R(\vX_1,\vX_2)\vX_2\cdot \vX_1=\sigma(\vX_1,\vX_1)\sigma(\vX_2,\vX_2)-\sigma^2(\vX_1,\vX_2)\\
  &=-\sigma^2(\vX_1,\vX_1)-\sigma^2 (\vX_1,\vX_2).
\end{array}
\end{equation}
The expression for $K$ in terms of $A_{\n_1}$ and $A_{\n_2}$ follows from \eqref{K} and \eqref{sigma_nu_lambda_rho_mu}: 
\begin{equation}\label{K_nu_lambda_rho_mu}
K=-(-\nu^2-\rho^2)-(-\lambda^2-\mu^2)=\nu^2+\lambda^2+\rho^2+\mu^2=-\det(A_{\n_1})-\det(A_{\n_2})\,.
\end{equation}
\medskip
 In order to express $K$ by means of $\Phi$, we have from \eqref{PhiPr}:
\begin{equation}\label{PhiPr_X}
\Phi^{\prime \bot}=\sigma ({\sqrt E}{\vX_1},{\sqrt E}{\vX_1})-\ii\sigma ({\sqrt E}{\vX_1},{\sqrt E}{\vX_2})=
E(\sigma (\vX_1,\vX_1)-\ii\sigma (\vX_1,\vX_2)) \,.
\end{equation}
We get from here
$$
{\|\Phi^{\prime \bot}\|}^2=\Phi^{\prime \bot}\cdot\overline{\Phi^{\prime \bot}}
=E^2(\sigma^2(\vX_1,\vX_1)+\sigma^2 (\vX_1,\vX_2))
$$
and taking into account \eqref{EG} 
\begin{equation*}
\sigma^2(\vX_1,\vX_1)+\sigma^2 (\vX_1,\vX_2)=
\frac{{\|\Phi^{\prime \bot}\|}^2}{E^2}=\frac{4{\|\Phi^{\prime \bot}\|}^2}{\|\Phi\|^4}\;.
\end{equation*}
Hence
\begin{equation}\label{K_Phi}
K= \ds\frac{-4{\|\Phi^{\prime \bot}\|}^2}{\|\Phi\|^4}\;.
\end{equation}

\smallskip

 The vector function $\Phi^{\prime \bot}$ is not holomorphic in general and we shall find another representation of 
$\|\Phi^{\prime \bot}\|^2$ through holomorphic functions.

The condition $\Phi^2=0$ means that $\Phi$ and $\bar\Phi$ are orthogonal with respect to the Hermitian product in 
${\CC}^4$. Taking into account formulas \eqref{Phi_def} and \eqref{xuxv} we obtain that they form an orthogonal basis 
of $T_{p,\CC}(\M)$ at any $p\in\M$. Therefore the tangential projection of $\Phi^\prime$ is 
\[
\Phi^{\prime\top}
=\ds\frac{\Phi^{\prime\top}\cdot\bar \Phi}{\|\Phi\|^2}\Phi+\ds\frac{\Phi^{\prime\top}\cdot \Phi}{\|\bar \Phi\|^2}\bar \Phi 
=\ds\frac{\Phi' \cdot \bar \Phi}{\|\Phi\|^2}\Phi + \ds\frac{\Phi' \cdot \Phi}{\|\bar \Phi\|^2}\bar \Phi\,.
\]

 Differentiating $\Phi^2=0$ we find the relation:
\begin{equation}\label{Phi.dPhi_dt}
\Phi\cdot\Phi^\prime=0\,.
\end{equation}
Applying the last equality to $\Phi^{\prime\top}$, we find the projections of $\Phi'$:
\begin{equation}\label{Phipn}
\Phi^{\prime\top}= \ds\frac{\Phi' \cdot \bar \Phi}{\|\Phi\|^2}\Phi\,; \quad\quad \Phi^{\prime\bot}=\Phi'-\Phi^{\prime\top}=
\Phi'-\ds\frac{\Phi' \cdot \bar \Phi}{\|\Phi\|^2}\Phi\,.
\end{equation}
Direct computations show that:
\begin{equation}\label{mPhipn2}
{\|\Phi^{\prime\bot}\|}^2 = \ds\frac{\|\Phi\|^2\|\Phi'\|^2-|\bar \Phi \cdot \Phi'|^2}{\|\Phi\|^2}\ .
\end{equation}
Replacing into \eqref{K_Phi} we get:
\begin{equation}\label{K_Phi.Phip}
K= \ds\frac{-4(\|\Phi\|^2\|\Phi'\|^2-|\bar \Phi \cdot \Phi'|^2)}{\|\Phi\|^6}\;.
\end{equation} 
Denoting the bivector product of $\Phi$ and $\Phi'$ by $\Phi\wedge\Phi'$, we have:
\[\|\Phi\wedge\Phi'\|^2=\|\Phi\|^2\|\Phi'\|^2-|\bar \Phi \cdot \Phi'|^2.\] 
Hence
\begin{equation}\label{K_Phi_bv}
K= \ds\frac{-4\|\Phi\wedge\Phi'\|^2}{\|\Phi\|^6}\;.
\end{equation}

\smallskip

Now we shall find another relation between $K$ and $\Phi$ using that $K$ is expressed by the second derivatives 
of the coefficients of the first fundamental form. In isothermal coordinates this means that $K$ is
expressed by the second derivatives of $E$, which means by the second derivatives of   
$\|\Phi\|^2$ according to \eqref{EG}. In order to obtain this relation we use the orthogonal decomposition 
of $\Phi'$:
\begin{equation}\label{Phi_p-Ort_1}
\Phi'=\Phi^{\prime\top}+\Phi^{\prime\bot}=
\ds\frac{\Phi' \cdot \bar \Phi}{\|\Phi\|^2}\Phi + \Phi^{\prime\bot}\,.
\end{equation}
Using that $\bar\Phi$ is anti-holomorphic, we calculate:
\[
\Phi' \cdot \bar \Phi = \frac{\partial\Phi}{\partial t} \cdot \bar \Phi = 
\frac{\partial(\Phi \cdot \bar \Phi)}{\partial t}- \Phi \cdot \frac{\partial \bar \Phi}{\partial t}=
\frac{\partial(\|\Phi\|^2)}{\partial t}\,
\]
Taking into account that $\|\Phi\|^2=2E$, we obtain:
\begin{equation}\label{d_ln_E_dt}
\ds\frac{\Phi' \cdot \bar \Phi}{\|\Phi\|^2}=
\frac{\partial(\|\Phi\|^2)}{\partial t} \frac{1}{\|\Phi\|^2}=
\frac{\partial E}{\partial t} \frac{1}{E} = \frac{\partial\ln E}{\partial t} \;.
\end{equation}
Thus, the orthogonal decomposition \eqref{Phi_p-Ort_1} of $\Phi'$ gets the form:
\begin{equation*}
\Phi'=
\frac{\partial\ln E}{\partial t}\Phi + \Phi^{\prime\bot}.
\end{equation*}
Differentiating the last equality with respect to $\bar t$ in view of the fact that $\Phi'$ and $\Phi$
are holomorphic, we obtain:
\begin{equation*}
0=\frac{\partial^2\ln E}{\partial\bar t\partial t}\Phi + \frac{\partial (\Phi^{\prime\bot})}{\partial\bar t}\;.
\end{equation*}
Multiplying the last equality by $\bar\Phi$ we find:
\begin{equation}\label{Delta_lnE_Phi_1}
\frac{\partial^2\ln E}{\partial\bar t\partial t}\|\Phi\|^2 +
\frac{\partial (\Phi^{\prime\bot})}{\partial\bar t} \bar\Phi = 0\;.
\end{equation}
Since  $\frac{\partial}{\partial\bar t} \frac{\partial}{\partial t} = \frac{1}{4}\Delta$ and 
$\|\Phi\|^2=2E$, the first addend becomes:
\[
\frac{\partial^2\ln E}{\partial\bar t\partial t}\|\Phi\|^2=
\frac{E\Delta\ln E}{2}\;.
\]
The second addend is transformed as follows:
\[
\begin{array}{rl}
\ds\frac{\partial (\Phi^{\prime\bot})}{\partial\bar t} \bar\Phi 
            &=\ds\frac{\partial (\Phi^{\prime\bot}\cdot\bar\Phi)}{\partial\bar t} -
              \Phi^{\prime\bot} \frac{\partial \bar\Phi}{\partial\bar t}=
              0-\Phi^{\prime\bot}\overline{\Phi'}\\[1ex]
						&=-\Phi^{\prime\bot}\overline{\Phi'}{}^\bot=
						  -\Phi^{\prime\bot}\overline{\Phi^{\prime\bot}}=-\|\Phi^{\prime\bot}\|^2 \,.
\end{array}
\]
Now equality \eqref{Delta_lnE_Phi_1}, gets the form:
\begin{equation*}
\frac{E\Delta\ln E}{2} - \|\Phi^{\prime\bot}\|^2 = 0\;.
\end{equation*}
Replacing here $-\|\Phi^{\prime\bot}\|^2$ with $\frac{1}{4}K\|\Phi\|^4$
in view of \eqref{K_Phi} we find:
\begin{equation*}
2E\Delta\ln E + K\|\Phi\|^4 = 0\,,
\end{equation*}
or 
\begin{equation}\label{Delta_lnE_2K}
\frac{\Delta\ln E}{E} + 2K = 0\,.
\end{equation}
This is the classical Gauss equation, applied to a maximal space-like surface with respect to isothermal coordinates.
Finally we obtain one more way to express $K$ through $\Phi$. Formula \eqref{Delta_lnE_2K} implies:
\begin{equation}\label{K_Delta_lnE_lnPhi}
K = \frac{-\Delta\ln E}{2E} = \frac{-2\Delta\ln \|\Phi\|}{\|\Phi\|^2}\;.
\end{equation}

\medskip

Further, we express the curvature of the normal connection $\varkappa$
by means of the Weingarten operators and $\Phi$.

The Ricci equation \eqref{Ricci} and \eqref{A1A2_nu_lambda_rho_mu} imply

\begin{equation}\label{kappa}
\begin{array}{rl}
\varkappa &=R^N(\vX_1,\vX_2)\n_2\cdot \n_1 = 
						[A_{\n_2},A_{\n_1}] \vX_1\cdot \vX_2=A_{\n_1}\vX_1\cdot A_{\n_2}\vX_2-A_{\n_2}\vX_1\cdot A_{\n_1}\vX_2\\
					&=(\nu\vX_1 + \lambda\vX_2)\cdot (\mu\vX_1 - \rho\vX_2)-(\rho\vX_1 + \mu\vX_2)\cdot (\lambda\vX_1 - \nu\vX_2)\\
					&=\nu\mu - \lambda\rho  - (\rho\lambda - \mu\nu) = 2\nu\mu - 2\rho\lambda \,.
\end{array}
\end{equation}

Denote by $\det (\va,\vb,\vc,\vd)$ the fourth order determinant formed by the coordinates of the vectors 
$\va$, $\vb$, $\vc$ and $\vd$ with respect to the standard basis in ${\CC}^4$. With the aid of 
\eqref{sigma_nu_lambda_rho_mu} we have:
$$
\begin{array}{rl}
\det (\x_u,\x_v,\sigma(\x_u,\x_u),\sigma(\x_u,\x_v))&=
\det (\sqrt E \vX_1,\sqrt E \vX_2,\sigma(\sqrt E \vX_1,\sqrt E \vX_1),\sigma(\sqrt E \vX_1,\sqrt E \vX_2))\\
                                              &=E^3 \det (\vX_1,\vX_2,\sigma(\vX_1,\vX_1),\sigma(\vX_1,\vX_2))\\
                                              &=E^3 \det (\vX_1,\vX_2,-\nu\n_1 - \rho\n_2,-\lambda\n_1 - \mu\n_2)\\
																							&=E^3 \det (\vX_1,\vX_2,\nu\n_1,\mu\n_2)+E^3 \det (\vX_1,\vX_2,\rho\n_2,\lambda\n_1)\\
                                              &=E^3(\nu\mu - \rho\lambda)\det (\vX_1,\vX_2,\n_1,\n_2)=E^3(\nu\mu - \rho\lambda).$$
\end{array}
$$
or
\begin{equation}\label{numu2}
\nu\mu - \rho\lambda = \ds \frac{1}{E^3}\det (\x_u,\x_v,\sigma(\x_u,\x_u),\sigma(\x_u,\x_v)).
\end{equation}

In the last determinant we replace $\x_u$ and $\x_v$ from \eqref{xuxv} and get:
\begin{equation}\label{det1}
\begin{array}{l}
\ds \det (\x_u,\x_v,\sigma(\x_u,\x_u),\sigma(\x_u,\x_v))=\frac{\ii}{4}\det (\Phi+\bar\Phi,\Phi-\bar\Phi,\sigma(\x_u,\x_u),\sigma(\x_u,\x_v))\\
\ds =\frac{\ii}{4}\det (\Phi,-\bar\Phi,\sigma(\x_u,\x_u),\sigma(\x_u,\x_v))+ \frac{\ii}{4}\det (\bar\Phi,\Phi,\sigma(\x_u,\x_u),\sigma(\x_u,\x_v))\\
\ds =-\frac{\ii}{2}\det (\Phi,\bar\Phi,\sigma(\x_u,\x_u),\sigma(\x_u,\x_v)).
\end{array}
\end{equation}
Analogously, we replace $\sigma(\x_u,\x_u)$ and $\sigma(\x_u,\x_v)$ from \eqref{sigma_uu_uv} and find
\begin{equation}\label{det2}
\begin{array}{l}
\ds \det (\Phi,\bar\Phi,\sigma(\x_u,\x_u),\sigma(\x_u,\x_v))=-\frac{\ii}{2}\det (\Phi,\bar\Phi,\Phi^{\prime \bot},{\overline{\Phi^\prime}}^\bot).
\end{array}
\end{equation}
Equalities \eqref{det2} and \eqref{det1} imply that:
\begin{equation*}
\begin{array}{l}
\ds \det (\x_u,\x_v,\sigma(\x_u,\x_u),\sigma(\x_u,\x_v))=-\frac{1}{4}\det (\Phi,\bar\Phi,\Phi^{\prime \bot},{\overline{\Phi^\prime}}^\bot)\\
\ds =-\frac{1}{4}\det (\Phi,\bar\Phi,\Phi^\prime -\Phi^{\prime \top},\overline{\Phi^\prime}-{\overline{\Phi^\prime}}^\top).
\end{array}
\end{equation*}
Since the tangent vectors $\Phi^{\prime \top}$ and ${\overline{\Phi^\prime}}^\top$ are linear combinations of
of $\Phi$ and $\bar\Phi$, it follows that:
\begin{equation}\label{det3}
\begin{array}{l}
\ds \det (\x_u,\x_v,\sigma(\x_u,\x_u),\sigma(\x_u,\x_v))=
-\frac{1}{4}\det (\Phi,\bar\Phi,\Phi^\prime,\overline{\Phi^\prime}).
\end{array}
\end{equation}

Now, \eqref{kappa},  \eqref{numu2} and \eqref{det3}, give:

\begin{equation*}
\begin{array}{l}
\ds \varkappa = 2\nu\mu - 2\rho\lambda = \frac{2}{E^3}\det (\x_u,\x_v,\sigma(\x_u,\x_u),\sigma(\x_u,\x_v))=
-\frac{1}{2E^3}\det (\Phi,\bar\Phi,\Phi^\prime,\overline{\Phi^\prime}).
\end{array}
\end{equation*}
In view of \eqref{EG} we find the expression of $\varkappa$ by means of $\Phi$:
\begin{equation}\label{kappa2}
\begin{array}{l}
\ds \varkappa = -\frac{4}{\|\Phi\|^6}\det (\Phi,\bar\Phi,\Phi^\prime,\overline{\Phi^\prime}).
\end{array}
\end{equation}

Thus, we have that the Gauss curvature $K$ and the normal curvature $\varkappa$ of any maximal space-like surface $\M=(\D,\x)$, parametrized by isothermal parameters, satisfy the following formulas:
\begin{equation}\label{K_kappa_nu_lambda_rho_mu}
K= \nu^2+\lambda^2+\rho^2+\mu^2 \,; \quad\quad
\varkappa = 2\nu\mu - 2\rho\lambda \,.
\end{equation}
and
\begin{equation}\label{K_kappa_Phi_R42}
K= \ds\frac{-4{\|\Phi^{\prime \bot}\|}^2}{\|\Phi\|^4}
=\ds\frac{-4\|\Phi\wedge\Phi'\|^2}{\|\Phi\|^6}\,; \quad\quad
\varkappa = -\ds\frac{4}{\|\Phi\|^6}\det (\Phi,\bar\Phi,\Phi^\prime,\overline{\Phi^\prime}).
\end{equation}

Using the above formulas we shall see how the curvatures $K$ and $\varkappa$ are transformed under a change
of the coordinates and under the basic geometric transformations of the surface $\M$. Let us begin with 
the change of the coordinates $t=t(s)$.  Since $K$ and $\varkappa$ are scalar invariants, 
then we have
\begin{equation}\label{K_kappa_s_R42}
K(s)=K(t(s))\,; \qquad \varkappa (s) = \varkappa (t(s))\,.
\end{equation}
Suppose that the maximal space-like surface $\M=(\D,\x)$ is transformed into $\hat\M=(\D,\hat\x)$ by a
motion (possibly improper): 
\begin{equation}\label{hat_M-M-mov_R42}
\hat\x(t)=A\x(t)+\vb\,; \qquad A \in \mathbf{O}(2,2,\RR), \ \vb \in \RR^4_2 \,.
\end{equation}
Differentiating the above equality we have $\hat\Phi(t)=A\Phi(t)$ and.
\[
\det (A\Phi,A\bar\Phi,A\Phi^\prime,A\overline{\Phi^\prime})=
\det A \cdot \det (\Phi,\bar\Phi,\Phi^\prime,\overline{\Phi^\prime})\,.
\]
If the motion is proper ($\det A = 1$), then $\varkappa$ is preserved, while under an improper motion 
($\det A = -1$) the sign of $\varkappa$ is changed.

Thus under a proper motion we have:
\begin{equation}\label{hat_K_kappa-K_kappa-prop_mov_R42}
\hat K=K\,; \qquad \hat\varkappa = \varkappa \,.
\end{equation}
Under an improper motion we have:
\begin{equation}\label{hat_K_kappa-K_kappa-unprop_mov_R42}
\hat K=K\,; \qquad \hat\varkappa = -\varkappa \,.
\end{equation}

 Suppose that  the maximal space-like surface $\M=(\D,\x)$ is transformed into $\hat\M=(\D,\hat\x)$ by a 
homothety: $\hat\x(t)=k\x(t)$. From here it follows that $\hat\Phi=k\Phi$. Then \eqref{K_kappa_Phi_R42} implies
that:
\begin{equation}\label{hat_K_kappa-K_kappa-homotet_R42}
\hat K=\frac{1}{k^2}K\,; \qquad \hat\varkappa = \frac{1}{k^2}\varkappa \,.
\end{equation}

Finally, let us consider the family $\M_\theta$ of associated maximal space-like surfaces with $\M$ and the map 
$\mathcal{F}_\theta: \x(t)\rightarrow \x_\theta(t)$. According to Proposition \ref{Isom_M_phi-M}\, this map 
is an isometry and $K_\theta(t)=K(t)$. Since $\Phi_\theta = \e^{-\ii\theta}\Phi$, then it follows from 
\eqref{K_kappa_Phi_R42} that $\varkappa$ is also invariant under $\mathcal{F}_\theta$. Thus we have:
\begin{equation}\label{K_kappa-1-param_family_R42}
K_\theta(t)=K(t)\,; \qquad \varkappa_\theta(t) = \varkappa(t) \,.
\end{equation}


\section{Canonical coordinates on maximal space-like surfaces in $\RR^4_2$.}\label{sect_can-def}

  Up to now, we have considered maximal space-like surfaces in $\RR^4_2$ with respect to arbitrary isothermal coordinates. 
It is known that the maximal space-like surfaces in $\RR^3_1$ with $K \neq 0$ carry locally special isothermal coordinates,
having additional properties. In \cite{G} it is shown that in a neighborhood of a point in which $K \neq 0$, 
can be introduced coordinates which are \textit{principal} and \textit{isothermal}. Using the standard denotations
for the coefficients of the second fundamental form $L$, $M$ and $N$, this means:
$E=G$, $F=0$ and $M=0$. Moreover, these coordinates can be normalized in such a way that $L=-N=1$. These properties of 
the special coordinates determine them up to renumbering and directions of the coordinate lines. Further we call 
these special coordinates \textit{canonical coordinates}. In \cite{G} it is also shown that the surfaces 
in $\RR^3_1$ with $K \neq 0$ carry locally special coordinates which are \textit{asymptotic} and \textit{isothermal}.
They are characterized by the conditions $L=N=0$ and $M=1$.

In terms of the notions and the denotations in $\RR^4_2$ 
the principal canonical coordinates for surfaces in $\RR^3_1$ are characterized by the conditions:
\begin{equation}\label{Can_Princ_R31}
\sigma^2(\x_u,\x_u)=-1\,; \qquad \sigma(\x_u,\x_v)=0\,,
\end{equation}
while the asymptotic canonical coordinates are characterized by the conditions:
\begin{equation}\label{Can_Asimpt_R31}
\sigma^2(\x_u,\x_u)=0\,; \qquad \sigma(\x_u,\x_v)=-1\,.
\end{equation}

In order to express the above conditions in terms of the function $\Phi$, we consider the scalar square 
of the second equality in \eqref{PhiPr}:
\begin{equation}\label{PhiPr_bot^2}
{\Phi^{\prime \bot}}^2=\sigma^2(\x_u,\x_u)-\sigma^2(\x_u,\x_v)-2\ii\,\sigma(\x_u,\x_u)\cdot\sigma(\x_u,\x_v).
\end{equation}
Now it is clear that the conditions \eqref{Can_Princ_R31} are equivalent to ${\Phi^{\prime \bot}}^2=-1$\,,
while the conditions \eqref{Can_Asimpt_R31} are equivalent to ${\Phi^{\prime \bot}}^2=1$\,.
These equalities show how to introduce \textit{canonical coordinates} on maximal space-like surfaces in 
$\RR^4_2$. Coordinates with these properties are introduced in \cite{S} by the Cartan moving frame method.
In the present work we give an alternative approach to the canonical coordinates on the base of the function $\Phi$.

 In $\RR^4_2$ we can introduce canonical coordinates, which are an analogue of the principal canonical coordinates 
in $\RR^3_1$ in the following way:
\begin{dfn}\label{Can1-def}
Let $\M$ be a maximal space-like surface in $\RR^4_2$, parametrized by isothermal coordinates $(u,v)$.
These coordinates are said to be \textbf{canonical coordinates of the first type} if the function 
$\Phi$ satisfies the condition ${\Phi^{\prime \bot}}^2=-1$\,.
\end{dfn}
The canonical coordinates, which are an analogue of the asymptotic isothermal coordinates in $\RR^3_1$ 
are introduced by
\begin{dfn}\label{Can2-def}
Let $\M$ be a maximal space-like surface in $\RR^4_2$, parametrized by isothermal coordinates $(u,v)$. 
These coordinates are said to be \textbf{canonical coordinates of the second type}, if the function 
$\Phi$ satisfies the condition ${\Phi^{\prime \bot}}^2=1$\,.
\end{dfn}

 With the aid of equality \eqref{PhiPr_bot^2} we characterize the canonical coordinates by means of the second 
fundamental form $\sigma$:
\begin{prop}\label{Can1-sigma}
Let $\M$ be a maximal space-like surface parametrized by isothermal coordinates $(u,v)$. These coordinates 
are canonical coordinates of the first type (the second type) if and only if the second fundamental form
$\sigma$ satisfies the conditions:
\begin{equation}\label{sigma_uv_can1}
\sigma (\x_u,\x_u)\bot \: \sigma(\x_u,\x_v)\,,\qquad \sigma^2(\x_u,\x_u)-\sigma^2(\x_u,\x_v)=\mp 1\,,
\end{equation}
where the sign "$-$" refers to the case of canonical coordinates of the first type, while the sign "$+$" 
refers to the case of canonical coordinates of the second type.  
\end{prop}

 The above definitions \ref{Can1-def} and \ref{Can2-def} of canonical coordinates in terms of the function $\Phi$, 
look purely analytical, but the Proposition \ref{Can1-sigma} shows that these coordinates are geometric.
Since the conditions \eqref{sigma_uv_can1} are invariant under an arbitrary motion $A \in \mathbf{O}(2,2,\RR)$
in $\RR^4_2$, then the canonical coordinates are also invariant under such a motion.

Another proof of the last property can be obtained directly from the function $\Phi$. 
Namely, it follows from \eqref{hat_Phi-Phi-mov}\,, consecutively
\begin{equation}\label{hat_PhiPr-PhiPr-mov}
\hat\Phi'=A\Phi'; \qquad  
\hat\Phi^{\prime \bot} = A\Phi^{\prime \bot}; \qquad
\left.\hat\Phi^{\prime \bot}\right.^2 = \left.\Phi^{\prime \bot}\right.^2 ; \qquad
A \in \mathbf{O}(2,2,\RR) \,,
\end{equation}
which shows that ${\Phi^{\prime \bot}}^2$ is invariant under a motion in $\RR^4_2$. 

Thus we have:
\begin{theorem}\label{Can_Move}
Let the maximal space-like surface $\hat\M$ be obtained from $\M$ by a motion in $\RR^4_2$.
If $(u,v)$ are canonical coordinates of the first type (the second type) on $\M$, then they are also
canonical coordinates of the first type (the second type) on $\hat\M$.
\end{theorem}


 To prove the local existence of canonical coordinates on the maximal space-like surface $\M$ we 
consider how ${\Phi^{\prime\bot}}^2$ is transformed under a change of the isothermal coordinates.
Let $(u,v)$ be isothermal coordinates on the maximal space-like surface $\M$. In terms of 
the complex parameter $t=u+\ii v$ let us make the change $t=t(s)$, where $s$ is a new complex parameter, 
which determines new isothermal coordinates. Denote by $\tilde\Phi(s)$ the complex function, corresponding to
the new coordinates. The map under a change of the isothermal coordinates is either holomorphic or anti-holomorphic.

In the case of a holomorphic map we have $\tilde\Phi=\Phi t'$ from \eqref{Phi_s-hol} and
$\tilde\Phi'_s=\Phi'_t t'^{\,2}+\Phi t'\,\!'$. Since $\Phi(t)$ is tangent to $\M$ at the any point, then $\Phi^\bot=0$
and therefore:
\begin{equation}\label{tildPhiPr2}
\tilde\Phi_s'^\bot = \Phi_t'^\bot t^{\prime \, 2}\,; \qquad 
\left.\tilde\Phi_s^{\prime\bot}\right.^2 = {\Phi_t^{\prime \bot}}^2 t'^{\,4}\,.
\end{equation}

In the case of an anti-holomorphic map it is enough to consider the special change $t=\bar s$.
According to \eqref{Phi_s-t_bs} we have $\tilde\Phi(s)=\bar\Phi(\bar s)$ and
$\tilde\Phi_s'(s)=\overline{\Phi_t'}(\bar s)$.
Now it follows that: 
\begin{equation}\label{tildPhiPr2-t_bs}
\tilde\Phi_s'^\bot(s) = \overline{\Phi_t'^\bot(\bar s)}\;; \qquad
\left.\tilde\Phi_s^{\prime\bot}\right.^2\!(s) = \overline{{\Phi_t^{\prime \bot}}^2 (\bar s)}\;.
\end{equation}
  
	Note that if ${\Phi_t^{\prime\bot}}^2=0$, formulas \eqref{tildPhiPr2} and \eqref{tildPhiPr2-t_bs} 
imply that $\left.\tilde\Phi_s^{\prime\bot}\right.^2=0$. This means that the condition 
${\Phi^{\prime\bot}}^2=\pm 1$ is impossible. Hence, the points in which ${\Phi^{\prime\bot}}^2=0$, 
have to be considered separately. 

We give the following
\begin{dfn}\label{DegP-def}
Let $\M$ be a maximal space-like surface, parametrized by isothermal coordinates $(u,v)$. The point 
$p \in \M$ is said to be a \textbf{degenerate point}, if the function $\Phi$ satisfies the condition 
${\Phi^{\prime \bot}}^2(p)=0$\, at this point.
\end{dfn}

With the aid of \eqref{PhiPr_bot^2} we describe the degenerate points by means of $\sigma$:
\begin{prop}\label{DegP-sigma}
Let $\M$ be a maximal space-like surface, parametrized by isothermal coordinates $(u,v)$. 
The point $p \in \M$ is degenerate if and only if the second fundamental form $\sigma$ at this point 
satisfies the conditions:
\begin{equation}\label{sigma_uv_DegP}
\sigma (\x_u,\x_u)\bot \: \sigma(\x_u,\x_v)\,; \qquad \sigma^2(\x_u,\x_u)=\sigma^2(\x_u,\x_v)\,.
\end{equation} 
\end{prop}

 The definition of a degenerate point by means of the function $\Phi$ has a geometric character. 
Namely, we have:
\begin{theorem}\label{DegP_Change_Move}
 Let $\M$ be a maximal space-like surface. The property of a point $p\in \M$ to be degenerate 
does not depend on the isothermal coordinates. Moreover, this property is invariant under a motion of
$\M$ in $\RR^4_2$.
\end{theorem}
\begin{proof}
The property of a point $p\in \M$ to be degenerate does not depend on the choice of isothermal coordinates
as a consequence from \eqref{tildPhiPr2} and \eqref{tildPhiPr2-t_bs}.

Since the second fundamental form $\sigma$, as well as the conditions \eqref{sigma_uv_DegP} are invariant 
under a motion in $\RR^4_2$, then the property of a point $p\in \M$ to be degenerate is also invariant
under a motion. 
\end{proof}

 As we have already mentioned, the function $\Phi^{\prime \bot}$,
is not holomorphic in general, but it occurs that the scalar square ${\Phi^{\prime \bot}}^2$
is always a holomorphic function. To prove this, let us consider again equalities \eqref{Phipn}.  
Taking a square of the second equality we find:
\[
{\Phi^{\prime\bot}}^2 =
{\Phi'}^2-2\Phi'\ds\frac{\Phi' \cdot \bar \Phi}{\|\Phi\|^2}\Phi+\left(\ds\frac{\Phi' \cdot \bar \Phi}{\|\Phi\|^2}\right)^2 \Phi^2.
\]
Applying $\Phi^2=0$ and $\Phi\cdot\Phi^\prime=0$ to the above equality, we get:
\begin{equation}\label{Phi_prim_bot^2-Phi_prim^2}
{\Phi^{\prime \bot}}^2={\Phi^\prime}^2 .
\end{equation}
It follows from here that ${\Phi^{\prime \bot}}^2$ is a holomorphic function, because $\Phi^\prime$ and
therefore ${\Phi^\prime}^2$ are holomorphic.

Using properties of the set of zeroes of a holomorphic function, we can formulate:
\begin{theorem}\label{Min_Surf_DegP}
If $\M$ is a connected maximal space-like surface in $\RR^4_2$, then either all of the points of $\M$ 
are degenerate, or the set of the degenerate points of $\M$ is countable without points of condensation in $\M$.
\end{theorem}
Finally we characterize the degenerate points in terms of the ellipse of the normal curvature, which was introduced by
\eqref{ellipse_curv-def_R42}. If the space-like surface $\M$ is maximal, then $\vH=0$\, and 
formula \eqref{ellipse_curv_R42} gives that $\sigma(\vX_2,\vX_2)=-\sigma(\vX_1,\vX_1)$ and consequently 
\eqref{ellipse_curv_R42} has the following more simple form:
\begin{equation}\label{ellipse_curv_min_surf_R42}
\sigma(\vX,\vX) = \sigma(\vX_1,\vX_1)\cos(2\psi) + \sigma(\vX_1,\vX_2)\sin(2\psi)\,.
\end{equation}
Taking the square of the last equality, we find:
\begin{equation}\label{ellipse_curv_min_surf_2_R42}
\begin{array}{l}
\sigma^2(\vX,\vX) = \ds\frac{\sigma^2(\vX_1,\vX_1)+\sigma^2(\vX_1,\vX_2)}{2} \ + \\[2.0ex]
\ds\frac{\sigma^2(\vX_1,\vX_1)-\sigma^2(\vX_1,\vX_2)}{2}\cos(4\psi) + \sigma(\vX_1,\vX_1)\sigma(\vX_1,\vX_2)\sin(4\psi)\,.
\end{array}
\end{equation}
The ellipse $\mathscr{E}_p$ is a circle, if and only if $\sigma^2(\vX,\vX)$ does not depend on $\psi$.
This means that the coefficients before $\cos(4\psi)$ and $\sin(4\psi)$ in the formula \eqref{ellipse_curv_min_surf_2_R42}
are zero. Thus we obtained the following property: $\mathscr{E}_p$ is a circle if and only if for any orthonormal basis 
$(\vX_1,\vX_2)$ of $T_p(\M)$ the following conditions are valid:
\begin{equation}\label{sigma_X12_super_conf}
\sigma (\vX_1,\vX_1)\bot \: \sigma(\vX_1,\vX_2)\,; \qquad \sigma^2(\vX_1,\vX_1)=\sigma^2(\vX_1,\vX_2)\,.
\end{equation} 
Suppose that $t=u+\ii v$ are isothermal coordinates on $\M$ in a neighborhood of $p$ and let $(\vX_1,\vX_2)$ be the 
normalized coordinate vectors $(\x_u,\x_v)$. Using again equalities $\sigma (\x_u,\x_u)=E\sigma(\vX_1,\vX_1)$ and
$\sigma (\x_u,\x_v)=E\sigma(\vX_1,\vX_2)$, we obtain that the conditions \eqref{sigma_X12_super_conf} are equivalent 
to the conditions \eqref{sigma_uv_DegP} for the point $p$ to be degenerate. Thus we have:
\begin{prop}\label{DegP-super_conf_R42}
If $\M$ is a maximal space-like surface in $\RR^4_2$, then a point $p \in \M$ is degenerate if and only the ellipse 
of the normal curvature $\mathscr{E}_p$ is a circle.
\end{prop}

Now, let us consider the question of existence of canonical coordinates.
Since the properties of both types of canonical coordinates are analogous, from now on we consider
canonical coordinates of the first type and call them simply canonical coordinates.
 We give the following 
\begin{dfn}\label{Min_Surf_Gen_Typ-def}
The maximal space-like surface $\M$ in $\RR^4_2$ is said to be of \textbf{general type} if it has no 
degenerate points.
\end{dfn}
\begin{theorem}\label{Can_Coord-exist}
If $\M$ is a maximal space-like surface in $\RR^4_2$ of general type, then it admits locally canonical coordinates.
\end{theorem}
\begin{proof}
Let the surface $\M$ be parametrized by isothermal coordinates $u + \ii v=t$ in a neighborhood of a point
$p \in \M$. Next we show how to introduce new isothermal coordinates $t=t(s)$, where $t(s)$ is a
holomorphic function, so that the new coordinates to be canonical. Taking into account Definition \ref{Can1-def} 
and formula \eqref{tildPhiPr2} it follows that the coordinates, determined by $s$, are canonical if:
\begin{equation}\label{condcan}
{\Phi_t^{\prime \bot}}^2 t'^{\,4} = \left.\tilde\Phi_s^{\prime\bot}\right.^2 = - 1 \,.
\end{equation}
Hence the function $t=t(s)$ satisfies the following ordinary complex differential equation of the first order:
\begin{equation*}
\sqrt[4]{-\,{\Phi_t^{\prime \bot}}^2}\:dt = ds\,.
\end{equation*}
In view of \eqref{Phi_prim_bot^2-Phi_prim^2} we can replace
${\Phi_t^{\prime \bot}}^2$ with ${\Phi_t^\prime}^2$ and we get:
\begin{equation}\label{eqcan}
ds = \sqrt[4]{-\,{\Phi_t^\prime}^2}\:dt\,.
\end{equation}
Since ${\Phi_t^\prime}^2$ is a holomorphic function of $t$, by integrating we find:
\begin{equation}\label{eqcan-sol}
s = \int\sqrt[4]{-\,{\Phi_t^\prime}^2}\:dt\,.
\end{equation} 
Since $\M$ has no degenerate points, then ${\Phi_t^\prime}^2\ne 0$\,. Therefore \eqref{eqcan-sol} 
determines $s$ as a holomorphic locally reversible function of $t$. This means that $s$ determines new 
isothermal coordinates in a neighborhood of the point in consideration. The equation \eqref{eqcan} 
is equivalent to \eqref{condcan} and consequently, the coordinates determined by $s$, are canonical. 
\end{proof}

Next we consider the question of the uniqueness of the canonical coordinates.
\begin{theorem}\label{Can_Coord-uniq}
Let $\M$ be a maximal space-like surface of general type, and let $t$ and $s$ are complex variables,
which determine canonical coordinates in a neighborhood of a point $p \in M$. If $t$ and $s$ generate one 
and the same orientation on $\M$, then they satisfy one of the following relations:
\begin{equation}\label{uniq-holo}
t=\pm s+c\,; \quad t=\pm \ii s+c\,.
\end{equation}
If $t$ and $s$ generate different orientations on $\M$, then they satisfy one of the following relations:
\begin{equation}\label{uniq-antiholo}
t=\pm \bar s + c\,; \quad t=\pm \ii \bar s + c\,.
\end{equation}
In the above equalities $c$ is an arbitrary complex constant. 
\end{theorem}
\begin{proof}
First we consider the case of the same orientation, which means that $t$ is a holomorphic function
of $s$. Then we can apply formula \eqref{tildPhiPr2}. By the conditions of the theorem $t$ and $s$
determine canonical coordinates and 
${\Phi_t^{\prime \bot}}^2 = \left.\tilde\Phi_s^{\prime\bot}\right.^2 = - 1$\,.
Consequently, \eqref{tildPhiPr2} is reduced to $- 1 = - 1 t'^4$. This implies that $t'^4 = 1$ and 
$t'(s) = \pm 1\,;\ \pm \ii$\,. The last equation is equivalent to \eqref{uniq-holo}.

 In the case of different orientations $t$ is an anti-holomorphic function of $s$. Introducing an additional
variable $r$ by the formula $r=\bar s$, then we can apply formula \eqref{tildPhiPr2-t_bs} to $r$ and $s$, 
which means that the variable $r$ also determines canonical coordinates. Since $t$ is a holomorphic 
function of $r$, then $t$ and $r$ satisfy one of the relations \eqref{uniq-holo}. 
This implies that $t$ and $s$ satisfy one of the relations \eqref{uniq-antiholo}. 
\end{proof}

\begin{remark}
Geometrically, the above eight relations \eqref{uniq-holo} and \eqref{uniq-antiholo} mean that
the canonical coordinates are unique up to renumbering and orientation of the parametric lines\,.
\end{remark}

\smallskip

 Next we shall find the relations between the degenerate points, canonical coordinates and the curvatures 
$K$ and $\varkappa$. First we shall characterize the degenerate points in the sense of Definition \ref{DegP-def}
by means of $K$ and $\varkappa$. We have:
\begin{theorem}\label{super_conf_K_kappa_R42}
If $\M$ is a maximal space-like surface in $\RR^4_2$, then the Gauss curvature $K$ and the curvature of the
normal connection $\varkappa$ of $\M$ satisfy the inequality:
\begin{equation}\label{K>=kappa}
K \geq |\varkappa | \,,
\end{equation}
with equality only in the degenerate points of $\M$.
\end{theorem}
\begin{proof}
Taking into account equalities \eqref{K_kappa_nu_lambda_rho_mu} we get consecutively:
\begin{equation*}\label{K+-kappa}
K\pm\varkappa = \nu^2+\lambda^2+\rho^2+\mu^2 \pm 2\nu\mu \mp 2\rho\lambda 
= (\nu \pm \mu)^2 + (\rho \mp \lambda)^2 \geq 0\,.
\end{equation*}
The last inequality is equivalent to \eqref{K>=kappa} with equality only if
\begin{equation*}\label{nu+-mu}
\nu \pm \mu = 0\,; \qquad \rho \mp \lambda = 0\,.
\end{equation*}
Applying these equalities to \eqref{sigma_nu_lambda_rho_mu} we find that they are equivalent to:
\begin{equation*}\label{sigma_nu_rho_DegP}
\sigma (\vX_1,\vX_1)=-\nu \n_1 - \rho \n_2      \,; \qquad
\sigma (\vX_1,\vX_2)= \mp\rho \n_1 \pm \nu \n_2 \,.
\end{equation*}
The last equalities are equivalent to the conditions:
\begin{equation}\label{sigma_X1X2_DegP}
\sigma (\vX_1,\vX_1)\bot \: \sigma(\vX_1,\vX_2)\,; \qquad \sigma^2(\vX_1,\vX_1)=\sigma^2(\vX_1,\vX_2)\,,
\end{equation} 

 On the other hand, it follows that equalities $\sigma (\x_u,\x_u)=E\sigma(\vX_1,\vX_1)$ and 
$\sigma (\x_u,\x_v)=E\sigma(\vX_1,\vX_2)$ imply that the conditions \eqref{sigma_X1X2_DegP}
are equivalent to the conditions \eqref{sigma_uv_DegP} for a point in $\M$ to be degenerate.
\end{proof}

 In the paper \cite{S}, points satisfying the equality $K=|\varkappa|$ are called \emph{isotropic}. 
As a consequence of the last theorem, the notions of isotropic points and degenerate points coincide.

\smallskip

 Let now $\M=(\D,\x)$ be a maximal space-like surface and $t=u+\ii v \in \D$ determines canonical 
coordinates in a neighborhood of a given point $p \in \M$. Denote again by $(\vX_1,\vX_2)$ the normalized tangent vectors
$(\x_u,\x_v)$. Since the coordinates are canonical, then $\sigma (\x_u,\x_u)\bot \, \sigma(\x_u,\x_v)$, 
according to \eqref{sigma_uv_can1}. The last equality is equivalent to $\sigma (\vX_1,\vX_1)\bot \, \sigma(\vX_1,\vX_2)$.
This means that the canonical coordinates generate geometric choice of the orthonormal basis in the normal 
space $N_p(\M)$ at the point $p \in \M$. Namely, $\n_1$ and $\n_2$ can be chosen to be collinear with
$\sigma (\vX_1,\vX_1)$ and $\sigma (\vX_1,\vX_2)$. More precisely, let $\n_1$ be the unit normal vector, 
with the \emph{opposite} direction of $\sigma (\vX_1,\vX_1)$, and $\n_2$ be the unit normal vector, such that the quadruple
$(\vX_1,\vX_2,\n_1,\n_2)$ is a \emph{right oriented basis} in $\RR^4_2$. 
Then $\n_2$ is \emph{collinear} with $\sigma (\vX_1,\vX_2)$.

 Having chosen the basis in the above way, formulas \eqref{sigma_nu_lambda_rho_mu} for $\sigma$ in the previous  
section get the form:
\begin{equation}\label{sigma_nu_mu_R42}
\begin{array}{l}
\sigma (\vX_1,\vX_1)=          - \nu \n_1\,, \\
\sigma (\vX_1,\vX_2)=          - \mu \n_2\,, \\
\sigma (\vX_2,\vX_2)=\phantom{-} \nu \n_1\,;
\end{array}
\qquad \nu>0 \,.
\end{equation}
This means that $\lambda=0$ and $\rho=0$. Then the Weingarten operators \eqref{A1A2_nu_lambda_rho_mu}
get the form: 
\begin{equation}\label{A1A2_nu_mu_R42}
A_{\n_1}=
\left(
\begin{array}{rr}
\nu  &  0\\
0    & -\nu
\end{array}
\right); \qquad
A_{\n_2}=
\left(
\begin{array}{rr}
0    & \mu\\
\mu  & 0
\end{array}
\right).
\end{equation}
Since the coordinates are canonical, then \eqref{sigma_uv_can1} gives
$\sigma^2(\x_u,\x_u) = \sigma^2(\x_u,\x_v) - 1$, which implies that
$\sigma^2(\x_u,\x_u) < \sigma^2(\x_u,\x_v) \leq 0$\, and
$\sigma^2(\vX_1,\vX_1) < \sigma^2(\vX_1,\vX_2) \leq 0$\,. Consequently, the functions $\nu$ and $\mu$ 
satisfy the following conditions:
\begin{equation}\label{numu_sigma_R42}
\begin{array}{cl}
   \nu^2   \!\!  & =  -\sigma^2(\vX_1,\vX_1) \,,\\
   \mu^2   \!\!  & =  -\sigma^2(\vX_1,\vX_2) \,;
\end{array}
\qquad
\nu > |\mu |\,.
\end{equation}

\smallskip

Let's connect the basis vectors $(\n_1,\n_2)$ and the functions $\nu$ and $\mu$ with the ellipse of the
normal curvature, defined by \eqref{ellipse_curv-def_R42}. 
It follows from $\sigma (\vX_1,\vX_1)\bot \, \sigma(\vX_1,\vX_2)$ and \eqref{ellipse_curv_min_surf_R42} that
$\n_1$ and $\n_2$ are collinear withe the principal axes of the ellipse and $\n_1$ lies on the major axis, while $\n_2$
lies on the minor axis. Furthermore, it follows from \eqref{numu_sigma_R42} that $\nu$ is equal to the semi-major 
axis and $|\mu|$ is equal to the semi-minor axis of the ellipse.
\smallskip

Let us make clear the relation between the pairs $(K,\varkappa)$ and $(\nu,\mu)$.
In canonical coordinates formulas \eqref{K_kappa_nu_lambda_rho_mu} get the form:
\begin{equation}\label{K_kappa_nu_mu_R42}
K= \nu^2+\mu^2\,; \quad\quad
\varkappa = 2\nu\mu \,.
\end{equation}
Therefore
\begin{gather}\label{K_kappa_nu_mu_2_R42}
\hspace{-2ex}
\begin{array}{rll}
K         &=& \nu^2+\mu^2\,, \\
\varkappa &=& 2\nu\mu\,;
\end{array}
\Leftrightarrow
\begin{array}{rll}
K+\varkappa &=& (\nu + \mu)^2\,, \\
K-\varkappa &=& (\nu - \mu)^2\,;
\end{array}
\Leftrightarrow
\begin{array}{rll}
\sqrt{K+\varkappa} &=& \nu + \mu \,,\\
\sqrt{K-\varkappa} &=& \nu - \mu \,.
\end{array}
\end{gather}
Now we can express $\nu$ and $\mu$ via $K$ and $\varkappa$\,:
\begin{equation}\label{nu_mu_K_kappa_R42}
\begin{array}{rll}
\nu &=& \ds\frac{1}{2}(\sqrt{K+\varkappa}+\sqrt{K-\varkappa})\,, \\[2ex]
\mu &=& \ds\frac{1}{2}(\sqrt{K+\varkappa}-\sqrt{K-\varkappa})\,.
\end{array}
\end{equation}
Since $K$ and $\varkappa$ are invariants, the last formulas imply that $\nu$ and $\mu$ are also invariants.
Formulas \eqref{A1A2_nu_mu_R42} show that $\nu$ and $\mu$ play the same role in the theory of maximal 
space-like surfaces of general type in $\RR^4_2$ as the normal curvature $\nu$ play in the theory 
of maximal space-like surfaces of general type in $\RR^3_1$. That is why the invariants $\nu$ and $\mu$
can be called \emph{normal curvatures}.

\smallskip

 The functions $\nu$ and $\mu$ determine completely the second fundamental form of $\M$, in view of formulas
\eqref{sigma_nu_mu_R42}. Next we show, that the functions $\nu$ and $\mu$ also determine the first fundamental 
form in canonical coordinates. The conditions \eqref{sigma_uv_can1} for canonical coordinates give us:                        
\[
E^2(\sigma^2(\vX_1,\vX_1)-\sigma^2(\vX_1,\vX_2))=-1\,.
\]
The last equality and \eqref{numu_sigma_R42} imply that $E^2(\nu^2-\mu^2)=1$\,. We express from here the
coefficient of the first fundamental form $E$ by means of $\nu$ and $\mu$:
\begin{equation}\label{E_nu_mu_R42}
E=\frac{1}{\sqrt{\nu^2-\mu^2}} \,.
\end{equation}

 Multiplying equalities \eqref{K_kappa_nu_mu_2_R42} we get:
\[
K^2-\varkappa^2 = (\nu^2-\mu^2)^2.
\]
From here and \eqref{E_nu_mu_R42} we obtain the coefficient $E$ in terms of $K$ and $\varkappa$:
\begin{equation}\label{E_K_kappa_R42}
E=\frac{1}{\sqrt[4]{ \vphantom{\mu^2} K^2-\varkappa^2} } \;.
\end{equation}

\smallskip
   
 Finally we establish how the canonical coordinates and also the invariants $K$, $\varkappa$, $\nu$ and $\mu$, 
considered as functions of these coordinates, are transformed under a change of the coordinates and under 
the basic geometric transformations of the maximal space-like surface $\M$ in $\RR^4_2$.

Let us start with a change of the canonical coordinates. Since the four functions $K$, $\varkappa$, $\nu$ and $\mu$ 
are scalar invariants, then the necessary formulas are given by Theorem \ref{Can_Coord-uniq}.
Denote by $t(s)$ any of the changes, described with equalities \eqref{uniq-holo} and \eqref{uniq-antiholo}:
\begin{equation}\label{can_coord_ch_R42}
t=\pm s+c\,; \quad t=\pm \ii s+c\,; \quad t=\pm \bar s + c\,; \quad t=\pm \ii \bar s + c\,.
\end{equation}
Then, for the new functions $\tilde K$, $\tilde\varkappa$, $\tilde\nu$ and $\tilde\mu$ of $s$ we have:
\begin{equation}\label{tilde_K_kappa_nu_mu-can_coord_ch_R42}
\tilde K(s)=K(t(s))\,; \quad \tilde\varkappa(s)=\varkappa(t(s))\,; \quad \tilde\nu(s)=\nu(t(s))\,; \quad \tilde\mu(s)=\mu(t(s))\,.
\end{equation}

 Now, let the complex variable $t$ determine canonical coordinates on the surface $\M$ and the surface $\hat\M$ 
be obtained from $\M$ by a motion in $\RR^4_2$. In the case of a proper motion, it follows from 
\eqref{hat_K_kappa-K_kappa-prop_mov_R42}, that $K$ and $\varkappa$ are invariants. Formulas 
\eqref{nu_mu_K_kappa_R42} imply that $\nu$ and $\mu$ are also invariants: 
\begin{equation}\label{hat_K_kappa_nu_mu-propmov_R42}
\hat K(t) = K(t)\,; \quad \hat\varkappa(t) = \varkappa(t)\,; \quad \hat\nu(t) = \nu(t)\,; \quad \hat\mu(t) = \mu(t)\,.
\end{equation}
In the case of an improper motion \eqref{hat_K_kappa-K_kappa-unprop_mov_R42} shows that $K$ is invariant but 
$\varkappa$ changes the sign. Then it follows from \eqref{nu_mu_K_kappa_R42} that $\nu$ is invariant, but 
$\mu$ changes the sign:
\begin{equation}\label{hat_K_kappa_nu_mu-unpropmov_R42}
\hat K(t) = K(t)\,; \quad \hat\varkappa(t) = -\varkappa(t)\,; \quad \hat\nu(t) = \nu(t)\,; \quad \hat\mu(t) = -\mu(t)\,.
\end{equation}

Let us consider the case of a similarity in $\RR^4_2$.
\begin{theorem}\label{Can_Sim}
Let $\M=(\D,\x)$ be a maximal space-like surface of general type, parametrized by canonical coordinates determined by the 
complex variable $t\in\D$. If the maximal space-like surface $\hat\M=(\D,\hat\x)$ is obtained from $\M$ by a similarity 
with coefficient $k>0$ in $\RR^4_2$, then $\hat\M$ is also of general type and the variable $s$ satisfying the equality
$t=\frac{1}{\sqrt{k}}s$, determines canonical coordinates on $\hat\M$.
\end{theorem}
\begin{proof}
It is enough to consider a homothety in $\RR^4_2$ and let $\hat\x=k\x$. It follows from here that 
$\hat\Phi=k\Phi$, $\hat\Phi'_t=k\Phi'_t$ and $\left.\hat\Phi_t^{\prime\bot}\right.^2 = k^2 {\Phi_t^{\prime \bot}}^2$,
respectively. Since $t$ determines canonical coordinates on $\M$, then the last equality reduces to
$\left.\hat\Phi_t^{\prime\bot}\right.^2 = - k^2$. Hence $\hat\M$ is also of general type. 
On the other hand, $t=\frac{1}{\sqrt{k}}s$ gives that $t'^{\,4}=\frac{1}{k^2}$.
Applying \eqref{tildPhiPr2} to $\hat\M$, we get: 
\[
\left.\tilde{\hat\Phi}_s^{\prime\bot}\right.^2 = \left.\hat\Phi_t^{\prime\bot}\right.^2 t'^{\,4} =
- k^2 \ds\frac{1}{k^2} = - 1\,.
\]
This means that $s$ determines canonical coordinates on $\hat\M$. 
\end{proof}

\begin{remark}
The formula $\left.\hat\Phi_t^{\prime\bot}\right.^2 = k^2 {\Phi_t^{\prime \bot}}^2$,
in the proof of the last theorem, means that the notion of a degenerate point is invariant under a similarity.
\end{remark}

 If $\hat\M$ is obtained from $\M$ by a homothety: $\hat\x=k\x$, then \eqref{hat_K_kappa-K_kappa-homotet_R42} 
imply that $\hat K=\frac{1}{k^2}K$ and $\hat\varkappa = \frac{1}{k^2}\varkappa$. From here and \eqref{nu_mu_K_kappa_R42} 
we get $\hat\nu = \frac{1}{k}\nu$ and $\hat\mu = \frac{1}{k}\mu$. 
The canonical coordinates $s$ on $\hat\M$ and the canonical coordinates $t$ on $\M$ are related by the equality 
$t=\frac{1}{\sqrt{k}}s$, according to Theorem \ref{Can_Sim}. Thus we have
\begin{gather}\label{hat_K_kappa_nu_mu-homotet_R42}
\textstyle{\hat K(s) = \frac{1}{k^2}\; K\!\left(\frac{1}{\sqrt{k}}s\right); \ \  
\hat\varkappa(s) = \frac{1}{k^2}\;\varkappa\!\left(\frac{1}{\sqrt{k}}s\right); \ \ 
\hat\nu(s) = \frac{1}{k}\;\nu\!\left(\frac{1}{\sqrt{k}}s\right); \ \ 
\hat\mu(s) = \frac{1}{k}\;\mu\!\left(\frac{1}{\sqrt{k}}s\right).}
\raisetag{-0.3em}
\end{gather}

\smallskip
 Now let us see how to obtain the canonical coordinates on associated maximal space-like surfaces introduced by
Definition \ref{1-param_family_assoc_surf}\,.
\begin{theorem}\label{Can_1-param_family}
Let $\M=(\D,\x)$ be a maximal space-like surface of general type and
$t\in\D$ be a complex variable, which determines canonical coordinates on $\M$.
Let $\M_\theta=(\D,\x_\theta)$ be the family of associated maximal space-like surfaces with $\M$.
Then for any $\theta$, $\M_\theta$ is also of general type and the variable $s$ introduced by
$t=\e^{\ii\frac{\theta}{2}} s$, determines canonical coordinates on $\M_\theta$.
\end{theorem}
\begin{proof}
We use the formula \eqref{Phi_1-param_family}, which gives that
$\Phi_\theta = \e^{-\ii\theta}\Phi$. From here we find $\Phi'_{\theta|t} = \e^{-\ii\theta}\Phi'_t$ and 
$\left.\Phi_{\theta|t}^{\prime\bot}\right.^2 = \e^{-2\ii\theta} {\Phi_t^{\prime \bot}}^2$.
Since $t$ determines canonical coordinates on $\M$, then the last equality gets the form
$\left.\Phi_{\theta|t}^{\prime\bot}\right.^2 = -\e^{-2\ii\theta}$, which implies that $\M_\theta$ 
is also of general type. On the other hand, the formula $t=\e^{\ii\frac{\theta}{2}} s$ implies that 
$t'^{\,4}=\e^{2\ii\theta}$. Applying \eqref{tildPhiPr2} to $\M_\theta$, we obtain: 
\[
\left.\tilde\Phi_{\theta|s}^{\prime\bot}\right.^2 = \left.\Phi_{\theta|t}^{\prime\bot}\right.^2 t'^{\,4} =
- \e^{-2\ii\theta} \e^{2\ii\theta} = - 1\,.
\]
This means that $s$ determines canonical coordinates on $\M_\theta$. 
\end{proof}

\begin{remark}
The formula $\left.\Phi_{\theta|t}^{\prime\bot}\right.^2 = \e^{-2\ii\theta} {\Phi_t^{\prime \bot}}^2$,
in the proof of the last theorem shows that the notion of a degenerate point is invariant for the family 
of associated surfaces in the following sense: If $p$ is a degenerate point in $\M$, then for any 
$\theta$ the point $\mathcal{F}_\theta(p)$, corresponding to на $p$ under the standard isometry 
$\mathcal{F}_\theta$ between $\M$ and $\M_\theta$, described in Proposition \ref{Isom_M_phi-M}\,,
 is also a degenerate point in $\M_\theta$.
\end{remark}

 The scalar invariants of $\M_\theta$ in view of \eqref{K_kappa-1-param_family_R42}, satisfy 
$K_\theta(t)=K(t)$ и $\varkappa_\theta(t) = \varkappa(t)$. taking into account \eqref{nu_mu_K_kappa_R42},
we have $\nu_\theta(t) = \nu(t)$ и $\mu_\theta(t) = \mu(t)$. The canonical coordinate $s$ of $\M_\theta$ 
and the canonical coordinate $t$ of $\M$ are related by the formula $t=\e^{\ii\frac{\theta}{2}} s$, 
according to Theorem \ref{Can_1-param_family}. Combining the last formulas, we obtain:
\begin{equation}\label{K_kappa_nu_mu_1-param_family_R42}
K_\theta(s) = K(\e^{\ii\frac{\theta}{2}} s)\;; \quad \varkappa_\theta(s) = \varkappa(\e^{\ii\frac{\theta}{2}} s)\;; \quad
\nu_\theta(s) = \nu(\e^{\ii\frac{\theta}{2}} s)\;; \quad \mu_\theta(s) = \mu(\e^{\ii\frac{\theta}{2}} s)\;.
\end{equation}

Since the conjugate maximal space-like surface $\bar\M$ of $\M$ belongs to the family of the associated surfaces with 
$\M$ by the condition $\theta=\frac{\pi}{2}$, then the corresponding formulas for the conjugate surface $\bar\M$
can be obtained from the above formulas replacing $\theta = \frac{\pi}{2}$.


\section{Frenet type formulas and a system of natural equations for a maximal space-like surface in $\RR^4_2$.}\label{sect_nat_eq}

 Let $\M=(\D,\x)$ be a maximal space-like surface of general type in $\RR^4_2$ and $t=u+\ii v$ determine canonical 
coordinates on $\M$. Consider the orthonormal frame field $(\vX_1,\vX_2,\n_1,\n_2)$, introduced in the previous 
section. As a consequence of Theorem \ref{Can_Coord-uniq} this basis is determined uniquely up to renumbering 
and directions of the vectors in the basis.
 The covariant derivatives of these vector fields satisfy the following Frenet type formulas:
\begin{equation}\label{Frene_X1_X2_n1_n2_R42}
\begin{array}{llrrr}
       \nabla_{\vX_1} \vX_1 & =                     & \gamma_1 \,\vX_2 & -\ \  \nu \,\n_1\,; &                     \\[0.5ex]
       \nabla_{\vX_1} \vX_2 & =\; -\gamma_1 \,\vX_1 &                  &                     & - \;\;\mu \,\n_2\,; \\[0.5ex]
       \nabla_{\vX_2} \vX_1 & =                     & \gamma_2 \,\vX_2 &                     & - \;\;\mu \,\n_2\,; \\[0.5ex]
			 \nabla_{\vX_2} \vX_2 & =\; -\gamma_2 \,\vX_1 &                  & +\ \  \nu \,\n_1\,; &                     \\[2.0ex] 
				
       \nabla_{\vX_1} \n_1 & =\; -\nu  \,\vX_1   &              &                        & +\ \ \beta_1 \,\n_2\,; \\[0.5ex]
       \nabla_{\vX_1} \n_2 & =                   & -\mu \,\vX_2 & -\ \ \beta_1 \,\n_1\,; &                        \\[0.5ex]
       \nabla_{\vX_2} \n_1 & =                   &  \nu \,\vX_2 &                        & +\ \ \beta_2 \,\n_2\,; \\[0.5ex]
			 \nabla_{\vX_2} \n_2 & =\; -\mu  \,\vX_1   &              & -\ \ \beta_2 \,\n_1\,. &                        \\[0.5ex]
       \end{array}
\end{equation}

 The functions $\nu$ and $\mu$ are the invariants, introduced in the previous section and $\gamma_1$, $\gamma_2$
are the geodesic curvatures of the canonical coordinate lines on $\M$. The functions $\gamma_1$ and $\gamma_2$ 
can be expressed by the coefficient $E$ of the first fundamental form:
\[
\begin{array}{lll}
\gamma_1 &=&(\nabla_{\vX_1} \vX_1) \cdot \vX_2 = \ds\frac{(\nabla_{\x_u} \x_u) \cdot \x_v}{(\sqrt{E})^3}
          = \ds\frac{-\x_u \cdot (\nabla_{\x_u}\x_v)}{E\sqrt{E}}\\[2.5ex]
 				 &=&\ds\frac{-\x_u \cdot (\nabla_{\x_v}\x_u)}{E\sqrt{E}}
				  = \ds\frac{-\nabla_{\x_v}(\x_u \cdot \x_u)}{2E\sqrt{E}}
					= \ds\frac{-E_v}{2E\sqrt{E}}\;.
\end{array}
\]
Note that the last computations are valid in arbitrary isothermal coordinates. In a similar way we get 
an analogous formula for $\gamma_2$. Thus, the geodesic curvatures of the coordinate lines of a surface $\M$, parametrized by isothermal coordinates are given by:
\begin{equation}\label{gamma_12-E_Can_R42}
\gamma_1 = \frac{-E_v}{2E\sqrt{E}}\,; \qquad \gamma_2 = \frac{ E_u}{2E\sqrt{E}}\;.
\end{equation}

\smallskip
 The integrability conditions of the system \eqref{Frene_X1_X2_n1_n2_R42} form the system of natural PDE's 
of the surface $\M$. In the present work we use an alternative approach to obtain these natural equations
using the complex vector fields $\Phi$ and $\bar\Phi$ instead the vector fields $\vX_1$ and $\vX_2$.

The equality $\Phi^2=0$ means that $\Phi$ and $\bar\Phi$ are orthogonal with respect to the Hermitian 
scalar product in $\CC^4$. Hence the quadruple $(\Phi,\bar\Phi,\n_1,\n_2)$ forms an orthogonal (not necessarily normalized)
moving frame field on $\M$. Next we find a complex variant of Frenet type formulas \eqref{Frene_X1_X2_n1_n2_R42}.

Since $\Phi$ is a holomorphic function, the first formula is $\frac{\partial\Phi}{\partial \bar t}=0$. 
In view of $\Phi^2=0$ it follows that $\frac{\partial\Phi}{\partial t}\Phi=0$, which means that 
$\frac{\partial\Phi}{\partial t}$ is orthogonal to $\bar\Phi$ with respect to the Hermitian scalar product
in $\CC^4$. Therefore $\frac{\partial\Phi}{\partial t}$ is a linear combination of $\Phi$, $\n_1$ and $\n_2$ 
and has the form:
\begin{equation*}
\frac{\partial\Phi}{\partial t}=\ds\frac{\Phi' \cdot \bar \Phi}{\|\Phi\|^2}\Phi-(\Phi'\cdot\n_1)\n_1-(\Phi'\cdot\n_2)\n_2\,.
\end{equation*}
The coefficient before $\Phi$ satisfies equality \eqref{d_ln_E_dt} and can be expressed by $E$.
For the coefficient before $\n_1$ we combine \eqref{PhiPr} and \eqref{sigma_nu_mu_R42} and get:
\begin{equation*}
\Phi'\cdot\n_1 = (\sigma(\x_u,\x_u)-\ii\sigma(\x_u,\x_v))\cdot\n_1=E\sigma(\vX_1,\vX_1)\cdot\n_1=E\nu \,.
\end{equation*}
Analogously for the coefficient before $\n_2$ we have:
\begin{equation*}
\Phi'\cdot\n_2 = (\sigma(\x_u,\x_u)-\ii\sigma(\x_u,\x_v))\cdot\n_2=-\ii E\sigma(\vX_1,\vX_2)\cdot\n_2=-\ii E\mu \,.
\end{equation*}
Thus we find the following formula for $\frac{\partial\Phi}{\partial t}$:
\begin{equation*}
\frac{\partial\Phi}{\partial t}=\ds\frac{\partial\ln E}{\partial t}\Phi - E\nu\n_1 + \ii E\mu\n_2 \,.
\end{equation*}

 The condition $\n_1^2=-1$ implies that $\frac{\partial\n_1}{\partial t}\n_1=0$ and hence $\frac{\partial\n_1}{\partial t}$
is represented in the form:
\begin{equation*}
\frac{\partial\n_1}{\partial t} =  
\frac{\frac{\partial\n_1}{\partial t}\bar\Phi}{\|\Phi\|^2}\Phi + 
\frac{\frac{\partial\n_1}{\partial t}\Phi}{\|\Phi\|^2}\bar\Phi + \beta\n_2 \,. 
\end{equation*}
The coefficients before $\Phi$ and $\bar\Phi$ are: 
\[
\frac{\partial\n_1}{\partial t}\bar\Phi=-\n_1\frac{\partial\bar\Phi}{\partial t}=0, \quad
\frac{\partial\n_1}{\partial t}\Phi=-\n_1\frac{\partial\Phi}{\partial t}=-E\nu \text{\ \ \ and\ \ \ } \|\Phi\|^2=2E\,. 
\]
Hence $\frac{\partial\n_1}{\partial t}$ is given by:
\begin{equation*}
\frac{\partial\n_1}{\partial t} =  -\frac{\nu}{2}\bar\Phi + \beta\n_2 \,. 
\end{equation*}
In analogous way we find $\frac{\partial\n_2}{\partial t}$:
\begin{equation*}
\frac{\partial\n_2}{\partial t} =  \ii\frac{\mu}{2}\bar\Phi - \beta\n_1 \,. 
\end{equation*}

 Up to now, we found formulas for four derivatives:
$\frac{\partial\Phi}{\partial \bar t}$, $\frac{\partial\Phi}{\partial t}$, $\frac{\partial\n_1}{\partial t}$ and 
$\frac{\partial\n_2}{\partial t}$. Applying complex conjugation to these formulas we obtain 
$\frac{\partial\bar\Phi}{\partial t}$, $\frac{\partial\bar\Phi}{\partial\bar t}$, $\frac{\partial\n_1}{\partial\bar t}$ and 
$\frac{\partial\n_2}{\partial\bar t}$, but they do not give any new information.
Hence, the Frenet type formulas for the complex frame field $(\Phi,\bar\Phi,\n_1,\n_2)$ are the following:

\begin{equation}\label{Frene_Phi_bar_Phi_n1_n2_R42}
\begin{array}{llcrrr}
 \ds\frac{\partial\Phi}{\partial \bar t} &=&  0\,;                                     & &            & \\[2ex]
 \ds\frac{\partial\Phi}{\partial t}      &=&  \ds\frac{\partial\ln E}{\partial t}\Phi  & &-\ E\nu\n_1\phantom{.} 
&+\ \ \ \ \,\ii E\mu\n_2\,;\\[2ex]
 \ds\frac{\partial\n_1}{\partial t}      &=&           &  -\ds\frac{\nu}{2}\bar\Phi  & &+\qquad\ \:\beta\n_2\,;\\[2ex]
 \ds\frac{\partial\n_2}{\partial t}      &=&           &\ii\ds\frac{\mu}{2}\bar\Phi  &-\ \ \ \,\beta\n_1. &\\
\end{array}
\end{equation}

\bigskip
 The integrability conditions of the above system of complex PDE's are the following:
\[
\frac{\partial^2\Phi}{\partial \bar t \partial t}=\frac{\partial^2\Phi}{\partial t \partial\bar t}\;; \qquad
\frac{\partial^2\n_1}{\partial \bar t \partial t}=\frac{\partial^2\n_1}{\partial t \partial\bar t}\;; \qquad
\frac{\partial^2\n_2}{\partial \bar t \partial t}=\frac{\partial^2\n_2}{\partial t \partial\bar t}\;.
\]

 The first equation yields $\frac{\partial^2\Phi}{\partial t \partial\bar t}=0$ and the equality of the 
mixed partial derivatives gives:
\[
\begin{array}{lll}
0=\ds\frac{\partial^2\Phi}{\partial \bar t \partial t} &=&
      \ds\frac{\partial^2\ln E}{\partial\bar t \partial t}\Phi + 
      \ds\frac{\partial\ln E}{\partial t}\ds\frac{\partial\Phi}{\partial\bar t} - 
      \ds\frac{\partial(E\nu\n_1)}{\partial\bar t} + \ii\ds\frac{\partial(E\mu\n_2)}{\partial\bar t}\\[2.0ex]
 				 &=&\ds\frac{\Delta\ln E}{4}\Phi-
				\ds\frac{\partial(E\nu)}{\partial\bar t}\n_1-E\nu\ds\frac{\partial\n_1}{\partial\bar t} + 
				\ii\ds\frac{\partial(E\mu)}{\partial\bar t}\n_2+\ii E\mu\ds\frac{\partial\n_2}{\partial\bar t}\\[2.0ex]
				 &=&\left(\ds\frac{\Delta\ln E}{4}+\ds\frac{E\nu^2}{2}+\ds\frac{E\mu^2}{2}\right)\Phi\\[2.0ex]
				 &+&\left(  -\ds\frac{\partial(E\nu)}{\partial\bar t}-\ii E\mu\bar\beta\right)\n_1
				  + \left(\ii\ds\frac{\partial(E\mu)}{\partial\bar t}-    E\nu\bar\beta\right)\n_2\,.
\end{array}
\]

 The coefficient before $\Phi$ is zero, i.e.:
\begin{equation}\label{Nat_Eq_Gauss_E_nu_mu_R42}
\frac{\Delta\ln E}{E}+2\nu^2+2\mu^2=0\,.
\end{equation} 
Taking into account \eqref{K_kappa_nu_mu_R42}, it follows easily that  the equation 
\eqref{Nat_Eq_Gauss_E_nu_mu_R42} is in fact the classical Gauss equation.

 Equating to zero the coefficients before $\n_1$ and $\n_2$ we get:
\begin{equation}\label{Nat_Eq_Codazzi_beta_E_nu_mu_R42}
E\nu\bar\beta = \ii \ds\frac{\partial(E\mu)}{\partial\bar t}\;; \qquad
E\mu\bar\beta = \ii \ds\frac{\partial(E\nu)}{\partial\bar t}\;.
\end{equation}
These equations are in fact the classical Codazzi equations. Next we see that in canonical coordinates 
the two equalities in \eqref{Nat_Eq_Codazzi_beta_E_nu_mu_R42} are not independent and the 
first of them implies the second one. Note that \eqref{E_nu_mu_R42} implies that $(E\nu)^2=(E\mu)^2 + 1$.
Differentiating the last equality, we find:
\[
\nu \ds\frac{\partial(E\nu)}{\partial\bar t} = \mu \ds\frac{\partial(E\mu)}{\partial\bar t}\;.
\]
Assuming that the first equality in \eqref{Nat_Eq_Codazzi_beta_E_nu_mu_R42} is satisfied, we replace 
$\ds\frac{\partial(E\mu)}{\partial\bar t}$ with $-\ii E\nu\bar\beta$ and get:
\[
\nu \ds\frac{\partial(E\nu)}{\partial\bar t} = -\ii \mu E\nu\bar\beta\,.
\]
Since by definition we have $\nu \neq 0$, then the last equality is equivalent to the second equality in 
\eqref{Nat_Eq_Codazzi_beta_E_nu_mu_R42}.

Taking into account the first equation in \eqref{Nat_Eq_Codazzi_beta_E_nu_mu_R42} we shall express explicitly $\beta$
by means of $\nu$ and $\mu$. 
First we add the two equalities in \eqref{Nat_Eq_Codazzi_beta_E_nu_mu_R42} and get:
\begin{equation*}
E(\nu+\mu)\bar\beta = \ii \ds\frac{\partial(E(\nu+\mu))}{\partial\bar t}\;,
\end{equation*}
from where it follows that
\begin{equation*}
\bar\beta = \ii \ds\frac{\partial\ln(E(\nu+\mu))}{\partial\bar t}\;.
\end{equation*}
In the last equality we replace $E$ from \eqref{E_nu_mu_R42} and find:
\begin{equation*}
\bar\beta = \ii \ds\frac{\partial}{\partial\bar t}\ln\sqrt{\frac{\nu+\mu}{\nu-\mu}}
          = \frac{\ii}{2} \ds\frac{\partial}{\partial\bar t}\ln\frac{\nu+\mu}{\nu-\mu}\;.
\end{equation*}
After a complex conjugation, we have:
\begin{equation}\label{beta_nu_mu_R42}
\beta = -\frac{\ii}{2} \ds\frac{\partial}{\partial t}\ln\frac{\nu+\mu}{\nu-\mu}\;.
\end{equation}

 Now, let us return to the integrability conditions of \eqref{Frene_Phi_bar_Phi_n1_n2_R42} and consider 
the equality of the mixed partial derivatives of $\n_1$ and $\n_2$. Taking the tangential components 
of these equalities we get equations equivalent to \eqref{Nat_Eq_Codazzi_beta_E_nu_mu_R42}. That is why 
we take only the normal components of these mixed partial derivatives.
\[
\begin{array}{lll}
 \left(\ds\frac{\partial^2\n_1}{\partial \bar t \partial t}\right)^\bot &=&
    -\ds\frac{\nu}{2}\left(\ds\frac{\partial \bar\Phi}{\partial \bar t}\right)^\bot +
     \ds\frac{\partial \beta}{\partial \bar t}\n_2 + \beta \left(\ds\frac{\partial \n_2}{\partial \bar t}\right)^\bot\\[2.0ex]
&=&  \phantom{-}\ds\frac{\nu}{2}\left(E\nu\n_1 + \ii E\mu\n_2 \right) + \ds\frac{\partial \beta}{\partial \bar t}\n_2
    + \beta (-\bar\beta\n_1)\\[2.0ex]
&=&  \phantom{-}\left(  \ds\frac{E\nu^2}{2} - |\beta|^2\right)\n_1 + 
      \left(  \ii\ds\frac{E\nu\mu}{2} + \ds\frac{\partial \beta}{\partial \bar t}\right)\n_2 \,.
\end{array}
\]
In analogous way we find for the other mixed derivative:
\[
\left(\ds\frac{\partial^2\n_1}{\partial t \partial \bar t}\right)^\bot =
      \left(  \ds\frac{E\nu^2}{2} - |\beta|^2\right)\n_1 +
		  \left( -\ii\ds\frac{E\nu\mu}{2} + \ds\frac{\partial \bar \beta}{\partial t}\right)\n_2 \;.
\]
The last formulas show that the coefficients before $\n_1$ coincide. 
Equating the coefficients before $\n_2$, we get an equality which is the Ricci fundamental equation:
\[
 \ii\ds\frac{E\nu\mu}{2} + \ds\frac{\partial \beta}{\partial \bar t} = 
-\ii\ds\frac{E\nu\mu}{2} + \ds\frac{\partial \bar \beta}{\partial t}\;.
\]
The last equation is equivalent to
\[
\ds\frac{\partial \beta}{\partial \bar t} - \ds\frac{\partial \bar \beta}{\partial t} = -\ii E \nu \mu 
\quad \Leftrightarrow \quad
2\Im \ds\frac{\partial \beta}{\partial \bar t} \ii = -\ii E \nu \mu \,.
\]
Finally we obtain:
\begin{equation}\label{Nat_Eq_Ricci_E_nu_mu_R42}
\Im \ds\frac{\partial \beta}{\partial \bar t} = -\ds\frac{E\nu\mu}{2} \,.
\end{equation}

 Equating the mixed partial derivatives of the second normal vector field $\n_2$ we do not obtain new equations.
Summarizing the above calculations, we have the following statement. 
\begin{prop}\label{Frene_integr_cond_R42}
The integrability conditions of the system \eqref{Frene_Phi_bar_Phi_n1_n2_R42} are given by the equalities:
\eqref{Nat_Eq_Gauss_E_nu_mu_R42}, \eqref{Nat_Eq_Codazzi_beta_E_nu_mu_R42} and \eqref{Nat_Eq_Ricci_E_nu_mu_R42}. 
\end{prop}

 In the above integrability conditions we can eliminate two of the functions: $E$ and $\beta$.
Thus we shall obtain a system of two partial differential equations for the functions $\nu$ and $\mu$. 
For that purpose let us replace the expression for $\beta$ from \eqref{beta_nu_mu_R42} into the
equality \eqref{Nat_Eq_Ricci_E_nu_mu_R42}.
\[
-\Im \frac{\ii}{2} \ds\frac{\partial^2}{\partial \bar t \partial t}\ln\frac{\nu+\mu}{\nu-\mu} = -\ds\frac{E\nu\mu}{2}\;.
\]
Using that
$\frac{\partial^2}{\partial \bar t \partial t} = \frac{1}{4}\Delta$, we get:
\[
\frac{1}{4} \Delta\ln\frac{\nu+\mu}{\nu-\mu} = E\nu\mu\,,
\]
which is equivalent to
\[
\frac{1}{E} \Delta\ln\frac{\nu+\mu}{\nu-\mu} - 4\nu\mu = 0\,.
\]
Replacing $E$ from \eqref{E_nu_mu_R42} into the last equation and in \eqref{Nat_Eq_Gauss_E_nu_mu_R42}
we obtain that $\nu$ and $\mu$ satisfy the following system of partial differential equations:

\begin{equation}\label{Nat_Eq_nu_mu_R42}
\begin{array}{l}
\sqrt{\nu^2-\mu^2}\; \Delta\ln \sqrt{\nu^2-\mu^2} - 2\nu^2 - 2\mu^2 = 0\,; \\[1.5ex]
\sqrt{\nu^2-\mu^2}\; \Delta\ln\ds\frac{\nu+\mu}{\nu-\mu} - 4\nu\mu = 0\,;
\end{array} \qquad\quad \nu^2-\mu^2>0\,;\quad \nu>0\,.
\end{equation}
 
 We call this system the \emph{system of natural PDE's} of the maximal space-like surfaces in $\RR^4_2$.
Since the Frenet type formulas \eqref{Frene_Phi_bar_Phi_n1_n2_R42} are the complex variant of 
the Frenet type formulas \eqref{Frene_X1_X2_n1_n2_R42}, then \eqref{Nat_Eq_nu_mu_R42} are also the 
integrability conditions of \eqref{Frene_X1_X2_n1_n2_R42}. 

 Let $\M=(\D,\x)$ and $\hat\M=(\D,\hat\x)$ be two maximal space-like surfaces of general type related by 
a proper motion of the type:
\begin{equation}\label{hat_M-M-prop_mov_R42}
\hat\x = A\x + \vb\,; \qquad A \in \mathbf{SO}(2,2,\RR), \ \ \vb \in \RR^4_2 \,.
\end{equation}
If $t=u+\ii v$ gives canonical coordinates on $\M$, then according to Theorem \ref{Can_Move}
these coordinates are canonical for the surface $\hat\M$ as well.
In view of \eqref{hat_K_kappa_nu_mu-propmov_R42} the pair $(\hat\nu,\hat\mu)$
coincides with the pair $(\nu,\mu)$, considered as functions of $t$. 

 The results obtained above for $\nu$ and $\mu$ can be summarized in the following statement:
\begin{theorem}\label{Thm-Nat_Eq_nu_mu_R42}
Let $\M$ be a maximal space-like surface of general type in $\RR^4_2$ parametrized by canonical coordinates.
Then the invariants $\nu$ and $\mu$ of $\M$, considered as functions of these coordinates satisfy
the system of natural equations \eqref{Nat_Eq_nu_mu_R42} of the maximal space-like surfaces in $\RR^4_2$.

If $\hat\M$ is obtained from $\M$ by a proper motion of the type \eqref{hat_M-M-prop_mov_R42},
then it generates the same solution to \eqref{Nat_Eq_nu_mu_R42}. 
\end{theorem}

 In the paper \cite{S}, it is obtained a system of PDE's, analogous to \eqref{Nat_Eq_nu_mu_R42},
but for the invariants $(K,\varkappa)$. Since the pairs $(\nu,\mu)$ and $(K,\varkappa)$ can be 
obtained one from the other, we shall see that the system of PDE's for $(K,\varkappa)$ in 
\cite{S} is equivalent to \eqref{Nat_Eq_nu_mu_R42}.

By means of multiplication and division in \eqref{K_kappa_nu_mu_2_R42} we get:
\begin{equation*}
K^2-\varkappa^2 = (\nu^2-\mu^2)^2\,; \qquad
\ds\frac{K+\varkappa}{K-\varkappa} = \left( \ds\frac{\nu+\mu}{\nu-\mu} \right)^2 .
\end{equation*}
Applying the last equalities and \eqref{K_kappa_nu_mu_R42}, to the system \eqref{Nat_Eq_nu_mu_R42},
we obtain that the curvatures $K$ and $\varkappa$ are solutions to the following system, which we also call
the \emph{system of natural PDE's} of the maximal space-like surfaces in $\RR^4_2$: 

\begin{equation}\label{Nat_Eq_K_kappa_R42}
\begin{array}{l}
\sqrt[4]{ \vphantom{\mu^2} K^2-\varkappa^2}\; \Delta\ln \sqrt[4]{ \vphantom{\mu^2} K^2-\varkappa^2} = 2K\,; \\[1.5ex]
\sqrt[4]{ \vphantom{\mu^2} K^2-\varkappa^2}\; \Delta\ln\sqrt{\ds\frac{\vphantom{\mu^2}K+\varkappa}{K-\varkappa}}=2\varkappa\,;
\end{array}  \qquad\quad K^2-\varkappa^2>0\,;\quad K>0\,.
\end{equation}

By using the identity $\arctanh x = \ds\frac{1}{2}\ln\ds\frac{1+x}{1-x}$, we can write this system in the form:

\begin{equation}\label{Nat_Eq_K_kappa_2_R42}
\begin{array}{l}
\sqrt[4]{ \vphantom{\mu^2} K^2-\varkappa^2}\; \Delta\ln \sqrt[4]{ \vphantom{\mu^2} K^2-\varkappa^2} = 2K\,; \\[1.5ex]
\sqrt[4]{ \vphantom{\mu^2} K^2-\varkappa^2}\; \Delta\arctanh\ds\frac{\varkappa}{K}=2\varkappa\,;
\end{array}  \qquad\quad K^2-\varkappa^2>0\,;\quad K>0\,.
\end{equation}

 The system of natural equations \eqref{Nat_Eq_K_kappa_R42} for $K$ and $\varkappa$ is equivalent to \eqref{Nat_Eq_nu_mu_R42}
and therefore we have the following analogue of Theorem \ref{Thm-Nat_Eq_nu_mu_R42}\,:
\begin{theorem}\label{Thm-Nat_Eq_K_kappa_R42}
Let $\M$ be a maximal space-like surface of general type in $\RR^4_2$ parametrized by canonical coordinates. 
Then the Gauss curvature $K$ and the curvature of the normal connection $\varkappa$ of $\M$, satisfy the system 
of natural equations 
\eqref{Nat_Eq_K_kappa_R42} of the maximal space-like surfaces in $\RR^4_2$.

If $\hat\M$ is obtained from $\M$ through a proper motion of the type \eqref{hat_M-M-prop_mov_R42},
then $\hat\M$ gives the same solution to the system \eqref{Nat_Eq_K_kappa_R42}. 
\end{theorem}

 The curvatures $K$ and $\varkappa$ in \eqref{Nat_Eq_K_kappa_R42} are scalar invariants,
but the system as a whole is not, since the Laplace operator is not invariant.
In order to obtain a completely invariant form of the system of natural equations, we can use the Laplace-Beltrami 
operator $\Delta_b$ for the surface $\M$, given by the formula:
\begin{equation}\label{def_Lpl-Beltr}
\Delta_b = *\,d*d \,,
\end{equation}
where "$*$" is the Hodge operator and $d$ is the exterior differential. 
This operator is invariant by definition with respect to the local coordinates and 
in isothermal coordinates it is related to the ordinary Laplace operator in $\RR^2$
by the equality:
\begin{equation}\label{Lpl-Beltr_Lpl}
\Delta_b = \frac{1}{E} \Delta\,.
\end{equation}
In canonical coordinates the equality \eqref{E_K_kappa_R42} is fulfilled  and hence this operator 
has the following representation:
\begin{equation}\label{Lpl-Beltr_Lpl_R42}
\Delta_b = \frac{1}{E}\, \Delta = \sqrt[4]{ \vphantom{\mu^2} K^2-\varkappa^2}\; \Delta\,.
\end{equation}
The last formula, applied to \eqref{Nat_Eq_K_kappa_R42}, gives:
\begin{equation}\label{Nat_Eq_Beltr_K_kappa_R42}
\begin{array}{l}
\Delta_b \ln \sqrt[4]{ \vphantom{\mu^2} K^2-\varkappa^2} = 2K\,; \\[1.5ex]
\Delta_b \ln\sqrt{\ds\frac{\vphantom{\mu^2}K+\varkappa}{K-\varkappa}}=2\varkappa\,;
\end{array} \qquad\ K^2-\varkappa^2>0\,;\quad K>0\,.
\end{equation}
The system of natural equations in this form is completely invariant and it is the same with respect 
to arbitrary (possibly non-isothermal) coordinates.  

 Multiplying the first equality in \eqref{Nat_Eq_Beltr_K_kappa_R42} with $2$ and then applying addition 
and subtraction with the second equation, we get the system of natural equations in the form given in \cite{S}: 
\begin{equation}\label{Nat_Eq_Sakaki_K_kappa_R42}
\begin{array}{l}
\Delta_b \ln (K + \varkappa ) = 2(2K+\varkappa) \,; \\[1.5ex]
\Delta_b \ln (K - \varkappa ) = 2(2K-\varkappa) \,;
\end{array} \qquad\ K>|\varkappa| \,.
\end{equation}

\medskip

 Let us return to formulas \eqref{Frene_Phi_bar_Phi_n1_n2_R42}. Next we shall see that the equations  
\eqref{Nat_Eq_nu_mu_R42} are not only necessary but also sufficient for the (local) existence of a solution 
to the system \eqref{Frene_Phi_bar_Phi_n1_n2_R42}. This is given by the following Bonnet type theorem
for the system \eqref{Frene_Phi_bar_Phi_n1_n2_R42}:
\begin{theorem}\label{Bone_Phi_bar_Phi_n1_n2_nu_mu_R42}
Let $\nu>0$ and $\mu$ be real functions, defined in a domain $\D\subset\RR^2$ and let the pair $(\nu,\mu)$ 
be a solution to the system of natural equations \eqref{Nat_Eq_nu_mu_R42} of the maximal space-like surfaces 
in $\RR^4_2$. For any point $p_0\in\D$, there exists a neighborhood $\D_0\subset\D$ of $p_0$ and
a vector function $\x: \D_0 \to \RR^4_2$, such that $(\D_0,\x)$ is a maximal space-like surface of general type
parametrized by canonical coordinates and the given functions are the scalar invariants $\nu$ and $\mu$ of 
this surface, defined by \eqref{sigma_nu_mu_R42}.
If $(\hat\D_0,\hat\x)$ is another surface with the same properties, then there exists a sub-domain
$\tilde\D_0$ of $\D_0$ containing the point $p_0$, such that $(\tilde\D_0,\hat\x)$ and $(\tilde\D_0,\x)$ 
are related by a proper motion in $\RR^4_2$ of the type \eqref{hat_M-M-prop_mov_R42}.
\end{theorem}
\begin{proof}
Let $(u,v) \in \D$ and $t=u+\ii v$. Let $p_0\in\D$ be a fixed point with coordinates $t_0=u_0+\ii v_0$. 
Define new functions $E$ and $\beta$ in the domain $\D$ by means of the equalities:
\begin{equation}\label{E_beta_nu_mu_R42}
E=\frac{1}{\sqrt{\nu^2-\mu^2}} \,; \qquad \beta = -\frac{\ii}{2} \ds\frac{\partial}{\partial t}\ln\frac{\nu+\mu}{\nu-\mu}\;,
\end{equation}
which coincide with \eqref{E_nu_mu_R42} and \eqref{beta_nu_mu_R42}.
Next we show that the four functions: $\nu$, $\mu$, $E$ and $\beta$ satisfy the integrability conditions
of the system \eqref{Frene_Phi_bar_Phi_n1_n2_R42}. Applying the first equality in \eqref{E_beta_nu_mu_R42} to
the first natural equation \eqref{Nat_Eq_nu_mu_R42}, we get exactly the integrability condition
\eqref{Nat_Eq_Gauss_E_nu_mu_R42}. Applying the two equalities in \eqref{E_beta_nu_mu_R42} to the second natural equation
\eqref{Nat_Eq_nu_mu_R42} and using $\Delta = \frac{4\partial^2}{\partial \bar t \partial t}$, we obtain exactly 
the integrability condition \eqref{Nat_Eq_Ricci_E_nu_mu_R42}. Finally, the first equality in \eqref{E_beta_nu_mu_R42}
implies that $\nu-\mu=\frac{1}{E^2(\nu+\m)}$ and $\nu+\mu=\frac{1}{E^2(\nu-\m)}$.
Applying anyone of the last equalities to the second formula in \eqref{E_beta_nu_mu_R42} by using complex conjugation, 
we find:
\begin{equation*}
E(\nu+\mu)\bar\beta =  \ii \ds\frac{\partial(E(\nu+\mu))}{\partial\bar t}\;; \qquad
E(\nu-\mu)\bar\beta = -\ii \ds\frac{\partial(E(\nu-\mu))}{\partial\bar t}\;.
\end{equation*}
Adding and subtracting the last equalities, we obtain the integrability conditions 
\eqref{Nat_Eq_Codazzi_beta_E_nu_mu_R42}.

Consider now equalities \eqref{Frene_Phi_bar_Phi_n1_n2_R42},  as a system of PDE's with unknown complex vector function
$\Phi$ and unknown  real vector functions $\n_1$ and $\n_2$. 
The integrability conditions are exactly \eqref{Nat_Eq_Gauss_E_nu_mu_R42}, \eqref{Nat_Eq_Codazzi_beta_E_nu_mu_R42} and
\eqref{Nat_Eq_Ricci_E_nu_mu_R42} according Proposition \ref{Frene_integr_cond_R42}, which are satisfied.
This means that the system has locally a unique solution under the given initial conditions.

 Let $(\x_{u0},\x_{v0},\n_{1|0},\n_{2|0})$ be a right oriented orthogonal vector quadruple in $\RR^4_2$, such
that $\x_{u0}^2=\x_{v0}^2=E(t_0)$ and $\n_{1|0}^2=\n_{2|0}^2=-1$. Then there exist vector functions 
$\Phi$, $\n_1$ and $\n_2$, defined in a neighborhood $\D_0$ of $t_0$, which satisfy \eqref{Frene_Phi_bar_Phi_n1_n2_R42} 
with initial conditions $\Phi(t_0)=\x_{u0}-\ii\x_{v0}$, $\n_1(t_0)=\n_{1|0}$ and $\n_2(t_0)=\n_{2|0}$. 
Next we prove that in a sufficiently small neighborhood $D_0$, for any $t\in\D_0$ \ $(\Phi,\bar\Phi,\n_1,\n_2)$ 
is an orthogonal quadruple with the properties $\|\Phi\|^2=2E$ and $\n_1^2=\n_2^2=-1$. To this end, let us consider
the functions:
$f_1=\|\Phi\|^2-2E$, $f_2=\n_1^2+1$, $f_3=\n_2^2+1$, $f_4=\Phi^2$, $f_5=\Phi\cdot\n_1$ and $f_6=\Phi\cdot\n_2$.
Direct calculations show that the functions $f_1$,\,$f_2$,\,$f_3$,\,$f_4$,\,$f_5$,\,$f_6$,\,$\bar f_4$,\,$\bar f_5$ 
and $\bar f_6$ are a solution to a homogeneous system of the first order, solved with respect to the derivatives.
The initial conditions for $\Phi$, $\n_1$ and $\n_2$ imply that all the functions $f_j$ have zero initial conditions. 
The zero functions are the unique solution to a homogeneous system with zero initial conditions. 
Hence all the functions $f_j$ are equal to zero, 
which implies the necessary properties for the functions $\Phi$, $\n_1$ and $\n_2$.

 These properties of $\Phi$ imply that $\Phi^2=0$ and $\|\Phi\|^2>0$, and it follows from the system 
\eqref{Frene_Phi_bar_Phi_n1_n2_R42} that $\frac{\partial \Phi}{\partial\bar t}=0$. This means that $\Phi$ 
satisfies the conditions \eqref{Phi_cond} and according to Theorem \ref{x_Phi-thm}\,, there exists a vector function
$\x$ in a neighborhood $\D_0$ of $t_0$ satisfying \eqref{Phi_def}, and such that $(\D_0,\x)$ is a space-like 
surface in isothermal coordinates $(u,v)$. It remains to check that this surface satisfies the conditions of the theorem.

 Since $\Phi$ is holomorphic, the surface $(\D_0,\x)$ is maximal, according to Theorem \ref{Max_x_Phi-thm}\,. 
The equality $\|\Phi\|^2=2E$ shows that the coefficient $E$ of the first fundamental form of $(\D_0,\x)$
coincides with the function $E$, defined by \eqref{E_beta_nu_mu_R42}. Taking a square of the second  equality in
\eqref{Frene_Phi_bar_Phi_n1_n2_R42}, we have: ${\Phi'}^2=E^2(-\nu^2+\mu^2)$. In view of 
\eqref{Phi_prim_bot^2-Phi_prim^2} and \eqref{E_beta_nu_mu_R42} we get: ${\Phi^{\prime \bot}}^2={\Phi^\prime}^2=-1$. 
The last equality shows that $(u,v)$ are canonical coordinates of the first type.
Taking scalar multiplication of the second equality in \eqref{Frene_Phi_bar_Phi_n1_n2_R42} separately with
$\n_1$ and $\n_2$ and taking into account \eqref{PhiPr}, we find:
\begin{equation}\label{sigma_n1_n2_R42}
\sigma(\x_u,\x_u)\cdot\n_1=E\nu; \ \sigma(\x_u,\x_v)\cdot\n_1=0; \  
\sigma(\x_u,\x_u)\cdot\n_2=0;    \ \sigma(\x_u,\x_v)\cdot\n_2=E\mu\,.
\end{equation}
Now it follows that $\sigma(\x_u,\x_u)$ and $\sigma(\x_u,\x_v)$ are collinear with $\n_1$ and $\n_2$, respectively.
Since $\nu>0$ then $\sigma(\x_u,\x_u)$ has the opposite direction of $\n_1$. This means that $(\n_1, \n_2)$
coincides with the canonical basis at $N_p(\M)$, generated by the canonical coordinates $(u,v)$. Now it follows from
\eqref{sigma_n1_n2_R42} that the pair of scalar invariants $(\nu, \mu)$ of $(\D_0,\x)$, defined by 
\eqref{sigma_nu_mu_R42}, coincides with the pair of the given functions $(\nu, \mu)$.
This completes the proof of the existence.

\smallskip

 Suppose that $(\hat\D_0,\hat\x)$ is another surface, satisfying the conditions of the theorem.
Then $(\hat\x_u(t_0),\hat\x_v(t_0),\hat\n_1(t_0),\hat\n_2(t_0))$ is a right oriented vector quadruple in $\RR^4_2$,
satisfying the conditions $\hat\x_u^2(t_0)=\hat\x_v^2(t_0)=E(t_0)$ and $\hat\n_1^2(t_0)=\hat\n_2^2(t_0)=-1$. 
Let $A \in \mathbf{SO}(2,2,\RR)$ denote the motion, which transforms the quadruple  
$(\x_u(t_0),\x_v(t_0),\n_1(t_0),\n_2(t_0))$ into  $(\hat\x_u(t_0),\hat\x_v(t_0),\hat\n_1(t_0),\hat\n_2(t_0))$.
Define the map $\hat{\hat\x}$ from $\D_0$ in $\RR^4_2$ by $\hat{\hat\x}=A\x$.
Then $(\D_0,\hat{\hat\x})$ is a maximal space-like surface, obtained from $(\D_0,\x)$ by a proper motion.
Therefore $(u,v)$ are canonical coordinates for $(\D_0,\hat{\hat\x})$ and the functions
$E$, $\nu$ and $\mu$ of $(\D_0,\hat{\hat\x})$ coincide with those of $(\D_0,\x)$.
Equality \eqref{beta_nu_mu_R42} implies that the function $\beta$ of $(\D_0,\hat{\hat\x})$ also coincides with
that of $(\D_0,\x)$. Hence the corresponding functions $\hat{\hat\Phi}$, ${\hat{\hat\n}}_1$ and 
${\hat{\hat\n}}_2$ also satisfy the system \eqref{Frene_Phi_bar_Phi_n1_n2_R42}. Taking into account 
the definition of $A$ it follows that $\hat{\hat\Phi}$, ${\hat{\hat\n}}_1$ and ${\hat{\hat\n}}_2$ satisfy 
the same initial conditions by $t=t_0$, as well as the functions $\hat\Phi$, $\hat\n_1$ and $\hat\n_2$ of
$(\hat\D_0,\hat\x)$. 
Since the system \eqref{Frene_Phi_bar_Phi_n1_n2_R42} has locally a unique solution under the given initial 
conditions, then there exists a connected neighborhood $\tilde\D_0$ of $t_0$, in which $\hat\Phi=\hat{\hat\Phi}$. 
On the other hand, $\hat{\hat\x}=A\x$ implies that $\hat{\hat\Phi}=A\Phi$. Hence $\hat\Phi=A\Phi$ in  
$\tilde\D_0$. In the case of a connected neighborhood $\tilde\D_0$, the last equality is equivalent to 
$\hat\x = A\x + \vb$,\ \ $\vb\in\RR^4_2$. 
\end{proof}

 Since the systems of natural equations \eqref{Nat_Eq_nu_mu_R42} and \eqref{Nat_Eq_K_kappa_R42} are equivalent via 
formulas \eqref{K_kappa_nu_mu_R42} and \eqref{nu_mu_K_kappa_R42}, then Theorem \ref{Bone_Phi_bar_Phi_n1_n2_nu_mu_R42} 
implies an analogous theorem for the pair $(K,\varkappa)$: 
\begin{theorem}\label{Bone_Phi_bar_Phi_n1_n2_K_kappa_R42}
Let $K>0$ and $\varkappa$ be real functions, defined in a domain $\D\subset\RR^2$ and let the pair 
$(K,\varkappa)$ be a solution to the system of natural equations \eqref{Nat_Eq_K_kappa_R42} of
the maximal space-like surfaces in $\RR^4_2$. Then for any point $p_0\in\D$, there exists a neighborhood 
$\D_0\subset\D$ of $p_0$ and a map $\x: \D_0 \to \RR^4_2$, such that $(\D_0,\x)$ is a maximal space-like surface
of general type in $\RR^4_2$, parametrized by canonical coordinates and the given functions $K$ and $\varkappa$  
are the Gauss curvature and the curvature of the normal connection for $(\D_0,\x)$, respectively.
If $(\hat\D_0,\hat\x)$ is another surface with the same properties, then there exists a sub-domain $\tilde\D_0$
of $\D_0$ and $\hat\D_0$, containing $p_0$, such that $(\tilde\D_0,\hat\x)$ is obtained from
$(\tilde\D_0,\x)$ by a proper motion in $\RR^4_2$ of the type \eqref{hat_M-M-prop_mov_R42}.
\end{theorem}


 At the end of the previous section we obtained transformation formulas for the invariants $K$, $\varkappa$, $\nu$ and $\mu$
through some of the basic geometric transformations of the maximal space-like surface $\M$ in $\RR^4_2$. 
It follows from the above results in this section that the statements at the end of the previous section are reversible.

For example, if the pairs $(\nu,\mu)$ of two different maximal space-like surfaces of general type are related by
equalities \eqref{hat_K_kappa_nu_mu-homotet_R42}, then these surfaces are homothetic up to a motion.

If two maximal space-like surfaces of general type are obtained one from the other by a proper motion, a homothety 
or a change of the canonical parameters, then from a geometric point of view they can be considered as identical.
That is why we give in a united form these transformation formulas. 
\begin{prop}\label{Thm-K_kappa_nu_mu_IdentSurf_R42}
Let $\M$ and $\hat\M$ be maximal space-like surfaces of general type and $p_0\in\M$ and $\hat p_0\in\hat\M$ be
fixed points in them. If $t$ determines canonical coordinates on $\M$ in a neighborhood of $p_0$, 
and $s$ determines canonical coordinates on $\hat\M$ in a neighborhood of $\hat p_0$, then the following 
conditions are equivalent:
\begin{enumerate}
	\item There exists a neighborhood of $\hat p_0$ in $\hat\M$, which is obtained from the corresponding neighborhood
	of $p_0$ in $\M$ by a proper motion, homothety or change of the canonical coordinates.
	\item There exists a neighborhood of $\hat p_0$ in $\hat\M$, in which the normal curvatures $\hat\nu(s)$ and 
	$\hat\mu(s)$ of $\hat\M$ are obtained from the normal curvatures $\nu(t)$ and $\mu(t)$ of $\M$, by formulas of the type: 
  \begin{equation}\label{hat_nu_mu-Ident_R42}
  \hat\nu(s) = a^2 \nu(\delta a \tilde s + b)\;; \qquad \hat\mu(s) = a^2 \mu(\delta a \tilde s + b)\;.
  \end{equation}
  where $\delta=\pm 1;\pm\ii$, $\tilde s$ denotes $s$ or $\bar s$, and  $a>0$, $b\in\CC$ are constants\,. 
	\item There exists a neighborhood of $\hat p_0$ in $\hat\M$, in which the Gauss curvature $\hat K(s)$ 
	and the curvature of the normal connection $\hat\varkappa(s)$ of $\hat\M$ are obtained from the corresponding curvatures
	$K(t)$ and $\varkappa(t)$ of $\M$, by formulas of the type: 
  \begin{equation}\label{hat_K_kappa-Ident_R42}
  \hat K(s) = a^4 K(\delta a \tilde s + b)\;; \qquad \hat\varkappa(s) = a^4 \varkappa(\delta a \tilde s + b)\;.
  \end{equation}
  where $\delta=\pm 1;\pm\ii$, $\tilde s$ denotes $s$ or $\bar s$, and $a>0$, $b\in\CC$ are constants\,. 
\end{enumerate}
\end{prop}
\begin{proof}
Formulas \eqref{hat_nu_mu-Ident_R42} and \eqref{hat_K_kappa-Ident_R42} can be obtained by a composition of
formulas \eqref{can_coord_ch_R42} and \eqref{hat_K_kappa_nu_mu-homotet_R42} and changing the denotations of the constants. 
Then the equivalence of the conditions 1. and 2. follows from Theorem \ref{Thm-Nat_Eq_nu_mu_R42} and
Theorem \ref{Bone_Phi_bar_Phi_n1_n2_nu_mu_R42}\,. In a similar way, the equivalence of the conditions 1. and 3. 
follows from Theorem \ref{Thm-Nat_Eq_K_kappa_R42} and Theorem \ref{Bone_Phi_bar_Phi_n1_n2_K_kappa_R42}\,.
\end{proof}


\section{Weierstrass representation of maximal space-like surfaces in $\RR^4_2$}\label{sect_W}

Let $\M=(\D,\x)$ be a maximal space-like surface in $\RR^4_2$ and $t=u+\ii v$ determine isothermal coordinates on $\M$. 
This condition in terms of the function $\Phi$ is expressed by $\Phi^2=0$. If the components of the function $\Phi$ are given by
$\Phi=(\phi_1,\phi_2,\phi_3,\phi_4)$, then the equality $\Phi^2=0$ can be written in the form:
\begin{equation}\label{phi2-0_coord_R42}
\phi_1^2+\phi_2^2-\phi_3^2-\phi_4^2=0\,.
\end{equation}
The last equality can be ''parametrized'' in different ways by means of three holomorphic functions. 
Similarly to minimal surfaces in $\RR^4$ and space-like surfaces with zero mean curvature vector in $\RR^4_1$, the relation 
\eqref{phi2-0_coord_R42} can be ''parametrized'' in such a way that the components of $\Phi$ are polynomials of these three 
holomorphic functions.

 First, let $(f,g_1,g_2)$ be a triple of holomorphic functions, defined in the domain $\D \subset \CC$ and
consider the function $\Phi$, defined by:
\begin{equation}\label{W_polinom_R42}
\Phi = \big(\, f(g_1 g_2+1) \,,\, \ii f(g_1 g_2-1) \,,\, f(g_1 + g_2) \,,\, \ii f(g_1 - g_2) \,\big).
\end{equation}
Next we shall find the conditions for $(f,g_1,g_2)$, under which the function $\Phi$ satisfies Theorem \ref{x_Phi-thm}\,. 
In order to simplify the calculations, we introduce the auxiliary vector function $\va$:
\begin{equation}\label{va-def_R42}
\va=\big(\, (g_1 g_2+1) \,,\, \ii (g_1 g_2-1) \,,\, (g_1 + g_2) \,,\, \ii (g_1 - g_2) \,\big).
\end{equation}
Then we have:
\begin{equation}\label{Phi-va_R42}
\Phi=f\va\,; \qquad \Phi^2=f^2\va^2\,; \qquad \|\Phi\|^2=|f|^2\|\va\|^2\,; \qquad \Phi'=f'\va + f\va'\,. 
\end{equation}

 First we calculate $\va^2$. If $\va=(a_1,a_2,a_3,a_4)$, then
$a_1^2=g_1^2 g_2^2 + 2 g_1 g_2 + 1$ and respectively $a_2^2=-g_1^2 g_2^2 + 2 g_1 g_2 - 1$.
This implies that $a_1^2+a_2^2 = 4 g_1 g_2$. Analogously we have
$a_3^2=g_1^2 + 2 g_1 g_2 + g_2^2$ and $a_4^2= - g_1^2 + 2 g_1 g_2 - g_2^2$, from where $a_3^2+a_4^2= 4 g_1 g_2$.
Therefore, $a_1^2+a_2^2=a_3^2+a_4^2$, which implies that $\va^2=0$. It follows from \eqref{Phi-va_R42} that 
equality $\Phi^2=0$ is valid.
This means that $\Phi$ satisfies the first of the conditions \eqref{Phi_cond} of Theorem \ref{x_Phi-thm}\,.

 Next we consider $\|\va\|^2$. We have $|a_1|^2=|g_1|^2 |g_2|^2 + g_1 g_2 + \bar g_1 \bar g_2 + 1$ and respectively
$|a_2|^2=|g_1|^2 |g_2|^2 - g_1 g_2 - \bar g_1 \bar g_2 + 1$. Hence $|a_1|^2+|a_2|^2=2|g_1|^2 |g_2|^2 + 2$\,.
Analogously we have 
$|a_3|^2=|g_1|^2 + g_1 \bar g_2 + \bar g_1  g_2 + |g_2|^2$ and $|a_4|^2=|g_1|^2 - g_1 \bar g_2 - \bar g_1  g_2 + |g_2|^2$,
which gives $|a_3|^2+|a_4|^2=2|g_1|^2 + 2|g_2|^2$. Therefore we get for $\|\va\|^2$:
\begin{equation}\label{n_va^2_g1g2_R42}
\|\va\|^2 = 2|g_1|^2 |g_2|^2 + 2 - 2|g_1|^2 - 2|g_2|^2 = 2(|g_1|^2-1)(|g_2|^2-1)\,.
\end{equation}
From the last formula and from \eqref{Phi-va_R42} we get the corresponding formula for $\|\Phi\|^2$:
\begin{equation}\label{n_Phi^2_g1g2_R42}
\|\Phi\|^2 = 2|f|^2(|g_1|^2-1)(|g_2|^2-1)\,.
\end{equation}
This means that $\Phi$ satisfies the second of the conditions \eqref{Phi_cond} in Teorem \ref{x_Phi-thm}\,,
if and only if the triple $(f,g_1,g_2)$ satisfies the conditions:
\begin{equation}\label{W_polinom_cond_R42}
f\neq 0\,; \  |g_1|>1\,; \  |g_2|>1 \qquad \text{or} \qquad  f\neq 0\,; \  |g_1|<1\,; \  |g_2|<1 \,.
\end{equation}

 Since the functions $(f,g_1,g_2)$ are holomorphic, then $\Phi$ is also holomorphic and consequently 
satisfy the third of the conditions \eqref{Phi_cond} of Theorem \ref{x_Phi-thm}\,.
According to the last mentioned theorem, the conditions \eqref{W_polinom_cond_R42} imply that there exists 
a sub-domain $\D_0\subset\D$ containing $t_0 \in \D$ and a function $\x : \D_0 \to \RR^4_2$, such that 
$(\D_0,\x)$ is a space-like surface in $\RR^4_2$ with isothermal coordinates $(u,v)$ determined by $t=u+\ii v$, 
which satisfies \eqref{Phi_def}. 

Further, it follows from Theorem \ref{Max_x_Phi-thm} that the surface $(\D_0,\x)$ is a maximal space-like surface.
We say that this surface is \textit{generated by the triple of functions $(f,g_1,g_2)$} through the formula
\eqref{W_polinom_R42}

\smallskip

 Next we see how the functions $f$, $g_1$ and $g_2$, from formula \eqref{W_polinom_R42}, 
can be expressed explicitly via the components of $\Phi$. It follows from \eqref{W_polinom_R42} immediately:
\[
\begin{array}{l}
\phi_1+\ii\phi_2=f(g_1 g_2+1)-f(g_1 g_2-1)=2f \,;\\
\phi_3+\ii\phi_4=f(g_1+g_2)-f(g_1-g_2)=2fg_2  \,;\\
\phi_3-\ii\phi_4=f(g_1+g_2)+f(g_1-g_2)=2fg_1  \,.
\end{array}
\]
From the above equalities we get for $f$, $g_1$ and $g_2$:
\begin{equation}\label{fg1g2_Phi_R42}
f=\ds\frac{1}{2}(\phi_1+\ii\phi_2)\,; \qquad
g_1=\ds\frac{\phi_3-\ii\phi_4}{\phi_1+\ii\phi_2}\;; \qquad
g_2=\ds\frac{\phi_3+\ii\phi_4}{\phi_1+\ii\phi_2}\;.
\end{equation}

 The last equalities show how for a given maximal space-like surface  we can obtain a representation 
of the type \eqref{W_polinom_R42}. Let $\M=(\D,\x)$ be such a surface and $\Phi$ be the corresponding 
function defined by \eqref{Phi_def} and suppose that $\phi_1+\ii\phi_2 \neq 0$\,. 
Defining the triple $(f,g_1,g_2)$ by means of formulas \eqref{fg1g2_Phi_R42}, we get:
\[
g_1 g_2+1 = \frac{(\phi_3-\ii\phi_4)(\phi_3+\ii\phi_4)}{(\phi_1+\ii\phi_2)^2} + 1
= \frac{\phi_3^2+\phi_4^2}{(\phi_1+\ii\phi_2)^2} + 1\,.
\]
Applying to the last formula the identity $\phi_3^2+\phi_4^2=\phi_1^2+\phi_2^2$ we find:
\[
g_1 g_2+1 = \frac{\phi_1^2+\phi_2^2}{(\phi_1+\ii\phi_2)^2} + 1
= \frac{(\phi_1-\ii\phi_2)(\phi_1+\ii\phi_2)}{(\phi_1+\ii\phi_2)^2} + 1
= \frac{\phi_1-\ii\phi_2}{\phi_1+\ii\phi_2} + 1 = \frac{2\phi_1}{\phi_1+\ii\phi_2}\,.
\]
Multiplying the last equality with the first equality in \eqref{fg1g2_Phi_R42}, we obtain $f(g_1 g_2+1)=\phi_1$.
Thus we obtained the required formula for the first coordinate in \eqref{W_polinom_R42}. In a similar way we find the 
formulas for the remaining three coordinates in \eqref{W_polinom_R42}. Hence we obtained that the given surface
has a representation of the type \eqref{W_polinom_R42}. This formula is the analogue of the classical 
Weierstrass representation for the minimal surfaces in $\RR^3$. We call formula \eqref{W_polinom_R42}
a \emph{Weierstrass representation} for the maximal space-like surfaces in $\RR^4_2$.

  The above representation \eqref{W_polinom_R42} and its properties can be summarized in the following
statement:  
\begin{theorem}\label{thm-W_polinom_R42}
Let $\M$ be a maximal space-like surface in $\RR^4_2$ with a fixed point $p$ in it and \,$t=u+\ii v$ determines 
isothermal coordinates in a neighborhood of $p$. Suppose that the function $\Phi$, defined by \eqref{Phi_def}, 
satisfies the condition $\phi_1 + \ii\phi_2 \neq 0$\,. Then there exists a neighborhood of $p$, in which the function
$\Phi$ has a Weierstrass representation \eqref{W_polinom_R42}, where $f$, $g_1$ and $g_2$ are a triple of 
holomorphic functions satisfying the conditions \eqref{W_polinom_cond_R42}. The functions 
$f$, $g_1$ and $g_2$ are determined uniquely by $\Phi$ according to \eqref{fg1g2_Phi_R42}.

 Conversely, let $(f,g_1,g_2)$ be three holomorphic functions, defined in a domain $\D \subset \CC$, 
satisfying the conditions \eqref{W_polinom_cond_R42} and let $t_0$ be a point in $\D$.
Then there exists a sub-domain $\D_0$ in $\D$, containing $t_0$ and a maximal space-like surface
$(\D_0,\x)$ such that the corresponding function $\Phi$ has the representation \eqref{W_polinom_R42} 
in $\D_0$ by means of the given functions and satisfies the condition $\phi_1 + \ii\phi_2 \neq 0$\,. 
\end{theorem}
\begin{remark}
The restriction $\phi_1 + \ii\phi_2 \neq 0$ in the last theorem is not an essential geometric condition for $\M$.
If a maximal space-like surface $\M$ satisfies the condition $\phi_1 + \ii\phi_2 = 0$, then it can 
be transformed by a proper motion into a surface with the property $\phi_1 + \ii\phi_2 \neq 0$.
Further, all geometric properties will be invariant under proper motions in $\RR^4_2$, and hence will be valid 
for surfaces satisfying the equality
 $\phi_1 + \ii\phi_2 = 0$\,. 
\end{remark}

 Our next goal is to obtain formulas for the change of the functions in the Weierstrass representation 
\eqref{W_polinom_R42} under a change of the isothermal coordinates and under the basic geometric transformations 
of the given maximal space-like surface.

First we consider a holomorphic change of the coordinates of the type $t=t(s)$. If $(\tilde f, \tilde g_1, \tilde g_2)$ 
is the triple of functions corresponding to $s$ in the representation \eqref{W_polinom_R42}, 
then it follows from \eqref{Phi_s-hol} that:
\begin{equation}\label{f_g_s-hol_R42}
\tilde f(s) = f(t(s)) t'(s)\,; \qquad \tilde g_1(s) = g_1(t(s))\,; \qquad \tilde g_2(s) = g_2(t(s))\,.
\end{equation}

 In the case of an anti-holomorphic change of the coordinates it is sufficient to consider the change $t=\bar s$.
Then according to formula \eqref{Phi_s-t_bs}, the representation \eqref{W_polinom_R42} satisfies:
\begin{equation*}
\tilde\Phi(s): \quad
\begin{array}{rllll}
\tilde\phi_1 &=& \phantom{-\ii} \bar f(\bar s) (\bar g_1(\bar s) \bar g_2(\bar s) + 1) &=& 
\phantom{\ii} \bar f(\bar s) \bar g_1(\bar s) \bar g_2(\bar s)
\left(\ds\frac{1}{\bar g_1(\bar s)} \frac{1}{\bar g_2(\bar s)} + 1\right) ;\\[2ex]
\tilde\phi_2 &=&          -\ii  \bar f(\bar s) (\bar g_1(\bar s) \bar g_2(\bar s) - 1) &=& 
         \ii  \bar f(\bar s) \bar g_1(\bar s) \bar g_2(\bar s)
\left(\ds\frac{1}{\bar g_1(\bar s)} \frac{1}{\bar g_2(\bar s)} - 1\right) ;\\[2ex]
\tilde\phi_3 &=& \phantom{-\ii} \bar f(\bar s) (\bar g_1(\bar s) + \bar g_2(\bar s)) &=& 
\phantom{\ii} \bar f(\bar s) \bar g_1(\bar s) \bar g_2(\bar s)
\left(\ds\frac{1}{\bar g_1(\bar s)} + \frac{1}{\bar g_2(\bar s)}\right) ;\\[2ex]
\tilde\phi_4 &=&          -\ii  \bar f(\bar s) (\bar g_1(\bar s) - \bar g_2(\bar s)) &=& 
         \ii  \bar f(\bar s) \bar g_1(\bar s) \bar g_2(\bar s)
\left(\ds\frac{1}{\bar g_1(\bar s)} - \frac{1}{\bar g_2(\bar s)}\right) .
\end{array}
\end{equation*}
Therefore the triple $(f,g_1,g_2)$ is transformed in the following way:
\begin{equation}\label{f_g_s-t_bs_R42}
\tilde f(s) = \bar f(\bar s)\bar g_1(\bar s)\bar g_2(\bar s)\,; \qquad \tilde g_1(s) = \frac{1}{\bar g_1(\bar s)}\:;
\qquad \tilde g_2(s) = \frac{1}{\bar g_2(\bar s)}\:.
\end{equation}

\smallskip

 Now, let us consider a motion in $\RR^4_2$ and find the transformation formulas for the triple $(f,g_1,g_2)$
in the representation \eqref{W_polinom_R42}. The basic result is that the functions $g_1$ and $g_2$ are
transformed with linear fractional transformations of a certain type. In order to describe these transformations,
we consider the indefinite Hermitian scalar product in $\CC^2$, defined as follows:
for any two elements $\vect{z}=(z_1,z_2)$ and $\vect{w}=(w_1,w_2)$ in $\CC^2$ the scalar product is defined by
$\vect{z}\cdot \bar{\vect{w}} = z_1\bar w_1 - z_2\bar w_2$. The space $\CC^2$ endowed by this scalar product
is denoted by $\CC^2_1$. We shall find that the functions $g_1$ and $g_2$ are transformed 
by means of linear fractional transformations, given by matrices, which are unitary or anti-unitary with respect to
the indefinite Hermitian scalar product in $\CC^2_1$.

For this purpose, we need some formulas from the theory of spinors in $\RR^4_2$. 
We recall these formulas in a form convenient to apply to maximal space-like surfaces. To any vector $\x$ 
in $\RR^4_2$ we associate a complex $2\times 2$ matrix $S$ in the following way: 
\begin{equation}\label{Spin_S-x_R42}
\x=(x_1,x_2,x_3,x_4)
 \ \leftrightarrow \ 
S=\left(
\begin{array}{cc}
   x_3 + \ii x_4  & x_1 + \ii x_2\\
   x_1 - \ii x_2  & x_3 - \ii x_4
\end{array}
\right) = 
\left(
\begin{array}{cc}
        z_1  &      z_2 \\
   \bar z_2  & \bar z_1
\end{array}
\right) .
\end{equation}
It follows immediately that the matrices of the type $S$ in \eqref{Spin_S-x_R42} form a real linear 
subspace of the space of all $2\times 2$ matrices. By a direct verification it follows that 
this space is closed with respect to the matrix multiplication and the correspondence \eqref{Spin_S-x_R42}
has the following important property: $\det S = -\x^2$. This means that any linear operator,
acting in the space of the matrices of the type $S$ in \eqref{Spin_S-x_R42} and preserving the determinant, 
generates an orthogonal operator in $\RR^4_2$. 

 If $U$ and $V$ are special unitary $2\times 2$ matrices in $\mathbf{SU}(1,1,\CC)$, then they have the form
\eqref{Spin_S-x_R42} and therefore $U S V^*$ is a matrix of the same type, where $V^*$ denotes the Hermitian 
conjugate (with respect to the scalar product in $\CC^2_1$) of $V$. What is more, $\det U S V^* = \det S$.
As we mentioned above, to any pair of matrices $(U,V)$ in $\mathbf{SU}(1,1,\CC) \times \mathbf{SU}(1,1,\CC)$ 
we can associate an orthogonal matrix $A$ in $\mathbf{O}(2,2,\RR)$. Hence, we have a group homomorphism
$(U,V) \rightarrow A$, which is given schematically in the following way:
\begin{equation}\label{Spin_UV-A_R42}
\hat S = U S V^* \ \rightarrow \ \hat\x = A\x\,.
\end{equation}
This homomorphism from $\mathbf{SU}(1,1,\CC) \times \mathbf{SU}(1,1,\CC)$ to $\mathbf{O}(2,2,\RR)$ is said 
to be spinor map.

It is proved in the theory of spinors that the kernel of the spinor map consists of two elements:
$(I,I)$ and $(-I,-I)$, where $I$ denotes the unit matrix. For the image of this map it is proved that
it coincides with the connected component of the unit in $\mathbf{O}(2,2,\RR)$, which is denoted 
in a standard way by $\mathbf{SO}^+(2,2,\RR)$.
In physical terms, this is the group of matrices giving such orthogonal transformations in $\RR^4_2$,
which preserve the orientation not only of the whole space $\RR^4_2$, but also the orientation of  
the two-dimensional time-like subspace and the orientation of the two-dimensional space-like subspace  
of $\RR^4_2$. The motions in $\RR^4_2$, preserving the orientation of the two-dimensional time-like 
subspace are called orthochronous motions in $\RR^4_2$.

Taking into account the kernel and the image of the spinor map \eqref{Spin_UV-A_R42} it follows that 
this map induces the following group isomorphism:
\begin{equation}\label{SU11^2/+-1_SO22+}
\mathbf{SU}(1,1,\CC) \times \mathbf{SU}(1,1,\CC)/\{(I,I),(-I,-I)\}\ \cong \ \mathbf{SO}^+(2,2,\RR)\,.
\end{equation}
This means that $\mathbf{SU}(1,1,\CC) \times \mathbf{SU}(1,1,\CC)$ is a double covering of 
$\mathbf{SO}^+(2,2,\RR)$ and therefore we can identify it with the spin-group $\mathbf{Spin}(2,2)$ 
of $\mathbf{SO}^+(2,2,\RR)$. In other words, \eqref{SU11^2/+-1_SO22+} gives a representation of
$\mathbf{Spin}(2,2)$ as $\mathbf{SU}(1,1,\CC) \times \mathbf{SU}(1,1,\CC)$, which is called
spinor representation.

 Until now, we have restricted $\x$ to be a real vector in $\RR^4_2$. Since the relations
\eqref{Spin_S-x_R42} and \eqref{Spin_UV-A_R42} are linear with respect to $\x$, then they continue to be valid if we 
replace $\x$ with an arbitrary complex vector in $\CC^4$. The only difference is that $S$ can be an arbitrary
complex matrix. Under a motion of the complex vector $\x$ with a matrix in $\mathbf{SO}^+(2,2,\RR)$, the
corresponding matrix $S$ is transformed by the same formula \eqref{Spin_UV-A_R42}, as in the case of a real vector $\x$.

\smallskip

 Now, let us return to the maximal space-like surfaces. Using the above formulas we shall find 
how the functions in the Weierstrass representation change under a motion of the given surface in $\RR^4_2$.  
Let $\M=({\D},\x)$ and $\hat\M=(\D,\hat\x)$ be two maximal space-like surfaces, parametrized by isothermal 
coordinates $t=u+\ii v$. Suppose that $\hat\M$ is obtained from $\M$ by a proper orthochronous motion in $\RR^4_2$
of the type $\hat\x(t)=A\x(t)+\vb$, where $A \in \mathbf{SO}^+(2,2,\RR)$ and $\vb \in \RR^4_2$. 
If $\Phi$ is the function, defined by \eqref{Phi_def}, then $\Phi$ is transformed by the formula 
$\hat\Phi=A\Phi$ as we know from \eqref{hat_Phi-Phi-mov}. Next we introduce the matrix $S_{\Phi}$, 
which is obtained  from $\Phi$ according to the rule \eqref{Spin_S-x_R42}:
\begin{equation}\label{SPhi_R42_def}
S_{\Phi}=\left(
\begin{array}{cc}
   \phi_3 + \ii\phi_4  &  \phi_1 + \ii\phi_2\\
   \phi_1 - \ii\phi_2  &  \phi_3 - \ii\phi_4 
\end{array}
\right) .
\end{equation}
Denote by $(U,V)$ an arbitrary pair of matrices in $\mathbf{SU}(1,1,\CC) \times \mathbf{SU}(1,1,\CC)$, 
corresponding to $A$ by the homomorphism \eqref{Spin_UV-A_R42}. If $S_{\hat\Phi}$ is the matrix obtained from
$\hat\M$, then according to \eqref{Spin_UV-A_R42} it is related to $S_{\Phi}$ in the following way:
\begin{equation}\label{hatSPhi_SPhi_R42}
S_{\hat\Phi} = U S_{\Phi} V^* .
\end{equation}

 Suppose that $\M$ is given by a Weierstrass representation of the type \eqref{W_polinom_R42}. We 
obtain by direct computations:
\begin{equation*}
\begin{array}{llll}
\phi_1+\ii\phi_2 = f(g_1g_2+1)&-&f(g_1g_2-1) &= 2f       \,;\\
\phi_1-\ii\phi_2 = f(g_1g_2+1)&+&f(g_1g_2-1) &= 2fg_1g_2 \,;\\
\phi_3+\ii\phi_4 = f(g_1+g_2) &-&f(g_1-g_2)  &= 2fg_2 \,;\\
\phi_3-\ii\phi_4 = f(g_1+g_2) &+&f(g_1-g_2)  &= 2fg_1 \,.\\
\end{array}
\end{equation*}
Hence the matrix $S_{\Phi}$ is represented by $f$, $g_1$ and $g_2$ in the following way:
\begin{equation}\label{SPhi_fg1g2_R42}
S_{\Phi}=\left(
\begin{array}{ll}
   2fg_2     &  2f \\
   2fg_1g_2  &  2fg_1
\end{array}
\right) .
\end{equation}
Denoting the elements of $S_{\Phi}$ by $s_{ij}$, we have the following for $f$, $g_1$ and $g_2$\,: 
\begin{equation}\label{fg1g2_s1234_R42}
f=\frac{1}{2}s_{12}\,; \qquad g_1=\frac{s_{22}}{s_{12}}\,; \qquad g_2=\frac{s_{11}}{s_{12}} \,.
\end{equation}
We have already seen that $S_{\Phi}$ was transformed by the rule \eqref{hatSPhi_SPhi_R42}. Next we shall 
find the transformation formulas for the functions $f$, $g_1$ and $g_2$ using \eqref{fg1g2_s1234_R42}.
For that purpose, let us denote the elements of $U$ and $V$ in the following way:
\begin{equation}\label{UV_ab_R42}
U=\left(
\begin{array}{rr}
    \bar a_1  &  b_1 \\
    \bar b_1  &  a_1
\end{array}
\right); \quad
V=\left(
\begin{array}{rr}
    \bar a_2  &  -b_2 \\
   -\bar b_2  &   a_2
\end{array}
\right);
\qquad a_j,b_j \in \CC\,;
\quad |a_j|^2-|b_j|^2=1\,.
\end{equation}
After multiplication of the matrices in \eqref{hatSPhi_SPhi_R42} and the corresponding simplification we get:
\begin{equation}\label{hatSPhi_fg1g2_R4}
S_{\hat\Phi}=\left(
\begin{array}{rr}
     2f(b_1 g_1 + \bar a_1)(a_2 g_2 + \bar b_2)   &   2f(b_1 g_1 + \bar a_1)(b_2 g_2 + \bar a_2) \\
     2f(a_1 g_1 + \bar b_1)(b_2 g_2 + \bar a_2)   &   2f(b_2 g_2 + \bar a_2)(a_1 g_1 + \bar b_1)
\end{array}
\right).
\end{equation}
Now, applying \eqref{fg1g2_s1234_R42} to $\hat f$, $\hat g_1$ and $\hat g_2$, we obtain the required 
transformation formulas for the functions in the Weierstrass representation of the type \eqref{W_polinom_R42} 
under a proper orthochronous motion of $\M$ in $\RR^4_2$\,:
\begin{gather}\label{hatfg_fg-orthchr_move_R42}
\hat f = f(b_1 g_1 + \bar a_1)(b_2 g_2 + \bar a_2) \,; \quad 
\hat g_1 = \frac{a_1 g_1 + \bar b_1}{b_1 g_1 + \bar a_1}\,; \quad 
\hat g_2 = \frac{a_2 g_2 + \bar b_2}{b_2 g_2 + \bar a_2}\,;
\qquad \begin{array}{l} a_j,b_j \in \CC\,; \\ |a_j|^2-|b_j|^2=1\,. \end{array}
\raisetag{-0.8ex}
\end{gather}

 Now we consider the inverse statement. Suppose that $\hat\M=(\D,\hat\x)$ and $\M=(\D,\x)$
are two maximal space-like surfaces whose Weierstrass representations of the type \eqref{W_polinom_R42} 
are related by \eqref{hatfg_fg-orthchr_move_R42}. Next we show that the surfaces are obtained one from 
the other by a proper orthochronous motion in $\RR^4_2$. For that purpose, we introduce the matrices $U$ 
and $V$ by \eqref{UV_ab_R42}. Let $A$ be the corresponding to $(U,V)$ matrix through the homomorphism 
\eqref{Spin_UV-A_R42}. With the aid of $A$ we can obtain a third surface $\hat{\hat\M}=(\D,\hat{\hat\x})$ 
applying the formula: $\hat{\hat\x} = A\x$. As we have proved up to now, $\hat{\hat\M}$ has a Weierstrass 
representation with functions satisfying \eqref{hatfg_fg-orthchr_move_R42}. Therefore, $\hat\M$ and 
$\hat{\hat\M}$ are generated by the same functions through the formulas \eqref{W_polinom_R42} and hence
they are obtained one from the other by a translation in $\RR^4_2$. Since by definition $\hat{\hat\M}$ 
is obtained from $\M$ by means of a proper orthochronous motion, then $\hat\M$ is also obtained from $\M$ 
by a proper orthochronous motion. Summarizing, we obtain the following:
\begin{theorem}\label{W-orthchr_move_R42} 
Let $\hat\M=(\D,\hat\x)$ and $\M=(\D,\x)$ be two maximal space-like  surfaces in $\RR^4_2$, 
given by Weierstrass representations of the type \eqref{W_polinom_R42}, where $\D$ is a connected domain 
in $\CC$. Then the following conditions are equivalent:
\begin{enumerate}
	\item $\hat\M$ and $\M$ are related by a proper orthochronous motion in $\RR^4_2$ of the type:\\
	$\hat\x(t)=A\x(t)+\vb$, where $A \in \mathbf{SO}^+(2,2,\RR)$ and $\vb \in \RR^4_2$.
	\item The functions in the Weierstrass representations of $\hat\M$ and $\M$ are related by
		the equalities \eqref{hatfg_fg-orthchr_move_R42}. 
\end{enumerate}
\end{theorem}

 Next we consider the cases of a motion in $\RR^4_2$, belonging to one of the three remaining connected components
of $\mathbf{O}(2,2,\RR)$. First we consider a proper non-orthochronous motion. We can obtain a concrete such a motion
if we change the sign in the second and the fourth coordinate. Let $\hat\M$ be obtained by $\M$ in this way.
Then we have:
\begin{equation*}
\begin{array}{ll}
\hat\Phi\!\!\! &=\, \big(\, f(g_1 g_2+1) \,,\, -\ii f(g_1 g_2-1) \,,\, f(g_1 + g_2) \,,\, -\ii f(g_1 - g_2) \,\big)\\[1.5ex]
         &=\, \ds\left(fg_1 g_2\left(\frac{1}{g_1}\frac{1}{g_2}+1\right),\, \ii fg_1 g_2\left(\frac{1}{g_1}\frac{1}{g_2}-1\right),\,
         fg_1 g_2\left(\frac{1}{g_1}+\frac{1}{g_2}\right),\, \ii fg_1 g_2\left(\frac{1}{g_1}-\frac{1}{g_2}\right) \right).
\end{array}
\end{equation*}
It follows from the above formulas, that under this concrete motion the triple $(f,g_1,g_2)$ is 
transformed in the following way:
$\hat f = fg_1 g_2$,\, $\hat g_1 = \ds\frac{1}{g_1}$ и $\hat g_2 = \ds\frac{1}{g_2}$. 
The functions $g_1$ and $g_2$ in the last formulas are transformed by linear fractional transformations 
given by anti-unitary ($\|U\vect{z}\|^2=-\|\vect{z}\|^2$) matrices with determinant equal to $-1$. Any
proper non-orthochronous motion can be obtained as a composition of this concrete motion and a proper
orthochronous motion in $\RR^4_2$. Therefore, if two maximal space-like surfaces are related by a proper 
non-orthochronous motion, then the functions in their Weierstrass representations are transformed by formulas
analogous to \eqref{hatfg_fg-orthchr_move_R42}. The only difference is that the matrices of the 
linear fractional transformations of $\hat g_1$ and $\hat g_2$ are anti-unitary with determinant equal to $-1$\,:
\begin{gather}\label{hatfg_fg-nonorthchr_move_R42}
\hat f = f(b_1 g_1 + \bar a_1)(b_2 g_2 + \bar a_2) \,; \quad 
\hat g_1 = \frac{a_1 g_1 + \bar b_1}{b_1 g_1 + \bar a_1}\,; \quad 
\hat g_2 = \frac{a_2 g_2 + \bar b_2}{b_2 g_2 + \bar a_2}\,;
\qquad \begin{array}{l} a_j,b_j \in \CC\,; \\ |a_j|^2-|b_j|^2=-1\,. \end{array}
\raisetag{-0.8ex}
\end{gather}

\smallskip

 Further we consider the case of an improper motion in $\RR^4_2$. 
Such a concrete motion is obtained changing only the sign in the fourth coordinate.
This means that in the Weierstrass representation \eqref{W_polinom_R42} $\phi_4$ changes the sign.
This transformation of $\Phi$ is obtained changing the places of $g_1$ and $g_2$. 
Any improper motion can be obtained by a composition of this concrete  motion and a proper 
motion in $\RR^4_2$. Hence the required formulas are obtained changing the places of $g_1$ and $g_2$
in the corresponding formulas for the proper motions. 
Thus for an improper non-orthochronous motion in $\RR^4_2$ we get from \eqref{hatfg_fg-orthchr_move_R42}:
\begin{gather}\label{hatfg_fg-unp_nonorthchr_move_R42}
\hat f = f(b_1 g_2 + \bar a_1)(b_2 g_1 + \bar a_2) \,; \quad 
\hat g_1 = \frac{a_1 g_2 + \bar b_1}{b_1 g_2 + \bar a_1}\,; \quad 
\hat g_2 = \frac{a_2 g_1 + \bar b_2}{b_2 g_1 + \bar a_2}\,;
\qquad \begin{array}{l} a_j,b_j \in \CC\,; \\ |a_j|^2-|b_j|^2=1\,. \end{array}
\raisetag{-0.8ex}
\end{gather}
Respectively, by an improper orthochronous motion in $\RR^4_2$ we have from \eqref{hatfg_fg-nonorthchr_move_R42}:
\begin{gather}\label{hatfg_fg-unp_orthchr_move_R42}
\hat f = f(b_1 g_2 + \bar a_1)(b_2 g_1 + \bar a_2) \,; \quad 
\hat g_1 = \frac{a_1 g_2 + \bar b_1}{b_1 g_2 + \bar a_1}\,; \quad 
\hat g_2 = \frac{a_2 g_1 + \bar b_2}{b_2 g_1 + \bar a_2}\,;
\qquad \begin{array}{l} a_j,b_j \in \CC\,; \\ |a_j|^2-|b_j|^2=-1\,. \end{array}
\raisetag{-0.8ex}
\end{gather}

\smallskip

 Now, let us find the change of the triples of functions under a homothety of the type $\hat\x=k\x$.
In this case $\hat\Phi=k\Phi$ and consequently the triple $(f,g_1,g_2)$ from the representation 
\eqref{W_polinom_R42} is changed as follows: 
\begin{equation}\label{hat_fg-fg-homotet_R42}
\hat f = k f\,; \qquad \hat g_1 = g_1\,; \qquad \hat g_2 = g_2 \,.
\end{equation}

\smallskip

  Finally, let us consider the one-parameter family $\M_\theta$ of associated maximal space-like surfaces 
to a given surface $\M$. According to \eqref{Phi_1-param_family} we have: $\Phi_\theta = \e^{-\ii\theta}\Phi$.
From here we obtain that the functions in the representation \eqref{W_polinom_R42} are changed as follows:
\begin{equation}\label{fg_1-param_family_R42}
f_\theta = \e^{-\ii\theta}f\,; \qquad g_{1|\theta} = g_1\,; \qquad g_{2|\theta} = g_2 \,.
\end{equation}


\section{Formulas for the basic invariants of a maximal space-like surface in $\RR^4_2$, given by a Weierstrass representation}
\label{sect_K_kappa-W}

 The purpose of this section is to find formulas for the coefficients of the first fundamental form 
and for the scalar invariants $K$ and $\varkappa$ of a given maximal space-like surface $\M$ via the complex functions
in the Weierstrass representation of $\M$. In order to simplify the calculations we use again the auxiliary vector
function $\va$ defined by equality \eqref{va-def_R42}.
We have:
\begin{equation}\label{va_vap-g_R42}
\begin{array}{lll}
\va       &=& \big(\, (g_1 g_2+1) \,,\, \phantom{-}\ii (g_1 g_2-1) \,,\, (g_1 + g_2) \,,\, \phantom{-}\ii (g_1 - g_2) \,\big)\,;\\
\bar\va   &=& \big(\, (\bar g_1 \bar g_2+1) \,,\, -\ii (\bar g_1 \bar g_2-1) \,,\, 
                      (\bar g_1 + \bar g_2) \,,\, -\ii (\bar g_1 - \bar g_2) \,\big)\,;\\
\va'      &=& \big(\, (g'_1 g_2+g_1 g'_2) \,,\, \phantom{-}\ii (g'_1 g_2+g_1 g'_2) \,,\, 
                      (g'_1 + g'_2) \,,\, \phantom{-}\ii (g'_1 - g'_2) \,\big)\,;\\
\bar{\va'}&=& \big(\, (\bar{g'_1} \bar g_2 + \bar g_1 \bar{g'_2}) \,,\, -\ii (\bar{g'_1} \bar g_2 + \bar g_1 \bar{g'_2}) \,,\, 
                      (\bar{g'_1} + \bar{g'_2}) \,,\, -\ii (\bar{g'_1} - \bar{g'_2}) \,\big)\,.
\end{array}
\end{equation}
We need the scalar products between $\va$, $\bar \va$, $\va'$ and  $\bar {\va'}$. 
After the definition of $\va$ we have already obtained that $\va^2=0$ and after differentiation and
complex conjugation we get:
\begin{equation}\label{a1_R42}
\va^2=\va \va'=\bar \va^2=\bar \va \bar {\va'}=0\,.
\end{equation}
For $\|\va\|^2$ we also have formula \eqref{n_va^2_g1g2_R42}. After a scalar multiplication of the equalities
\eqref{va_vap-g_R42} we find the following formulas:
\begin{equation}\label{a3_R42}
\begin{array}{rl}
\va\bar {\va'} 
           &= 2(\bar{g'_1} \bar g_2 + \bar g_1 \bar{g'_2}) g_1 g_2 - 2(g_1\bar{g'_1} + g_2\bar{g'_2})\\
					 &= 2 g_1\bar{g'_1}(|g_2|^2-1) + 2 g_2\bar{g'_2}(|g_1|^2-1).
\end{array}
\end{equation}
\begin{equation}\label{a4_R42}
 \bar \va \va' = \overline{\va\bar {\va'}} = 2 \bar g_1 g'_1 (|g_2|^2-1) + 2 \bar g_2 g'_2 (|g_1|^2-1).
\end{equation} 
\begin{equation}\label{a5_R42}
{\va'}^2 = -(g'_1 + g'_2)^2 + (g'_1 - g'_2)^2 = -4g'_1 g'_2 \,. 
\end{equation}
\begin{equation}\label{a6_R42}
\begin{array}{rl}
\|\va'\|^2 = \va'\bar {\va'} 
                       &= 2|g_2|^2 |g'_1|^2 + 2\bar g_1 g'_1 g_2 \bar{g'_2} + 2g_1 \bar{g'_1} \bar g_2 g'_2 + 2|g_1|^2 |g'_2|^2
											   -2|g'_1|^2 - 2|g'_2|^2\\
					             &= 2|g'_1|^2 (|g_2|^2-1) + 2|g'_2|^2 (|g_1|^2-1) + 2\bar g_1 g'_1 g_2 \bar{g'_2} + 2g_1 \bar{g'_1} \bar g_2 g'_2\,.
\end{array}
\end{equation}

 Next we  obtain formulas for ${\va^{\prime\bot}}^2$ and ${\|\va^{\prime\bot}\|}^2$, expressed by $g_1$ and $g_2$. 
The vector $\va$ has the same algebraic properties as $\Phi$. Therefore this vector satisfies identities
analogous to \eqref{mPhipn2} and \eqref{Phi_prim_bot^2-Phi_prim^2}:
\begin{equation}\label{a_prim_bot^2_R42}
{\va^{\prime \bot}}^2={\va^\prime}^2\,; \qquad
{\|\va^{\prime\bot}\|}^2 = \ds\frac{\|\va\|^2\|\va'\|^2-|\bar \va \cdot \va'|^2}{\|\va\|^2}\:.
\end{equation}
From here with the aid of \eqref{a5_R42} we obtain the required formula for ${\va^{\prime\bot}}^2$:
\begin{equation}\label{apn2_R42}
{\va^{\prime\bot}}^2 = {\va'}^2 = -4g'_1 g'_2\,.
\end{equation}
Denote by $k_1$ the numerator in the last formula \eqref{a_prim_bot^2_R42} for ${\|\va^{\prime\bot}\|}^2$. 
Applying formulas \eqref{n_va^2_g1g2_R42}, \eqref{a4_R42} and \eqref{a6_R42}, after simplification we 
obtain $k_1$:
\begin{equation}\label{k1_R42}
\begin{array}{rl}
k_1 &= \|\va\|^2\|\va'\|^2-|\bar \va \cdot \va'|^2\\
    &= -4\big(|g'_1|^2 (|g_2|^2-1)^2 + |g'_2|^2 (|g_1|^2-1)^2\big)\,.
\end{array}
\end{equation}
Denote by $k_2$ the determinant of the vectors $\va$, $\bar \va$, $\va'$ and $\bar {\va'}$  multiplied by $-1$.
Using formulas \eqref{va_vap-g_R42} after simplification we find $k_2$:
\begin{equation}\label{k2_R42}
\begin{array}{rl}
k_2 &= -\det(\va,\bar \va , \va' , \bar {\va'})\\
    &= -4\big(|g'_1|^2 (|g_2|^2-1)^2 - |g'_2|^2 (|g_1|^2-1)^2\big)\,.
\end{array}
\end{equation}

\smallskip

 Let $\M$ be given by a Weierstrass representation of the type \eqref{W_polinom_R42}. Using the above 
formulas for the vector function $\va$, we can easily express the invariants of $\M$ via 
$f$, $g_1$ and $g_2$. Combining \eqref{EG}, \eqref{Phi-va_R42} and \eqref{n_va^2_g1g2_R42} we find 
the coefficient $E$ of the first fundamental form: 
\begin{equation}\label{E_fg1g2_R42}
E=\frac{1}{2}\|\Phi\|^2=|f|^2(|g_1|^2-1)(|g_2|^2-1)\,.
\end{equation}
Hence, the first fundamental form of $\M$ is given by the equality:
\begin{equation}\label{I_fg1g2_R42}
\mathbf{I}=E\,(du^2 + dv^2) = |f|^2(|g_1|^2-1)(|g_2|^2-1)|dt|^2 \,.
\end{equation}

\smallskip

 Equality \eqref{Phi-va_R42} implies that:
\begin{equation}\label{Phipn-va_R42}
\Phi^{\prime \bot}=(f'\va+f\va')^\bot = f\va^{\prime \bot}\,; \qquad
{\Phi^{\prime \bot}}^2=f^2 {\va^{\prime \bot}}^2\,.
\end{equation}
In order to obtain ${\Phi^\prime}^2$, we apply consecutively 
\eqref{Phi_prim_bot^2-Phi_prim^2}, \eqref{Phipn-va_R42} and \eqref{apn2_R42}:
\begin{equation}\label{Phip^2_fg1g2_R42}
{\Phi^\prime}^2={\Phi^{\prime \bot}}^2=f^2{\va^{\prime \bot}}^2=f^2{\va^\prime}^2=-4f^2 g'_1 g'_2\,.
\end{equation}

\smallskip

  Now, consider formula \eqref{K_kappa_Phi_R42} for $K$. Expressing $\Phi^{\prime \bot}$ through \eqref{Phipn-va_R42},
	we get:
\[
K = \ds\frac{-4{\|\Phi^{\prime \bot}\|}^2}{\|\Phi\|^4} = \ds\frac{-4{\|f\va^{\prime \bot}\|}^2}{\|f\va\|^4}
= \ds\frac{-4{|f|^2\|\va^{\prime \bot}\|}^2}{|f|^4\|\va\|^4} = \ds\frac{-4{\|\va^{\prime \bot}\|}^2}{|f|^2\|\va\|^4}\,.
\]
Replacing ${\|\va^{\prime\bot}\|}^2$ by \eqref{a_prim_bot^2_R42} in the last formula and using the function 
$k_1$, defined by \eqref{k1_R42}, we find:
\[
K = \ds\frac{-4(\|\va\|^2\|\va'\|^2-|\bar \va \cdot \va'|^2)}{|f|^2\|\va\|^6} = \ds\frac{-4k_1}{|f|^2\|\va\|^6}\,.
\]
In order to obtain a similar formula for $\varkappa$, we use the second formula in \eqref{K_kappa_Phi_R42}. 
We express $\Phi$ in this formula by $f$ and $\va$, and using the function $k_2$ defined by \eqref{k2_R42}, we obtain:
\[
\begin{array}{rll}
\ds \varkappa &=& -\ds\frac{4}{\|\Phi\|^6}\det (\Phi,\bar\Phi,\Phi^\prime,\overline{\Phi^\prime})
               =  -\ds\frac{4}{\|f\va\|^6}\det (f\va,\bar f \bar \va,f'\va+f\va',\bar f' \bar \va + \bar f \bar {\va'})\\[6mm]
              &=& -\ds\frac{4}{|f|^6\|\va\|^6}\det (f\va,\bar f \bar \va,f\va',\bar f \bar {\va'}) 
               =  -\ds\frac{4|f|^4}{|f|^6\|\va\|^6}\det (\va,\bar \va,\va',\bar {\va'}) = \ds\frac{4k_2}{|f|^2\|\va\|^6}\,.
\end{array}
\]
Finally we have the following formulas for the curvatures $K$ and $\varkappa$:
\begin{equation}\label{K_kappa1_R42}
K = \ds\frac{-4k_1}{|f|^2\|\va\|^6}\,; \qquad \varkappa = \ds\frac{4k_2}{|f|^2\|\va\|^6}\,.
\end{equation}
We express in the last formulas $\|\va\|^2$, $k_1$ and $k_2$ respectively through \eqref{n_va^2_g1g2_R42}, 
\eqref{k1_R42} and \eqref{k2_R42}.
Thus, for any maximal space-like surface $\M$ given by a Weierstrass representation of the type 
\eqref{W_polinom_R42}, the Gauss curvature $K$ and the curvature of the normal connection $\varkappa$ are given by:
\begin{equation}\label{K_kappa_fg1g2_R42}
\begin{array}{llr}
K         &=& \ds\frac{ 2}{|f|^2(|g_1|^2-1)(|g_2|^2-1)}
              \left(\ds\frac{|g'_1|^2}{(|g_1|^2-1)^2}+\ds\frac{|g'_2|^2}{(|g_2|^2-1)^2}\right);\\[4ex]
\varkappa &=& \ds\frac{-2}{|f|^2(|g_1|^2-1)(|g_2|^2-1)}
              \left(\ds\frac{|g'_1|^2}{(|g_1|^2-1)^2}-\ds\frac{|g'_2|^2}{(|g_2|^2-1)^2}\right).
\end{array}
\end{equation}


\section{Canonical Weierstrass representation for maximal space-like surfaces of general type in $\RR^4_2$}\label{sect_W_can}

 In this section we consider Weierstrass representation of maximal space-like surfaces of general type
parametrized by canonical coordinates.

Let $\M$ be a maximal space-like surface in $\RR^4_2$ given by a Weierstrass representation of the type
\eqref{W_polinom_R42}. First we express the conditions for a point to be degenerate and the conditions 
for the coordinates to be canonical through the functions in the Weierstrass representation. 
According to Definition \ref{DegP-def} and \eqref{Phi_prim_bot^2-Phi_prim^2}, the degenerate points coincide
with the zeroes of ${\Phi^\prime}^2$. Then it follows from equality \eqref{Phip^2_fg1g2_R42} that 
the degenerate points are characterized by the condition $-4f^2 g'_1 g'_2 = 0$. Since  $f\neq 0$, then
the set of the degenerate points coincides with the zeroes of $g'_1 g'_2$. Hence we have:
\begin{prop}\label{DegP_fg1g2_R42}
Let $\M$ be a maximal space-like surface in $\RR^4_2$ given by a Weierstrass representation of the type
\eqref{W_polinom_R42}. Then the point $p \in \M$ is degenerate if and only if $g'_1(p)g'_2(p)=0$\,. 
\end{prop}
From here and Definition \ref{Min_Surf_Gen_Typ-def} we have:
\begin{prop}\label{Gen_Typ_fg1g2_R42}
Let $\M$ be a maximal space-like surface in $\RR^4_2$ given by a Weierstrass representation of the type
\eqref{W_polinom_R42}. Then the surface $\M$ is of general type if and only if $g'_1 g'_2 \neq 0$\,. 
\end{prop}

  Let $\M$ be a maximal space-like surface of general type in $\RR^4_2$. Theorem \ref{Can_Coord-exist}\, 
gives that the surface admits locally canonical coordinates. Suppose that $\phi_1 + \ii\phi_2 \neq 0$\,
on $\M$. Since the canonical coordinates are isothermal, then the surface $\M$ has a Weierstrass 
representation of the type \eqref{W_polinom_R42}. According to Definition \ref{Can1-def}, the function $\Phi$
satisfies the condition ${\Phi^{\prime \bot}}^2={\Phi^\prime}^2=-1$\,. Then, it follows from \eqref{Phip^2_fg1g2_R42} that, 
the functions $f$, $g_1$ and $g_2$ are related by the equality $-4f^2 g'_1 g'_2=-1$\,.Thus, we have:
\begin{equation}\label{fg1g2_Can1_R42}
f = \ds\frac{1}{2\sqrt{g'_1 g'_2}}\,.
\end{equation}

  Applying the last equality to the representation \eqref{W_polinom_R42} we obtain that any maximal space-like surface
of general type in $\RR^4_2$, satisfying the condition $\phi_1 + \ii\phi_2 \neq 0$\,, has a local Weierstrass 
representation of the type:
\begin{equation}\label{W_Can1_polinom_R42}
\Phi= \left(\,
              \ds\frac{1}{2}\;   \ds\frac {g_1 g_2+1}{\sqrt{g'_1 g'_2}}\,,\,
              \ds\frac{\ii}{2}\; \ds\frac {g_1 g_2-1}{\sqrt{g'_1 g'_2}}\,,\,
              \ds\frac{1}{2}\;   \ds\frac {g_1 + g_2}{\sqrt{g'_1 g'_2}}\,,\,
              \ds\frac{\ii}{2}\; \ds\frac {g_1 - g_2}{\sqrt{g'_1 g'_2}}\,\right),
\end{equation}
where the pair of functions $(g_1,g_2)$ satisfy the conditions:
\begin{equation}\label{W_Can1_polinom_cond_R42}
g'_1 g'_2 \neq 0\,; \  |g_1|>1\,; \  |g_2|>1 \qquad \text{or} \qquad  g'_1 g'_2 \neq 0\,; \  |g_1|<1\,; \  |g_2|<1 \,.
\end{equation}
We call this representation a \textbf{canonical Weierstrass representation} for maximal space-like surfaces 
of general type in $\RR^4_2$.

  Conversely, if $g_1$ and $g_2$ are holomorphic functions defined in a domain $\D$ in ${\CC}$, and satisfy
the conditions \eqref{W_Can1_polinom_cond_R42}, then formulas \eqref{W_Can1_polinom_R42} give a
maximal space-like surface $\M$ in $\RR^4_2$, parametrized by canonical coordinates.
We say that this surface is \textit{generated by the holomorphic functions $g_1$ and $g_2$}
by formulas \eqref{W_Can1_polinom_R42}. We note that $\Phi$ determines uniquely the functions $g_1$ and $g_2$ 
through equalities \eqref{W_Can1_polinom_R42}. Applying formulas \eqref{fg1g2_Phi_R42} to the representation
\eqref{W_Can1_polinom_R42} we find that $g_1$ and $g_2$ are expressed explicitly by $\Phi$:
\begin{equation}\label{g-Phi_Can1_R42}
g_1=\ds\frac{\phi_3-\ii\phi_4}{\phi_1+\ii\phi_2}\;; \qquad
g_2=\ds\frac{\phi_3+\ii\phi_4}{\phi_1+\ii\phi_2}\;.
\end{equation}

 Summarizing we have:
\begin{theorem}\label{thm-W_Can1_polinom_R42}
Let $\M$ be a maximal space-like surface in $\RR^4_2$ and the complex parameter $t=u+\ii v$ give canonical 
coordinates in a neighborhood of a point $p \in \M$. 
Suppose the function $\Phi$ satisfies the condition $\phi_1 + \ii\phi_2 \neq 0$\,.  Then there exists
a neighborhood of $p$, where the function $\Phi$ has a Weierstrass representation of the type 
\eqref{W_Can1_polinom_R42}, where $g_1$ and $g_2$ are holomorphic functions satisfying the conditions 
\eqref{W_Can1_polinom_cond_R42}. The functions $g_1$ and $g_2$ are determined uniquely by $\Phi$ via formulas
\eqref{g-Phi_Can1_R42}.

 Conversely, let $g_1$ and $g_2$ be holomorphic functions, defined in a domain $\D$ in ${\CC}$, 
satisfying conditions \eqref{W_Can1_polinom_cond_R42} and let $p \in \D$. Then there exists
a sub-domain $\D_0 \subset \D$, containing $p$ and a maximal space-like surface of general type $(\D_0,\x)$ 
in $\RR^4_2$ such that the corresponding function $\Phi$ has a representation \eqref{W_Can1_polinom_R42}
in $\D_0$ and satisfies the condition $\phi_1 + \ii\phi_2 \neq 0$\,. 
\end{theorem}
\begin{remark}
The restriction $\phi_1 + \ii\phi_2 \neq 0$ in the last theorem is not an essential geometric condition for $\M$.
If any maximal space-like surface of general type in $\RR^4_2$ satisfies the condition $\phi_1 + \ii\phi_2 = 0$,
then the surface can be transformed by a proper motion in $\RR^4_2$ into a surface with the property
$\phi_1 + \ii\phi_2 \neq 0$. Since all geometric properties which we shall obtain further 
are invariant under a proper motion in $\RR^4_2$, then they will be valid also for surfaces satisfying 
the equality $\phi_1 + \ii\phi_2 = 0$\,. 
\end{remark}
\begin{remark}
In the canonical Weierstrass representation \eqref{W_Can1_polinom_R42} there is a factor with a square root.
Consequently, choosing different branches of the square root, the formula gives two maximal space-like 
surfaces of general type in $\RR^4_2$, which differ from one another by the sign.  These two different 
surfaces are related by a proper orthochronous motion in $\RR^4_2$.
\end{remark}

 Now we can give an algorithm for a transition from a classical to canonical Weierstrass representation
of $\M$. Suppose that the maximal space-like surface of general type $\M$ is given by а classical representation
\eqref{W_polinom_R42}. Combining equalities \eqref{eqcan-sol} and \eqref{Phip^2_fg1g2_R42}, we find the 
following formula for the canonical coordinates $s$ on $\M$ :
\begin{equation}\label{eq_can1-sol_fg1g2_R42}
s = \int\sqrt{\vphantom{K^2} 4f^2(t) g'_1(t) g'_2(t)}\:dt\;.
\end{equation} 
From the last equality we express $t$ as a function of $s$. Then the new functions $\tilde g_1$ and 
$\tilde g_2$, giving the canonical representation of the type \eqref{W_Can1_polinom_R42} are obtained
according to \eqref{f_g_s-hol_R42}, from the old functions $g_1$ and $g_2$ trough the formulas 
$\tilde g_1(s) = g_1(t(s))$ and $\tilde g_2(s) = g_2(t(s))$.

\medskip

 Next we  obtain transformation formulas for the functions $g_1$ and $g_2$ in the canonical Weierstrass 
representation \eqref{W_Can1_polinom_R42} of a given maximal space-like surface of general type $\M$ in $\RR^4_2$, 
under a change of the canonical coordinates and under the basic geometric transformations of the surface in 
$\RR^4_2$. 

We begin with the case of a holomorphic change of the canonical coordinates. According to Theorem 
\ref{Can_Coord-uniq} such a change has the form $t=\delta s+c$, where $\delta = \pm 1;\, \pm\ii$, and 
$c = \rm {const}$. If $\tilde g_1(s)$ and $\tilde g_2(s)$ are the generating functions for the representation 
\eqref{W_Can1_polinom_R42} in the new coordinates, then according to \eqref{f_g_s-hol_R42}, these functions 
are expressed by the old functions $g_1$ and $g_2$ in the following way:
\begin{equation}\label{g_s-hol_R42}
\tilde g_1(s) = g_1(\delta s+c)\,; \qquad \tilde g_2(s) = g_2(\delta s+c)\,.
\end{equation}

 Consider also the case of an anti-holomorphic change of the canonical coordinates. According to Theorem
\ref{Can_Coord-uniq}, such a change is of the type $t=\delta \bar s + c$, where $\delta = \pm 1;\, \pm\ii$,
and $c = \rm{const}$. Applying equalities \eqref{f_g_s-hol_R42} and \eqref{f_g_s-t_bs_R42} to the representation 
\eqref{W_Can1_polinom_R42}, we get the following relation between the generating functions:
\begin{equation}\label{g_s-antihol_R42}
\tilde g_1(s) = \frac{1}{\bar g_1(\delta \bar s + c)}\,; \qquad \tilde g_2(s) = \frac{1}{\bar g_2(\delta \bar s + c)}\,.
\end{equation}

\smallskip

 Further, we consider the case of motion in $\RR^4_2$. Let the maximal space-like surface of general type 
$\hat\M=(\D,\hat\x)$ be obtained from $\M=(\D,\x)$ through a proper orthochronous motion in $\RR^4_2$ of the 
type $\hat\x(t)=A\x(t)+\vb$, where $A \in \mathbf{SO}^+(2,2,\RR)$ and $\vb \in \RR^4_2$. If $t$ gives canonical 
coordinates on $\M$, then according to Theorem \ref{Can_Move}, these coordinates are canonical also for $\hat\M$. 
Then it follows from equalities \eqref{hatfg_fg-orthchr_move_R42} that the functions $\hat g_1$, $\hat g_2$, 
and $g_1$, $g_2$ from the canonical representations \eqref{W_Can1_polinom_R42} are related by linear fractional
transformations with special unitary matrices from $\mathbf{SU}(1,1,\CC)$:
\begin{equation}\label{hatg_g-orthchr_move_Can1_R42}
\hat g_1 = \frac{a_1 g_1 + \bar b_1}{b_1 g_1 + \bar a_1}\,; \quad 
\hat g_2 = \frac{a_2 g_2 + \bar b_2}{b_2 g_2 + \bar a_2}\,;
\qquad \begin{array}{l} a_j,b_j \in \CC\,; \\ |a_j|^2-|b_j|^2=1\,. \end{array}
\end{equation}

 As in the case of the general representation \eqref{W_polinom_R42} it follows that the inverse statement 
is also valid: If the functions $\hat g_1$, $\hat g_2$ and $g_1$, $g_2$ in the canonical representations 
\eqref{W_Can1_polinom_R42} satisfy \eqref{hatg_g-orthchr_move_Can1_R42}, then the generated surfaces are
related by a proper orthochronous motion in $\RR^4_2$. Thus we obtained similarly to Theorem
\ref{W-orthchr_move_R42} that the following statement is valid:
\begin{theorem}\label{W-orthchr_move_Can1_R42} 
Let $\hat\M=(\D,\hat\x)$ and $\M=(\D,\x)$ be two maximal space-like surfaces of general type in $\RR^4_2$, 
given by Weierstrass representations of the type \eqref{W_Can1_polinom_R42}, where $\D$ is a connected 
domain in $\CC$. Then the following conditions are equivalent:
\begin{enumerate}
	\item $\hat\M$ and $\M$ are related by a proper orthochronous motion in $\RR^4_2$ of the type:\\
	$\hat\x(t)=A\x(t)+\vb$, where $A \in \mathbf{SO}^+(2,2,\RR)$ and $\vb \in \RR^4_2$.
	\item The functions $\hat g_1$, $\hat g_2$ and $g_1$, $g_2$ form the canonical Weierstrass representations
	\eqref{W_Can1_polinom_R42} of $\hat\M$ and $\M$ are related by the equalities	\eqref{hatg_g-orthchr_move_Can1_R42}. 
\end{enumerate}
\end{theorem}

 The cases of a motion in  $\RR^4_2$, belonging to one of the remaining three connected components of 
$\mathbf{O}(2,2,\RR)$ are considered in a similar way. 

If $\hat\M$ is obtained from $\M$ through a proper non-orthochronous motion, then similarly to 
\eqref{hatfg_fg-nonorthchr_move_R42}, we get the following transformation formulas:
\begin{equation}\label{hatg_g-nonorthchr_move_Can1_R42}
\hat g_1 = \frac{a_1 g_1 + \bar b_1}{b_1 g_1 + \bar a_1}\,; \quad 
\hat g_2 = \frac{a_2 g_2 + \bar b_2}{b_2 g_2 + \bar a_2}\,;
\qquad \begin{array}{l} a_j,b_j \in \CC\,; \\ |a_j|^2-|b_j|^2=-1\,. \end{array}
\end{equation}

 Under an improper non-orthochronous motion in $\RR^4_2$, similarly to \eqref{hatfg_fg-unp_nonorthchr_move_R42} we have:
\begin{equation}\label{hatg_g-unp_nonorthchr_move_Can1_R42}
\hat g_1 = \frac{a_1 g_2 + \bar b_1}{b_1 g_2 + \bar a_1}\,; \quad 
\hat g_2 = \frac{a_2 g_1 + \bar b_2}{b_2 g_1 + \bar a_2}\,;
\qquad \begin{array}{l} a_j,b_j \in \CC\,; \\ |a_j|^2-|b_j|^2=1\,. \end{array}
\end{equation}

 Finally, under an improper orthochronous motion in $\RR^4_2$, similarly to \eqref{hatfg_fg-unp_orthchr_move_R42}
we have:
\begin{equation}\label{hatg_g-unp_orthchr_move_Can1_R42}
\hat g_1 = \frac{a_1 g_2 + \bar b_1}{b_1 g_2 + \bar a_1}\,; \quad 
\hat g_2 = \frac{a_2 g_1 + \bar b_2}{b_2 g_1 + \bar a_2}\,;
\qquad \begin{array}{l} a_j,b_j \in \CC\,; \\ |a_j|^2-|b_j|^2=-1\,. \end{array}
\end{equation}

\smallskip

 Next we find the transformation formulas for the generating functions in the canonical Weierstrass 
representations in the case when $\hat\M$ is obtained from $\M$ by a homothety with coefficient $k$.
If $t$ gives canonical coordinates on $\M$, then according to Theorem \ref{Can_Sim}, canonical coordinates
on $\hat\M$ are obtained by the equality $t=\frac{1}{\sqrt{k}}s$. Combining equalities
\eqref{hat_fg-fg-homotet_R42} and \eqref{f_g_s-hol_R42} we obtain that the transformation formulas for the functions
in the representation \eqref{W_Can1_polinom_R42} of $\hat\M$ and $\M$, have the form:
\begin{equation}\label{hat_g-g-homotet_Can1_R42}
\hat g_1(s) = g_1\left(\frac{1}{\sqrt{k}}s\right); \qquad \hat g_2(s) = g_2\left(\frac{1}{\sqrt{k}}s\right).
\end{equation}

 Finally we consider the family $\M_\theta$ of associated maximal space-like surfaces to a given surface $\M$.
If $t$ gives canonical coordinates on $\M$, then according to \ref{Can_1-param_family}, canonical coordinates
on $\M_\theta$ are obtained by the equality $t=\e^{\ii\frac{\theta}{2}} s$. Combining \eqref{fg_1-param_family_R42} and
\eqref{f_g_s-hol_R42} we find that the transformation formulas for the functions in the representations
\eqref{W_Can1_polinom_R42}, have the form:
\begin{equation}\label{g_1-param_family_Can1_R42}
g_{1|\theta} (s) = g_1(\e^{\ii\frac{\theta}{2}} s) \,; \qquad g_{2|\theta} (s) = g_2(\e^{\ii\frac{\theta}{2}} s) \,.
\end{equation}

 Now we can easily obtain the relation between the canonical representations of a maximal space-like surface of a 
general type $\M$ in $\RR^4_2$ and its conjugate surface $\bar\M$, which was introduced by Definition
\ref{Conj_Min_Surf}\,. As we noted  after Definition \ref{1-param_family_assoc_surf}\,, the conjugate maximal 
space-like surface $\bar\M$ of $\M$ belongs to the one-parameter family of associated maximal space-like surfaces
to $\M$ and is obtained by $\theta=\frac{\pi}{2}$. Consequently, the corresponding formulas for $\bar\M$ follow 
from the formulas for $\M_\theta$, replacing $\theta$ with $\frac{\pi}{2}$.

\smallskip

  If two maximal space-like surfaces of general type in $\RR^4_2$, are obtained one from the other by a proper motion,
homothety and a change of the canonical coordinates, then they can be considered identical from a geometric 
point of view. That is why, as in the case of $K$ and $\varkappa$, we unify the transformation formulas into one
in the following way: 
\begin{prop}\label{Thm-g1g2_IdentSurf_R42}
Let $\M$ and $\hat\M$ be maximal space-like surfaces of general type in $\RR^4_2$ and $p_0\in\M$, $\hat p_0\in\hat\M$ 
be fixed points. Let $t$ give canonical coordinates on $\M$ in a neighborhood of $p_0$ and $s$ give 
canonical coordinates on $\hat\M$ in a neighborhood of $\hat p_0$. Then the following conditions are equivalent:
\begin{enumerate}
	\item There exists a neighborhood of $\hat p_0$ in $\hat\M$, which is obtained from the corresponding
	neighborhood of $p_0$ in $\M$ by a proper motion, homothety and a change of the canonical coordinates.
	\item There exists a neighborhood of $\hat p_0$ in $\hat\M$, in which the functions $\hat g_1(s)$ and $\hat g_2(s)$
	in the canonical representation \eqref{W_Can1_polinom_R42} of $\hat\M$ are obtained from the corresponding functions
	$g_1(t)$ and $g_2(t)$ of $\M$, by formulas of the following type: 
  \begin{gather}\label{hat_g_g-Ident_R42}
  \hat g_1(s) = \frac{a_1 \tilde g_1(\delta a \tilde s + b) + \bar b_1}{b_1 \tilde g_1(\delta a \tilde s + b) + \bar a_1}\,; \quad
	\hat g_2(s) = \frac{a_2 \tilde g_2(\delta a \tilde s + b) + \bar b_2}{b_2 \tilde g_2(\delta a \tilde s + b) + \bar a_2}\,; \qquad 
	\begin{array}{l} a_j,b_j \in \CC\,; \\ |a_j|^2-|b_j|^2=\pm 1\,. \end{array}
	\raisetag{-0.9ex}
  \end{gather}
  where $\delta=\pm 1;\pm\ii$,\; $\tilde g(\tilde s)$ denotes $g(s)$ or $\bar g(\bar s)$, and  $a>0$, $b\in\CC$ are constants.
	\end{enumerate}
\end{prop}
\begin{proof}
Formulas \eqref{hat_g_g-Ident_R42} are obtained by a composition of the formulas 
\eqref{hatg_g-orthchr_move_Can1_R42}, \eqref{hatg_g-nonorthchr_move_Can1_R42}, 
\eqref{hat_g-g-homotet_Can1_R42}, \eqref{g_s-hol_R42} and \eqref{g_s-antihol_R42} and a change of the denotations 
of the constants, from where it follows the required equivalence. 
\end{proof}


\section{Explicit solving of the system of natural equations of the maximal space-like surfaces in $\RR^4_2$}\label{sect_sol_nat_eq}

Our first task in this section is to express the basic invariants of a given maximal space-like surface of 
general type in $\RR^4_2$, through the complex functions $g_j$, involved in the canonical Weierstrass representation
\eqref{W_Can1_polinom_R42}. For this purpose we use the corresponding general formulas from Section \ref{sect_K_kappa-W}\,, 
as well as the condition \eqref{fg1g2_Can1_R42} for the coordinates to be canonical.

 Let $\M$ be a maximal space-like surface of general type in $\RR^4_2$, parametrized by canonical coordinates
$t=u+\ii v$ and consider the canonical Weierstrass representation of the type \eqref{W_Can1_polinom_R42}. 
First we find the coefficient $E$ of the first fundamental form of $\M$. In the general formula \eqref{E_fg1g2_R42} 
we replace $f$ with $\ds\frac{1}{2\sqrt{g'_1 g'_2}}$, according to the condition \eqref{fg1g2_Can1_R42}. 
Thus we find:
\begin{equation}\label{E_g1g2_Can1_R42}
E=\frac{(|g_1|^2-1)(|g_2|^2-1)}{4|g'_1 g'_2|}\;.
\end{equation}
Then the first fundamental form of $\M$ is as follows:
\begin{equation}\label{I_g1g2_Can1_R42}
\mathbf{I}=E\,(du^2 + dv^2) = \frac{(|g_1|^2-1)(|g_2|^2-1)}{4|g'_1 g'_2|}|dt|^2.
\end{equation}

\smallskip

 Next we find the Gauss curvature $K$ and the curvature of the normal connection $\varkappa$. Replacing
$f$ with $\ds\frac{1}{2\sqrt{g'_1 g'_2}}$ in \eqref{K_kappa_fg1g2_R42}, we get:
\begin{equation}\label{K_kappa_g1g2_Can1_R42}
\begin{array}{llr}
K         &=& \ds\frac{ 8|g'_1 g'_2|}{(|g_1|^2-1)(|g_2|^2-1)}
              \left(\ds\frac{|g'_1|^2}{(|g_1|^2-1)^2}+\ds\frac{|g'_2|^2}{(|g_2|^2-1)^2}\right);\\[3ex]
\varkappa &=& \ds\frac{-8|g'_1 g'_2|}{(|g_1|^2-1)(|g_2|^2-1)}
              \left(\ds\frac{|g'_1|^2}{(|g_1|^2-1)^2}-\ds\frac{|g'_2|^2}{(|g_2|^2-1)^2}\right).
\end{array}
\end{equation}

\smallskip

 Finally we find the normal curvatures $\nu$ and $\mu$, defined by \eqref{sigma_nu_mu_R42}.
If the maximal space-like surface $\M$ is given by \eqref{W_Can1_polinom_R42}, 
taking into account the relation \eqref{nu_mu_K_kappa_R42} between $(\nu,\mu)$ and $(K,\varkappa)$ 
and \eqref{K_kappa_g1g2_Can1_R42} we have: 
\begin{equation}\label{nu_mu_g1g2_Can1_R42}
\begin{array}{llr}
\nu &=& 2 \ \sqrt{\ds\frac{|g'_1 g'_2|}{(|g_1|^2-1)(|g_2|^2-1)}}
              \left(\ds\frac{|g'_1|}{|\,|g_1|^2-1\,|}+\ds\frac{|g'_2|}{|\,|g_2|^2-1\,|}\right);\\[3.5ex]
\mu &=& -2 \ \sqrt{\ds\frac{|g'_1 g'_2|}{(|g_1|^2-1)(|g_2|^2-1)}}
              \left(\ds\frac{|g'_1|}{|\,|g_1|^2-1\,|}-\ds\frac{|g'_2|}{|\,|g_2|^2-1\,|}\right).
\end{array}
\end{equation}

\medskip 

 As an immediate application of these formulas, we shall characterize the maximal space-like surfaces $\M$ of general type
lying entirely in a hyperplane in $\RR^4_2$. First we consider the case when the three-dimensional plane is determined by
the equality $x_4=k$, where $k = \rm{const}$. This is equivalent to the equality $\phi_4=0$, where $\phi_4$ is the fourth 
coordinate of $\Phi$. Then the fourth formula in the representation \eqref{W_Can1_polinom_R42} implies that $g_1=g_2$. 
Next, let the hyperplane be determined by the condition $\x\cdot\n=k$, where $\n \in \RR^4_2$. Then the vector $\n$ 
is time-like because in this case only, the orthogonal complement $\{\n\}^\bot$ contains two-dimensional planes with positive definite metric.
Any hyperplane in $\RR^4_2$ of the type $\x\cdot\n=k$, where $\n$ is time-like, can be transformed into 
the hyperplane $x_4=k$ by a proper orthochronous motion. Under this motion, the pair of functions are changed 
through the formulas \eqref{hatg_g-orthchr_move_Can1_R42}. Hence, the surface $\M$ is contained in a hyperplane of
$\RR^4_2$, if and only if:
\begin{equation*}
\frac{a_1 g_1 + \bar b_1}{b_1 g_1 + \bar a_1}=\frac{a_2 g_2 + \bar b_2}{b_2 g_2 + \bar a_2}\,;
\qquad \begin{array}{l} a_j,b_j \in \CC\,; \\ |a_j|^2-|b_j|^2=1\,. \end{array}
\end{equation*}
The last condition is equivalent to:
\begin{equation}\label{g1g2_3dim_subsp_R42}
g_2=\frac{a g_1 + \bar b}{b g_1 + \bar a}\,;
\qquad \begin{array}{l} a,b \in \CC\,; \\ |a|^2-|b|^2=1\,. \end{array}
\end{equation}
Now, the results for the maximal space-like surfaces of general type in $\RR^3_1$, obtained in \cite{G-K-K}, and
\eqref{W_Can1_polinom_cond_R42} give that \eqref{g1g2_3dim_subsp_R42} is equivalent to:
\begin{equation*}
\ds\frac{4|g'_1|^2}{(|g_1|^2-1)^2}=\ds\frac{4|g'_2|^2}{(|g_2|^2-1)^2}\,.
\end{equation*}
It follows from the formula for $\varkappa$ in \eqref{K_kappa_g1g2_Can1_R42} that the last equality is 
equivalent to $\varkappa=0$\,. Further, in view of \eqref{nu_mu_g1g2_Can1_R42}, this is equivalent to $\mu=0$\,. 
Summarizing, we have:
\begin{theorem}\label{Thm-Min_Surf_in_hyperspace_R42}
Let $\M=(\D,\x)$ be a maximal space-like surface of general type in $\RR^4_2$, given by the canonical Weierstrass 
representation \eqref{W_Can1_polinom_R42}, where $\D$ is a connected domain in $\CC$.  
Then the following conditions are equivalent:
\begin{enumerate}
	\item $\M$ lies entirely in a hyperplane of $\RR^4_2$.
	\item The functions $g_1$ and $g_2$ in the canonical Weierstrass representation \eqref{W_Can1_polinom_R42} of $\M$
	are related by a linear fractional transformation with special unitary matrix over
  $\CC^2_1$ of the type \eqref{g1g2_3dim_subsp_R42}.
	\item The curvature of the normal connection is zero:\ $\varkappa=0$\,. 
	\item The normal curvature $\mu$ is zero: \ $\mu=0$\,. 
\end{enumerate}
\end{theorem}
\begin{proof}
The above arguments before the theorem show that the four conditions in the theorem are equivalent locally.
Taking into account that the domain $\D$ is connected, standard topological arguments give that the conditions
are equivalent globally. 
\end{proof}
\begin{remark}
We can obtain a proof of the above theorem without reference to the theory of the maximal space-like surfaces in $\RR^3_1$.
The equivalence 1. $\Leftrightarrow$ 4. follow from the fourth equation 
in \eqref{Frene_Phi_bar_Phi_n1_n2_R42} and \eqref{beta_nu_mu_R42}.
While the equivalence 3. $\Leftrightarrow$ 4. follow from \eqref{K_kappa_nu_mu_R42} and $\nu>0$\,.
\end{remark}

 Let us return to the system of natural equations \eqref{Nat_Eq_nu_mu_R42} of maximal space-like surfaces
in $\RR^4_2$. Next we show that the formulas for $\nu$ and $\mu$ can be interpreted as local formulas 
giving the general solution to this system. First we show that \eqref{nu_mu_g1g2_Can1_R42} gives locally
all solutions of the system of natural equations through pairs of holomorphic functions.
\begin{theorem}\label{Nat_Eq_nu_mu_solv_g1g2_R42}
Let $(\nu>0,\mu)$ be a pair of holomorphic functions, defined in a domain $\D\subset\RR^2\equiv\CC$ and
let $(\nu,\mu)$ be a solution to the system of natural equations \eqref{Nat_Eq_nu_mu_R42}. 
For any $p_0\in\D$ there exists a neighborhood $\D_0\subset\D$ of $p_0$ and a pair of holomorphic 
functions $(g_1,g_2)$, defined in $\D_0$ and satisfying the conditions $g'_1g'_2\ne 0$ and 
$(|g_1|^2-1)(|g_2|^2-1)>0$\,, such that the solution $(\nu,\mu)$ is given by formulas
\eqref{nu_mu_g1g2_Can1_R42} in the domain $\D_0$.

Conversely, if $(g_1,g_2)$ is a pair of holomorphic functions, defined in a domain $\D\subset\CC$ and satisfying
the conditions $g'_1g'_2\ne 0$ and $(|g_1|^2-1)(|g_2|^2-1)>0$\,, then formulas \eqref{nu_mu_g1g2_Can1_R42}
give a solution $(\nu>0,\mu)$ to the system of natural equations \eqref{Nat_Eq_nu_mu_R42}.
\end{theorem}
\begin{proof}
Let $(\nu>0,\mu)$ be a solution to the system of natural equations in a given domain $\D$  and let
$p_0$ be a point in $\D$. Applying Theorem \ref{Bone_Phi_bar_Phi_n1_n2_nu_mu_R42} to $(\nu,\mu)$, 
we get a maximal space-like surface of general type $(\D_1,\x)$ in $\RR^4_2$, given by canonical coordinates
in a neighborhood $\D_1\subset\D$ of $p_0$. Furthermore, the given functions $\nu$ and $\mu$ are the normal curvatures 
of this surface. Applying Theorem \ref{thm-W_Can1_polinom_R42}\, to this surface, we have that the surface 
has a canonical Weierstrass representation of the type \eqref{W_Can1_polinom_R42} in a neighborhood $\D_0\subset\D_1$ of
$p_0$. Then the normal curvatures $\nu$ and $\mu$ of the surface satisfy \eqref{nu_mu_g1g2_Can1_R42}.

 Let now $(g_1,g_2)$ be two holomorphic functions, defined in a domain $\D\subset\CC$, satisfying te conditions
$g'_1g'_2\ne 0$ and $(|g_1|^2-1)(|g_2|^2-1)>0$ and let $p_0$ be a point in $\D$. According to Theorem
\ref{thm-W_Can1_polinom_R42}\,, the pair $(g_1,g_2)$ generates through formulas \eqref{W_Can1_polinom_R42}
a maximal space-like surface of general type $(\D_1,\x)$ in $\RR^4_2$, where $\D_1\subset\D$ is a neighborhood of 
$p_0$. The normal curvatures $\nu$ and $\mu$ of this surface satisfy \eqref{nu_mu_g1g2_Can1_R42}. Since 
the coordinates are canonical, then the normal curvatures $\nu$ and $\mu$ are a solution to the system
of natural equations \eqref{Nat_Eq_nu_mu_R42}. This means that formulas \eqref{nu_mu_g1g2_Can1_R42} give
a solution $(\nu>0,\mu)$ of the system of natural equations \eqref{Nat_Eq_nu_mu_R42} in a neighborhood of an 
arbitrary point in $\D$. This implies that \eqref{nu_mu_g1g2_Can1_R42} is a solution in the whole domain $\D$.
\end{proof}

 Further we consider the following natural question: \emph{When two different pairs of holomorphic functions generate 
through formulas \eqref{nu_mu_g1g2_Can1_R42} one and the same solution to the system of natural equations
\eqref{Nat_Eq_nu_mu_R42}?}

The answer is given by the following:
\begin{theorem}\label{Nat_eq_same_nu_mu_hatg_g_R42}
Let $(g_1,g_2)$ and $(\hat g_1,\hat g_2)$ be two pairs of holomorphic functions, defined in a connected domain $\D\subset\CC$, 
satisfying the conditions $g'_1g'_2\ne 0$\,, $(|g_1|^2-1)(|g_2|^2-1)>0$\,,
 $\hat g'_1 \hat g'_2\ne 0$ and $(|\hat g_1|^2-1)(|\hat g_2|^2-1)>0$\,.
The two pairs generate through formulas \eqref{nu_mu_g1g2_Can1_R42} one and the same solution to the system of
natural equations \eqref{Nat_Eq_nu_mu_R42} if and only if they are related by linear fractional transformations
of the type \eqref{hatg_g-orthchr_move_Can1_R42} or \eqref{hatg_g-nonorthchr_move_Can1_R42}.
\end{theorem}
\begin{proof}
Let $(\D,\x)$ and $(\D,\hat\x)$ denote the maximal space-like surfaces, generated by $(g_1,g_2)$ and
$(\hat g_1,\hat g_2)$, respectively by canonical Weierstrass representation of the type \eqref{W_Can1_polinom_R42}
and let $(\nu,\mu)$ and $(\hat\nu,\hat\mu)$ be their normal curvatures. First we suppose that $(g_1,g_2)$ and
$(\hat g_1,\hat g_2)$ are related by linear fractional transformations of the type \eqref{hatg_g-orthchr_move_Can1_R42} 
or \eqref{hatg_g-nonorthchr_move_Can1_R42}. Then both surfaces are obtained from each other through a proper motion
in $\RR^4_2$. From here, applying Theorem \ref{Thm-Nat_Eq_nu_mu_R42} we find $\hat\nu=\nu$ and $\hat\mu=\mu$,
which means that $(g_1,g_2)$ and $(\hat g_1,\hat g_2)$ generate one and the same solution.
For the inverse, suppose that both pairs generate one and the same solution to the system of natural equations.
According to Theorem \ref{Bone_Phi_bar_Phi_n1_n2_nu_mu_R42} we have that the two surfaces locally in a 
neighborhood of any point in $\D$ are obtained from each other by a proper motion in $\RR^4_2$.
It follows from here that in a neighborhood of any point in $\D$, $(g_1,g_2)$ and $(\hat g_1,\hat g_2)$ are
related by linear fractional transformations of the type \eqref{hatg_g-orthchr_move_Can1_R42} or
\eqref{hatg_g-nonorthchr_move_Can1_R42}. Since the domain $\D$ is connected, it follows in a standard way,
that the linear fractional transformations are one and the same in the whole domain $\D$. 
\end{proof}

 Consider the system of natural equations \eqref{Nat_Eq_K_kappa_R42} with respect to the pair $(K,\varkappa)$,
which is equivalent to \eqref{Nat_Eq_nu_mu_R42} through the equalities \eqref{K_kappa_nu_mu_R42} and
\eqref{nu_mu_K_kappa_R42}. On the other hand, formulas \eqref{K_kappa_g1g2_Can1_R42} and \eqref{nu_mu_g1g2_Can1_R42} 
are also equivalent through \eqref{K_kappa_nu_mu_R42} and \eqref{nu_mu_K_kappa_R42}. Consequently, for the
system \eqref{Nat_Eq_K_kappa_R42}, similarly to Theorem \ref{Nat_Eq_nu_mu_solv_g1g2_R42}, we have:
\begin{theorem}\label{Nat_Eq_K_kappa_solv_g1g2_R42}
Let $(K>0,\varkappa)$ be a pair of functions, defined in a domain $\D\subset\RR^2\equiv\CC$ and let $(K,\varkappa)$ 
be a solution to the system of natural equations \eqref{Nat_Eq_K_kappa_R42}.
For any point $p_0\in\D$ there exists a neighborhood $\D_0\subset\D$ of $p_0$ and a pair of holomorphic functions 
$(g_1,g_2)$, defined in $\D_0$, satisfying the conditions $g'_1g'_2\ne 0$ and $(|g_1|^2-1)(|g_2|^2-1)>0$\,, 
such that the solution $(K,\varkappa)$ is given by formulas \eqref{K_kappa_g1g2_Can1_R42} in the domain $\D_0$.

 Conversely, if $(g_1,g_2)$ is a pair of functions, defined in a domain $\D\subset\CC$ and satisfying
the conditions $g'_1g'_2\ne 0$, $(|g_1|^2-1)(|g_2|^2-1)>0$\,, then formulas \eqref{K_kappa_g1g2_Can1_R42}
give a solution $(K>0,\varkappa)$ to the system of natural equations \eqref{Nat_Eq_K_kappa_R42}.
\end{theorem}

 The analogue of the Theorem \ref{Nat_eq_same_nu_mu_hatg_g_R42} is the following:
\begin{theorem}\label{Nat_eq_same_K_kappa_hatg_g_R42}
Let $(g_1,g_2)$ and $(\hat g_1,\hat g_2)$ be two pairs of holomorphic functions, defined in a connected domain
$\D\subset\CC$, satisfying the conditions $g'_1g'_2\ne 0$\,, $(|g_1|^2-1)(|g_2|^2-1)>0$\,,
$\hat g'_1 \hat g'_2\ne 0$ and $(|\hat g_1|^2-1)(|\hat g_2|^2-1)>0$\,.
Both pairs generate through the equalities \eqref{K_kappa_g1g2_Can1_R42} one and the same solution
to the system of natural equations \eqref{Nat_Eq_K_kappa_R42} if and only if the pairs are related by linear
fractional transformations of the type \eqref{hatg_g-orthchr_move_Can1_R42} or \eqref{hatg_g-nonorthchr_move_Can1_R42}.
\end{theorem}

\medskip

 From the theorems proved in this section, we can derive similarly to the spaces $\RR^4$ and $\RR^4_1$, 
natural correspondences between three types of classes of objects: classes of maximal space-like surfaces of general type
in $\RR^4_2$, classes of solutions to the system of natural equations of the maximal space-like surfaces and 
classes of pairs of holomorphic functions in $\CC$. These correspondences can be determined only locally as in 
$\RR^4$ and $\RR^4_1$. In order to obtain these correspondences we need some additional definitions and denotations.
 
 Consider the set of maximal space-like surfaces of general type in $\RR^4_2$ of the type $(\D,\x)$,
where $\D$ is a neighborhood of zero in $\RR^2\equiv\CC$, in which the surface is parametrized with canonical
coordinates. Two surfaces $(\D,\x)$ and $(\hat\D,\hat\x)$ are said to be \emph{equivalent}, if there exists
a neighborhood of zero $\D_0\subset\D\cap\hat\D$, such that both surfaces are related in $\D_0$ by a proper motion
in $\RR^4_2$ of the type \eqref{hat_M-M-prop_mov_R42}. The set of the equivalence classes of this relation
is denoted by $\mathbf{MS}_{\RR^4_2}$.

 Consider the set of solutions of the system of natural equations \eqref{Nat_Eq_K_kappa_R42} of maximal 
space-like surfaces in $\RR^4_2$ of the type $(\D,K>0,\varkappa)$, where the functions $K$ and $\varkappa$ 
are defined in a neighborhood of zero $\D$ in $\RR^2\equiv\CC$. Two solutions $(\D,K,\varkappa)$ and
$(\hat\D,\hat K,\hat\varkappa)$ are said to be \emph{equivalent}, if there exists a neighborhood of zero
$\D_0\subset\D\cap\hat\D$, such that both solutions coincide in $\D_0$. The set of the equivalence classes 
of this relation is denoted by $\mathbf{SNE}_{\RR^4_2}$.

 Consider the set of pairs of holomorphic functions of the type $(\D,g_1,g_2)$, where the functions $g_1$ and $g_2$
are defined in a neighborhood $\D$ of zero in $\RR^2\equiv\CC$, and satisfy the conditions $g'_1g'_2\ne 0$, 
$(|g_1|^2-1)(|g_2|^2-1)>0$\,. Two pairs of functions $(\D,g_1,g_2)$ and $(\hat\D,\hat g_1,\hat g_2)$ are 
said to be \emph{equivalent}, if there exists a neighborhood of zero $\D_0\subset\D\cap\hat\D$, such that both
pairs of functions are related in $\D_0$ by  linear fractional transformations of the type
\eqref{hatg_g-orthchr_move_Can1_R42} or \eqref{hatg_g-nonorthchr_move_Can1_R42}. The set of the equivalence classes 
of this relation is denoted by $\mathbf{H}_{\RR^4_2}$.

 Let $\M$ be a maximal space-like surface of general type in $\RR^4_2$, parametrized by canonical coordinates. 
Associate to $\M$ its Gauss curvature $K$ and the curvature of the normal connection $\varkappa$. It follows
from Theorem \ref{Thm-Nat_Eq_K_kappa_R42} that this correspondence induces a map from $\mathbf{MS}_{\RR^4_2}$ 
into $\mathbf{SNE}_{\RR^4_2}$. What is more, Theorem \ref{Thm-Nat_Eq_K_kappa_R42} and Theorem
\ref{Bone_Phi_bar_Phi_n1_n2_K_kappa_R42} imply that this map is a bijection. 

Let $(g_1,g_2)$ be a pair of holomorphic functions, satisfying the conditions $g'_1g'_2\ne 0$ and 
$(|g_1|^2-1)(|g_2|^2-1)>0$\,. Associate to $(g_1,g_2)$ the maximal space-like surface of general type in
$\RR^4_2$, given by the canonical Weierstrass representation \eqref{W_Can1_polinom_R42}.
Then, from Theorem \ref{thm-W_Can1_polinom_R42} and equalities \eqref{hatg_g-orthchr_move_Can1_R42} and
\eqref{hatg_g-nonorthchr_move_Can1_R42} we have that this correspondence induces a map from $\mathbf{H}_{\RR^4_2}$ 
into $\mathbf{MS}_{\RR^4_2}$, which is bijection.

Let $(g_1,g_2)$ be a pair of holomorphic functions, satisfying the conditions $g'_1g'_2\ne 0$ and 
$(|g_1|^2-1)(|g_2|^2-1)>0$\,. Associate to $(g_1,g_2)$ the solution to the system of natural equations
\eqref{Nat_Eq_K_kappa_R42}, obtained by formulas \eqref{K_kappa_g1g2_Can1_R42}.
Theorem \ref{Nat_Eq_K_kappa_solv_g1g2_R42} and Theorem \ref{Nat_eq_same_K_kappa_hatg_g_R42} imply that this
correspondence induces a map from $\mathbf{H}_{\RR^4_2}$ into $\mathbf{SNE}_{\RR^4_2}$, which is bijection.

It is easily seen that the map from $\mathbf{H}_{\RR^4_2}$ into $\mathbf{SNE}_{\RR^4_2}$ is a composition of 
the previous two maps. The above three maps diagrammatically are presented in Figure \ref{Diag_MS_SNE_H_R42}.
The results obtained above can be summarized in the following:
\begin{figure}
		\centering
			\includegraphics[width=0.33\textwidth]{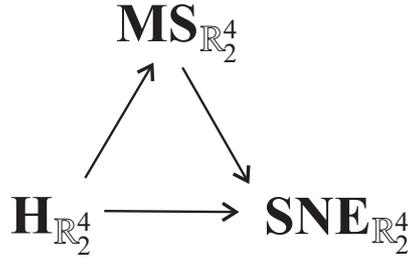}  
		\vspace*{-1ex}	
		\caption{Commutative diagram of bijections}
		\label{Diag_MS_SNE_H_R42}
\end{figure}\FloatBarrier

\begin{theorem}\label{Thm-Comut_Diag_R42}
The diagram in the Figure \ref{Diag_MS_SNE_H_R42} is commutative and the three maps are bijections\,. 
\end{theorem}


\section{Correspondence between maximal space-like surfaces in $\RR^4_2$ and pairs of maximal space-like surfaces 
in $\RR^3_1$}\label{sect_corresp_R42_R31}

  At the end of the previous section we established a correspondence between the set $\mathbf{MS}_{\RR^4_2}$ of
the equivalence classes of maximal space-like surfaces in $\RR^4_2$ and the set $\mathbf{H}_{\RR^4_2}$ of
the equivalence classes of pairs of holomorphic functions. Next we consider the analogous correspondence in
$\RR^3_1$. The necessary formulas are obtained in \cite{G} and \cite{G-K-K}, but these formulas can be obtained
also as a special case from the formulas for the surfaces in $\RR^4_2$ in this work. For that purpose we identify
the space $\RR^3_1$ with the subspace $x_4=0$ in $\RR^4_2$, considering in this way the surfaces in $\RR^3_1$ 
as surfaces in $\RR^4_2$.

  First we consider the question of canonical coordinates. We saw that the canonical coordinates of the first type 
in $\RR^3_1$ are characterized by equalities \eqref{Can_Princ_R31}. This means that, the canonical coordinates 
are principal and the coefficients of the second fundamental form are $L=-N=\pm 1$ and $M=0$\,.
Under a change of the canonical coordinates of the type\; $t=\pm \ii s+c$\; or\; $t=\pm \bar s + c$\; 
the coefficients of the second fundamental form change the sign. Consequently, always there exist coordinates, 
such that $L=-N= 1$ and $M=0$\,. We call these coordinates canonical coordinates for maximal space-like surfaces 
in $\RR^3_1$. The admissible changes for the canonical coordinates are of the type\;
$t=\pm s+c$\; or\; $t=\pm \ii \bar s + c$. Since under an improper motion in $\RR^3_1$ the coefficients of 
the second fundamental form change the sign, then the changes of the type\; $t=\pm \ii s+c$\; or\; $t=\pm \bar s + c$\; 
give the canonical coordinates of the surface, obtained by such a motion.

 Now we proceed to the natural equation. Let $\M$ be a maximal space-like surface of general type in $\RR^3_1$.
Since we identify $\RR^3_1$ with the subspace $x_4=0$ in $\RR^4_2$, then Theorem 
\ref{Thm-Min_Surf_in_hyperspace_R42}\, is valid for $\M$ and we have $\mu=0$\,. Applying this equality to the system 
\eqref{Nat_Eq_nu_mu_R42}, we obtain that the normal curvature $\nu$ of $\M$ satisfies the following PDE:
\begin{equation}\label{Nat_Eq_nu_R31}
\Delta\ln \nu - 2\nu = 0\,.
\end{equation} 
We call the above equation a \emph{natural equation} of the maximal space-like surfaces in $\RR^3_1$.
In the proof of Theorem \ref{Thm-Min_Surf_in_hyperspace_R42} we also established, that the functions in the 
Weierstrass representation of $\M$, considered as a surface in $\RR^4_2$, satisfy the condition $g_1=g_2$.
Applying this condition to \eqref{W_Can1_polinom_R42} we get that the surface $\M$ is locally given in 
$\RR^3_1$, through the formula: 
\begin{equation}\label{W_Can_Princ_polinom_R31}
\Phi=\left(\, \frac{1}{2} \frac{g^2+1}{g'} \,,\, \frac{\ii}{2} \frac{g^2-1}{g'} \,,\, \frac{g}{g'} \,\right),
\end{equation}
where the function $g$ satisfies the conditions $g'\neq 0$ and $|g|\neq 1$\,.
We call the above representation a \textit{canonical Weierstrass representation} for maximal space-like surfaces 
of general type in $\RR^3_1$.

  Since the representation \eqref{W_Can_Princ_polinom_R31} is a special case of \eqref{W_Can1_polinom_R42},
then the transformation formulas for $g$ under the different transformations of $\M$ in $\RR^3_1$ can be 
obtained from the corresponding formulas in $\RR^4_2$. Thus, under a holomorphic change of the canonical 
coordinates it follows from \eqref{g_s-hol_R42} that the function $g$ is transformed by:
\begin{equation}\label{g_s-hol_Can_Princ_R31}
\tilde g(s) = g(\pm s+c).
\end{equation}
Under an anti-holomorphic change it follows from \eqref{g_s-antihol_R42}, respectively:
\begin{equation}\label{g_s-antihol_Can_Princ_R31}
\tilde g(s) = \frac{1}{\bar g(\pm\ii \bar s + c)}.
\end{equation}

 Formulas \eqref{hatg_g-orthchr_move_Can1_R42} and \eqref{hatg_g-nonorthchr_move_Can1_R42} give that
under a proper motion in $\RR^3_1$ the function $g$ is transformed by:
\begin{equation}\label{hatg_g-move_Can_Princ_R31}
\hat g = \frac{a g + \bar b}{b g + \bar a}\:;
\qquad a,b \in \CC\,; \  |a|^2-|b|^2=\pm 1 \,,
\end{equation}
where $|a|^2-|b|^2= +1$ in the case of an orthochronous motion, and $|a|^2-|b|^2= -1$ in the case of 
a non-orthochronous motion.

 In the case of an improper motion in $\RR^3_1$, as we mentioned in the beginning, there has to be made
a change of the coordinates of the type:\; $t=\pm \ii s+c$\; or\; $t=\pm \bar s + c$. Combining equalities
\eqref{hatg_g-unp_nonorthchr_move_Can1_R42}, \eqref{hatg_g-unp_orthchr_move_Can1_R42},
\eqref{g_s-hol_R42} and \eqref{g_s-antihol_R42} we have, that under an improper motion the function $g$ 
is transformed through anyone of the following formulas:
\medskip
\begin{equation}\label{hatg_g-unp_move_t_is_bs_Can_Princ_R31}
\hat g(s) = \frac{a g(\pm\ii s+c) + \bar b}{b g(\pm\ii s+c) + \bar a} \quad \text{or} \quad 
\hat g(s) = \frac{a \bar g(\pm\bar s+c) + \bar b}{b \bar g(\pm\bar s+c) + \bar a}\,;
\qquad a,b \in \CC\,; \  |a|^2-|b|^2=\pm 1 \,.
\end{equation}
\medskip

 In the case of a homothety with coefficient $k$ it follows from \eqref{hat_g-g-homotet_Can1_R42}, that:
\begin{equation}\label{hat_g-g-homotet_Can_Princ_R31}
\hat g(s) = g\left(\frac{1}{\sqrt{k}}s\right).
\end{equation}

 Finally, for the family $\M_\theta$ of associated maximal space-like surfaces to a given one, we have from
\eqref{g_1-param_family_Can1_R42} the following:
\begin{equation}\label{g_1-param_family_Can_Princ_R31}
g_\theta (s) = g(\e^{\ii\frac{\theta}{2}} s) \,.
\end{equation}

\smallskip

 Let $\M$ be a maximal space-like surface of general type in $\RR^3_1$, given by \eqref{W_Can_Princ_polinom_R31}.
Applying the equality $g_1=g_2$ to formula \eqref{E_g1g2_Can1_R42} we find that the coefficient $E$ of
the first fundamental form of $\M$, in canonical coordinates is expressed by $g$ in the following way:
\begin{equation}\label{E-g_Can_Princ_R31}
E=\frac{(|g|^2-1)^2}{4|g'|^2}\:.
\end{equation}

Respectively, applying the equality $g_1=g_2$ to formula \eqref{nu_mu_g1g2_Can1_R42} we get that the normal
curvature $\nu$ of $\M$ is expressed by $g$ in the following way:
\begin{equation}\label{nu_g_Can_Princ_R31}
\nu = \ds\frac{4|g'|^2}{(|g|^2-1)^2}\:.
\end{equation}
Consequently, the formula \eqref{nu_g_Can_Princ_R31} describes locally all solutions of the natural equation
\eqref{Nat_Eq_nu_R31}. Furthermore, two different functions $g$ and $\hat g$ generate one and the same
solution if and only if they are related by an equation of the type \eqref{hatg_g-move_Can_Princ_R31}.

\smallskip

Now, similarly to the case of surfaces in $\RR^4_2$, we define natural correspondences between the following 
equivalence classes of objects: classes of maximal space-like surfaces of general type
in $\RR^3_1$, classes of solutions to the natural equation of the maximal space-like surfaces and 
classes of holomorphic functions in $\CC$. These correspondences can be defined only locally. Further, we 
introduce some additional definitions and denotations, similar to those in $\RR^4_2$. 

 Consider the set of maximal space-like surfaces of general type in $\RR^3_1$ of the type $(\D,\x)$,
where $\D$ is a neighborhood of zero in $\RR^2\equiv\CC$. We suppose as usual, that the surface is 
parametrized by canonical coordinates. Two surfaces $(\D,\x)$ and $(\hat\D,\hat\x)$ of this type are said 
to be \emph{equivalent}, if there exists a neighborhood of zero $\D_0\subset\D\cap\hat\D$, such that the 
two surfaces are related in $\D_0$ by a proper motion in $\RR^3_1$ of the type $\hat\x = A\x + \vb$, 
where $A \in \mathbf{SO}(2,1,\RR)$ and $\vb \in \RR^3_1$.
The set of the equivalence classes of this relation is denoted by $\mathbf{MS}_{\RR^3_1}$.

 Next we consider the set of the solutions of the natural equation \eqref{Nat_Eq_nu_R31} of the maximal 
space-like surfaces in $\RR^3_1$ of the type $(\D,\nu)$, where the domain $\D$ of $\nu$ is a neighborhood of zero in 
$\RR^2\equiv\CC$. Two solutions $(\D,\nu)$ and $(\hat\D,\hat\nu)$ are said to be \emph{equivalent}, if there exists
a neighborhood of zero $\D_0\subset\D\cap\hat\D$, such that both solutions coincide in $\D_0$.
The set of the equivalence classes of this relation is denoted by $\mathbf{SNE}_{\RR^3_1}$.

 Finally we consider the set of holomorphic functions of the type $(\D,g)$, where the domain $\D$ of
the function $g$ is a neighborhood of zero in $\RR^2\equiv\CC$, and the holomorphic function $g$
satisfies the conditions $g' \neq 0$ and $|g|\neq 1$. Two functions $(\D,g)$ and $(\hat\D,\hat g)$ are
said to be \emph{equivalent}, if there exists a neighborhood of zero $\D_0\subset\D\cap\hat\D$, such that
both functions are related in $\D_0$ by a linear fractional transformation of the type 
\eqref{hatg_g-move_Can_Princ_R31}. The set of the equivalence classes of this relation is denoted by 
$\mathbf{H}_{\RR^3_1}$.
\begin{figure}
		\centering
			\includegraphics[width=0.33\textwidth]{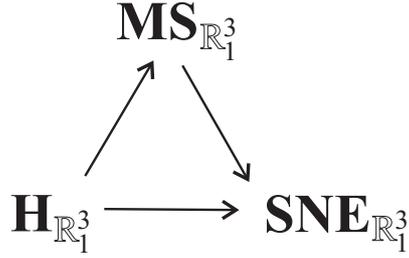}  
		\vspace*{-1ex}	
		\caption{Commutative diagram of bijections}
		\label{Diag_MS_SNE_H_R31}
\end{figure}\FloatBarrier

 Using the formulas for the maximal space-like surfaces of general type in $\RR^3_1$, obtained above, we define 
similarly to the case in $\RR^4_2$ correspondences between the three types of objects.  
These correspondences are presented diagrammatically in Figure \ref{Diag_MS_SNE_H_R31} and analogously to Theorem
\ref{Thm-Comut_Diag_R42} we have:
\begin{theorem}\label{Thm-Comut_Diag_R31}
The diagram in the Figure \ref{Diag_MS_SNE_H_R31} is commutative and the three maps are bijections\,. 
\end{theorem}

Let us compare the definitions of $\mathbf{H}_{\RR^4_2}$ and $\mathbf{H}_{\RR^3_1}$. Two pairs of holomorphic functions
$(g_1,g_2)$ and $(\hat g_1,\hat g_2)$ generate one and the same element in $\mathbf{H}_{\RR^4_2}$ if and only if they
are related by a linear fractional transformation of the type \eqref{hatg_g-orthchr_move_Can1_R42} or 
\eqref{hatg_g-nonorthchr_move_Can1_R42}, and both linear fractional transformations are independent one from the other.
Then $g_1$ and $\hat g_1$ generate one and the same element in $\mathbf{H}_{\RR^3_1}$, and respectively
$g_2$ and $\hat g_2$ also generate one and the same element in $\mathbf{H}_{\RR^3_1}$. This means that we have
a map from the set $\mathbf{H}_{\RR^4_2}$ into the Cartesian product of tho copies of the set $\mathbf{H}_{\RR^3_1}$, 
and the so determined map is injective. Now, let $g_1$ and $g_2$ generate two arbitrary elements of $\mathbf{H}_{\RR^3_1}$.
If $(|g_1|^2-1)(|g_2|^2-1)>0$\,, then the pair $(g_1,g_2)$ generates an element of $\mathbf{H}_{\RR^4_2}$,
which is transformed into the element of $\mathbf{H}_{\RR^3_1} \times \mathbf{H}_{\RR^3_1}$, generated by $g_1$ and $g_2$.
If we have $(|g_1|^2-1)(|g_2|^2-1)<0$\,, then we can replace $g_2$ with $\ds\frac{1}{g_2}$, according to
\eqref{hatg_g-move_Can_Princ_R31}. Then the pair $\left(g_1,\ds\frac{1}{g_2}\right)$ generates an element of 
$\mathbf{H}_{\RR^4_2}$, which is transformed into the element of $\mathbf{H}_{\RR^3_1} \times \mathbf{H}_{\RR^3_1}$, 
generated by $g_1$ and $g_2$. From these arguments it follows that we can identify the set $\mathbf{H}_{\RR^4_2}$ 
with the Cartesian product of two copies of the set $\mathbf{H}_{\RR^3_1}$:
\begin{equation*}
\mathbf{H}_{\RR^4_2} \equiv \mathbf{H}_{\RR^3_1} \times \mathbf{H}_{\RR^3_1}\,.
\end{equation*}

From here and from the triangle diagrams in Figure \ref{Diag_MS_SNE_H_R42} and Figure \ref{Diag_MS_SNE_H_R31},
which consist of bijections, we also obtain the following correspondences:
\begin{equation*}
\mathbf{SNE}_{\RR^4_2} \equiv \mathbf{SNE}_{\RR^3_1} \times \mathbf{SNE}_{\RR^3_1}\,; \qquad
\mathbf{MS}_{\RR^4_2} \equiv \mathbf{MS}_{\RR^3_1} \times \mathbf{MS}_{\RR^3_1}\,.
\end{equation*}

 Our next goal is to obtain the above correspondences explicitly. We begin with the correspondence between
the solutions of the system of natural equations \eqref{Nat_Eq_K_kappa_R42} of the maximal space-like surfaces
in $\RR^4_2$ and the pairs of solutions to the natural equation \eqref{Nat_Eq_nu_R31} of the maximal space-like 
surfaces in $\RR^3_1$. Let $(\D,K>0,\varkappa)$ be a solution to \eqref{Nat_Eq_K_kappa_R42}, where $\D$ is
the domain in $\RR^2$ of the functions $K$ and $\varkappa$. Then the pair $(K,\varkappa)$ is given by formulas
\eqref{K_kappa_g1g2_Can1_R42}, where $(g_1,g_2)$ are holomorphic functions, satisfying the conditions
$g'_1g'_2\ne 0$ and $(|g_1|^2-1)(|g_2|^2-1)>0$.
Consider the solutions $\nu_1$ and $\nu_2$ to the natural equation \eqref{Nat_Eq_nu_R31},
obtained through \eqref{nu_g_Can_Princ_R31} from $g_1$ and $g_2$, respectively. Comparing formulas 
\eqref{K_kappa_g1g2_Can1_R42} and \eqref{nu_g_Can_Princ_R31} it follows that the pair $(K,\varkappa)$ is 
expressed through the pair $(\nu_1,\nu_2)$ in the following way:
\begin{equation}\label{K_kappa_nu1nu2_Can1_R42}
		 K=           \frac{1}{2} \sqrt{\nu_1\,\nu_2}\,(\nu_1+\nu_2)\,; \qquad 
		 \varkappa = -\frac{1}{2} \sqrt{\nu_1\,\nu_2}\,(\nu_1-\nu_2)\,.
\end{equation}

 These formulas are obtained locally but it is easily seen that, using these formulas we can express $\nu_1$ 
and $\nu_2$ through $K$ and $\varkappa$ uniquely in the whole domain $\D$. For the purpose we get from
\eqref{K_kappa_nu1nu2_Can1_R42}:
\begin{equation*}
\begin{array}{llr}
		 K-\varkappa &=& \sqrt{\nu_1^3\,\nu_2}\,;\\[1ex] 
		 K+\varkappa &=& \sqrt{\nu_1\,\nu_2^3}\,;
\end{array}	 \quad  \Leftrightarrow \quad  
\begin{array}{lll}
		 K^2-\varkappa^2 &=& \nu_1^2\,\nu_2^2  \,;\\[1ex]  
		 \ds\frac{K-\varkappa}{K+\varkappa} &=& \ds\frac{\nu_1}{\nu_2}\,;
\end{array}  \quad  \Leftrightarrow \quad
\begin{array}{lll}
		 \ds\frac{(K-\varkappa)^3}{K+\varkappa} &=& \nu_1^4 \,;\\[2ex]  
		 \ds\frac{(K+\varkappa)^3}{K-\varkappa} &=& \nu_2^4 \,.
\end{array}
\end{equation*}
From here we express $\nu_1$ and $\nu_2$ through $K$ and $\varkappa$:
\begin{equation}\label{nu1_nu2_K_kappa_Can1_R42}
\nu_1 = \ds\frac{K-\varkappa}{\sqrt[4]{\vphantom{{2^2}^2}K^2-\varkappa^2}}\;; \qquad
\nu_2 = \ds\frac{K+\varkappa}{\sqrt[4]{\vphantom{{2^2}^2}K^2-\varkappa^2}}\;.
\end{equation}
The last formulas show that $\nu_1$ and $\nu_2$ are defined in the whole domain $\D$ and therefore the 
formulas \eqref{K_kappa_nu1nu2_Can1_R42} are valid in the whole domain $\D$. Conversely, if $\nu_1$ and $\nu_2$ are
two arbitrary solutions to the natural equation \eqref{Nat_Eq_nu_R31}, then they have locally the form 
\eqref{nu_g_Can_Princ_R31} for given holomorphic functions $g_1$ and $g_2$. Besides we can suppose that
$(|g_1|^2-1)(|g_2|^2-1)>0$, because otherwise we can change $g_2$ with $\ds\frac{1}{g_2}$. Then defining 
$K>0$ and $\varkappa$ through \eqref{K_kappa_nu1nu2_Can1_R42}, they will satisfy \eqref{K_kappa_g1g2_Can1_R42}. 
Consequently, $K$ and $\varkappa$ will give locally and hence globally a solution to the system
\eqref{Nat_Eq_K_kappa_R42}. These arguments give the following:
\begin{theorem}\label{Thm-K_kappa_nu1nu2_Can1_R42}
Let the pair of functions $(K>0,\varkappa)$, defined in $\D\subset\RR^2$, be a solution to the system
of natural equations \eqref{Nat_Eq_K_kappa_R42} of the maximal space-like surfaces in $\RR^4_2$.
Then there exist two solutions $\nu_1$ and $\nu_2$ to the natural equation \eqref{Nat_Eq_nu_R31} 
of the maximal space-like surfaces in $\RR^3_1$, defined in the whole domain $\D$ such that
the equations \eqref{K_kappa_nu1nu2_Can1_R42} are valid. 
The functions $\nu_1$ and $\nu_2$ are determined uniquely by  $K$ and $\varkappa$ through equalities 
\eqref{nu1_nu2_K_kappa_Can1_R42}.
Conversely, if $\nu_1$ and $\nu_2$ are two arbitrary solutions to the natural equation \eqref{Nat_Eq_nu_R31},
then equalities \eqref{K_kappa_nu1nu2_Can1_R42} give a solution $(K>0,\varkappa)$ to the
system \eqref{Nat_Eq_K_kappa_R42}.
\end{theorem}
\begin{remark}
 Note that the formulas \eqref{K_kappa_nu1nu2_Can1_R42} can be obtained directly from the system
\eqref{Nat_Eq_K_kappa_R42} by algebraic transformations without using the Weierstrass representations. This approach 
was used in the paper \cite{G-K-3}.
\end{remark}

 Now we proceed to the correspondence between the maximal space-like surfaces of general type in $\RR^4_2$ and
the pairs of maximal space-like surfaces of general type in $\RR^3_1$. This correspondence is obtained with 
the help of the canonical Weierstrass representations in $\RR^4_2$ and $\RR^3_1$ of the type 
\eqref{W_Can1_polinom_R42} and \eqref{W_Can_Princ_polinom_R31}, respectively. Since the canonical coordinates
exist in general only locally in a neighborhood of a point, then this correspondence is only local. That is why
we consider the pairs $(\M,p)$, where $\M$ is a maximal space-like surface of general type and $p$ is a fixed 
point in $\M$. Furthermore, we suppose that we have fixed canonical coordinates in a neighborhood of $p$. 
As usual, we identify two such pairs $(\M,p)$ and $(\hat\M, \hat p)$, if they are related by a proper motion,
$p$ and $\hat p$ being correspondent under this motion.

 Thus, let $(\M,p)$ be a maximal space-like surface of general type in $\RR^4_2$ and let
$t=u+\ii v$ give canonical coordinates in a neighborhood of $p$, and $t_0$ give the coordinates of $p$. 
In view of Theorem \ref{thm-W_Can1_polinom_R42}\,, the surface $\M$ has a canonical Weierstrass representation
of the type \eqref{W_Can1_polinom_R42} through a pair of holomorphic functions $g_1$ and $g_2$. Anyone of 
the functions $g_1$ and $g_2$ generates by the canonical Weierstrass representation 
\eqref{W_Can_Princ_polinom_R31} a maximal space-like surface of general type in $\RR^3_1$. Denote by $\M_1$ 
and $\M_2$ the obtained surfaces in $\RR^3_1$ and by $p_1$ and $p_2$ respectively the points on them with 
coordinate $t=t_0$. In this way we associated to $(\M,p)$ two surfaces $(\M_1,p_1)$ and $(\M_2,p_2)$. 
If $(\M,p)$ and $(\hat\M, \hat p)$ are related by a proper motion in $\RR^4_2$, then the corresponding functions
in the representation \eqref{W_Can1_polinom_R42} are related by a linear fractional transformation of the type
\eqref{hatg_g-orthchr_move_Can1_R42} or \eqref{hatg_g-nonorthchr_move_Can1_R42}. Then equality 
\eqref{hatg_g-move_Can_Princ_R31} gives that the corresponding surfaces in $\RR^3_1$ are also related by proper 
motions. Consequently the correspondence under consideration is invariant when identifying two surfaces related by 
a proper motion in $\RR^4_2$.

 Next we prove that the correspondence is reversible. For the purpose, let $(\M_1,p_1)$ and $(\M_2,p_2)$ be two
maximal space-like surfaces of general type in $\RR^3_1$ with fixed points $p_1$ and $p_2$, respectively.
Suppose that in neighborhoods of $p_1$ and $p_2$ are introduced principal canonical coordinates on $\M_1$ and
$\M_2$, respectively, and $t_0$ gives the coordinates of the points $p_1$ and $p_2$. Any of the surfaces
$\M_1$ and $\M_2$ has canonical Weierstrass representation of the type \eqref{W_Can_Princ_polinom_R31} in 
neighborhoods of the points $p_1$ and  $p_2$ by given holomorphic functions $g_1$ and $g_2$. As above, 
we can assume that $(|g_1|^2-1)(|g_2|^2-1)>0$, because otherwise we can again replace $g_2$ with 
$\ds\frac{1}{g_2}$\,. Then the pair of functions $(g_1,g_2)$ generates a maximal space-like surface of general 
type $\M$ in $\RR^4_2$ by the canonical Weierstrass representation of the type \eqref{W_Can1_polinom_R42}. 
Furthermore, let $p$ be the point in $\M$ with coordinates $t=t_0$. Thus, we associated to the surfaces $(\M_1,p_1)$ 
and $(\M_2,p_2)$ in $\RR^3_1$ a surface $(\M,p)$ in $\RR^4_2$. Let $(\hat\M_1,\hat p_1)$ and $(\hat\M_2,\hat p_2)$ 
be obtained respectively from $(\M_1,p_1)$ and $(\M_2,p_2)$ by proper motions in $\RR^3_1$. Then the generating 
functions $\hat g_j(s)$ for $(\hat\M_j,\hat p_j)$ are obtained from $g_j(t)$ through formulas of the type
\eqref{hatg_g-move_Can_Princ_R31}. Moreover, the conditions\, $(|g_1|^2-1)(|g_2|^2-1)>0$\, and\, 
$(|\hat g_1|^2-1)(|\hat g_2|^2-1)>0$\, guarantee that the sign of the expression $|a_j|^2-|b_j|^2$ in
\eqref{hatg_g-move_Can_Princ_R31} is the same in the cases $j=1$ and $j=2$\,, as it is in the equations
\eqref{hatg_g-orthchr_move_Can1_R42} and \eqref{hatg_g-nonorthchr_move_Can1_R42}. Therefore, the correspondence 
from $\RR^3_1$ to $\RR^4_2$ is invariant under the identification of the surfaces related by a proper motion in 
$\RR^3_1$. It is clear that the last correspondence is the inverse of the correspondence from $\RR^4_2$ to $\RR^3_1$. 

The above correspondence between maximal space-like surfaces of general type in $\RR^4_2$ and the pairs of maximal 
space-like surfaces of general type in $\RR^3_1$ we denote in the following way:
\begin{equation}\label{Mp_R42-M1p1_M2p2_R31}
		\big(\M,p\big) \leftrightarrow \big((\M_1,p_1),(\M_2,p_2)\big) \,.
\end{equation}

\smallskip

 Next we consider the properties of this correspondence under a change of the canonical coordinates and the basic
geometric transformations of the surfaces in $\RR^4_2$ and $\RR^3_1$. First we consider a \emph{change of the coordinates}.
As we mentioned at the beginning of this section, there are two changes in $\RR^3_1$. Under a change of the type 
$t=\pm s+c$, the functions $g_j(t)$ in $\RR^3_1$ are replaced by $g_j(\pm s+c)$, according to \eqref{g_s-hol_Can_Princ_R31}. 
Formula \eqref{g_s-hol_R42} shows that the functions $g_j(t)$ are replaced in the same way in $\RR^4_2$. Under a
change of the type $t=\pm\ii \bar s + c$, the functions $g_j(t)$ are also replaced in the same way in $\RR^3_1$
and in $\RR^4_2$, according to \eqref{g_s-antihol_Can_Princ_R31} and \eqref{g_s-antihol_R42}. Consequently,
the correspondence \eqref{Mp_R42-M1p1_M2p2_R31} is invariant under a change of the canonical coordinates in $\RR^3_1$.

 In $\RR^4_2$ there are two more changes of the canonical coordinates\; $t=\pm \ii s+c$\; and\; $t=\pm \bar s + c$\;,
according to Theorem \ref{Can_Coord-uniq}\,. We saw at the beginning of this section, that these changes give in $\RR^3_1$
coordinates of a surface, obtained by an improper motion. These changes will be considered further, considering 
improper motions.

 Next we proceed to \emph{motions} in $\RR^3_1$ and $\RR^4_2$. As we noted, the functions $g_j(t)$ are transformed  
in one and the same way under proper motions in $\RR^3_1$ and $\RR^4_2$. Consequently the correspondence 
\eqref{Mp_R42-M1p1_M2p2_R31} is invariant under proper motions in $\RR^3_1$ and $\RR^4_2$. 

 Now, consider the case of \emph{improper motions}. Let $(\tilde\M_1,\tilde p_1)$ and $(\tilde\M_2,\tilde p_2)$ 
be obtained respectively from  $(\M_1,p_1)$ and $(\M_2,p_2)$ by improper motions in $\RR^3_1$. In order to obtain 
canonical coordinates on $(\tilde\M_j,\tilde p_j)$, we make change of the type\; $t=\pm \ii s+c$\; 
or\; $t=\pm \bar s + c$\,. Then, the generating functions $\tilde g_j(s)$ for $(\tilde\M_j,\tilde p_j)$ are obtained
from $g_j(t)$ by some of the formulas \eqref{hatg_g-unp_move_t_is_bs_Can_Princ_R31}. The conditions
$(|g_1|^2-1)(|g_2|^2-1)>0$\, and\, $(|\tilde g_1|^2-1)(|\tilde g_2|^2-1)>0$\, again guarantee that the sign
of the expression $|a_j|^2-|b_j|^2$ in \eqref{hatg_g-unp_move_t_is_bs_Can_Princ_R31} is the same for $j=1$ and $j=2$\,,
as in the formulas \eqref{hatg_g-orthchr_move_Can1_R42} and \eqref{hatg_g-nonorthchr_move_Can1_R42}.
On the other hand, the new coordinates $s$ are canonical for the same surface $(\M,p)$ in $\RR^4_2$ and 
the functions $\tilde g_j(s)$ generate a surface, obtained by a \emph{proper} motion from$(\M,p)$, 
according to \eqref{g_s-hol_R42}, \eqref{g_s-antihol_R42}, \eqref{hatg_g-orthchr_move_Can1_R42} and 
\eqref{hatg_g-nonorthchr_move_Can1_R42}. Consequently, \eqref{Mp_R42-M1p1_M2p2_R31} implies that:
\begin{equation}\label{Mp_R42-M1p1_M2p2_unp_move_R31}
		\big(\M,p\big) \leftrightarrow \big((\tilde\M_1,\tilde p_1),(\tilde\M_2,\tilde p_2)\big) \leftrightarrow 
		\big((\M_1,p_1),(\M_2,p_2)\big) \,.
\end{equation}

 Let now $(\tilde\M,\tilde p)$ be obtained by an \emph{improper motion} in $\RR^4_2$ from $(\M,p)$.
In view of formulas \eqref{hatg_g-unp_nonorthchr_move_Can1_R42} and \eqref{hatg_g-unp_orthchr_move_Can1_R42} 
we know that, this means displacement of the generating functions $g_1$ and $g_2$, which is equivalent to the
displacement of $(\M_1,p_1)$ and $(\M_2,p_2)$. Then \eqref{Mp_R42-M1p1_M2p2_R31} implies that:
\begin{equation}\label{Mp_R42-M1p1_M2p2_unp_move_R42}
		\big(\tilde\M,\tilde p\big) \leftrightarrow \big((\M_2,p_2),(\M_1,p_1)\big) \leftrightarrow 
		\big((\tilde\M_2,\tilde p_2),(\tilde\M_1,\tilde p_1)\big) \,.
\end{equation}

 Note that, if we change the canonical coordinates on only one of the surfaces in $\RR^3_1$, then by the 
correspondence \eqref{Mp_R42-M1p1_M2p2_R31} it is obtained in general a new surface in $\RR^4_2$. In a
similar way, if we have an improper motion of the only one of the surfaces in $\RR^3_1$, say $(\M_1,p_1)$, 
then the pair $(\tilde\M_1,\tilde p_1)$ and $(\M_2,p_2)$ give a new surface in $\RR^4_2$, different from both: 
$(\M,p)$ and $(\tilde\M,\tilde p)$. 

\smallskip

 Now, proceed to the case of \emph{homotheties} with coefficient $k$ in $\RR^3_1$ and $\RR^4_2$. 
Taking into account formulas \eqref{hat_g-g-homotet_Can_Princ_R31} and \eqref{hat_g-g-homotet_Can1_R42} we
conclude that the generating functions $g_j$ are transformed in one and the same way in $\RR^3_1$ and 
$\RR^4_2$:\; $\hat g_j(s) = g_j\left(\frac{1}{\sqrt{k}}s\right)$. This means that the correspondence
\eqref{Mp_R42-M1p1_M2p2_R31} is invariant under a homothety. Introduce the denotation $(k\M,kp)$ for the 
surface, obtained by a homothety from $(\M,p)$ in $\RR^3_1$ or  $\RR^4_2$, where $kp$ is the image of $p$ 
under this homothety. Then \eqref{Mp_R42-M1p1_M2p2_R31} implies that:
\begin{equation}\label{Mp_R42-M1p1_M2p2_R31_homotet}
		\big(k\M,kp\big) \leftrightarrow \big((k\M_1,kp_1),(k\M_2,kp_2)\big) \,.
\end{equation}

\smallskip

 Let us consider the \emph{one-parameter families} of associated maximal space-like surfaces of $(\M_1,p_1)$ and $(\M_2,p_2)$ 
in $\RR^3_1$ and of $(\M,p)$ in $\RR^4_2$. Taking into account formulas \eqref{g_1-param_family_Can_Princ_R31} 
and \eqref{g_1-param_family_Can1_R42} we see that the generating functions $g_j$ are obtained in one and 
the same way in $\RR^3_1$ and $\RR^4_2$:\; $g_{j|\theta} (s) = g_j(\e^{\ii\frac{\theta}{2}} s)$.
This means that the correspondence \eqref{Mp_R42-M1p1_M2p2_R31} preserves the one-parameter families of associated 
maximal space-like surfaces. Let $p_\theta=\mathcal{F}_\theta(p)$ denote the point in $\M_\theta$, 
corresponding to $p$ under the standard isometry $\mathcal{F}_\theta$ between $\M$ and $\M_\theta$, 
defined in Proposition \ref{Isom_M_phi-M}\,. Then we have:
\begin{equation}\label{Mp_R42-M1p1_M2p2_R31_1-param}
		\big(\M_\theta,p_\theta\big) \leftrightarrow \big((\M_{1|\theta},p_{1|\theta}),(\M_{2|\theta},p_{2|\theta})\big) \,.
\end{equation}

 From the notes after Definition \ref{1-param_family_assoc_surf} we know that the \emph{conjugate} maximal space-like 
surface $\bar\M$ of $\M$ belongs to the one-parameter family of associated maximal space-like surfaces to $\M$ 
and is obtained by $\theta=\frac{\pi}{2}$. Let $\bar p$ denote the point in $\bar\M$, corresponding to $p$ in $\M$. 
Then \eqref{Mp_R42-M1p1_M2p2_R31_1-param} implies that:
\begin{equation}\label{Mp_R42-M1p1_M2p2_R31_conj}
		\big(\bar\M,\bar p\big) \leftrightarrow \big((\bar\M_1,\bar p_1),(\bar\M_2,\bar p_2)\big) \,.
\end{equation}

\smallskip

 Finally we show how the scalar invariants of $(\M,p)$ in $\RR^4_2$ are related to the corresponding invariants of
$(\M_1,p_1)$ and $(\M_2,p_2)$ in $\RR^3_1$. Let $E$ be the coefficient of the first fundamental form of $\M$, 
with respect to the canonical coordinates on $\M$ and let $E_1$ and $E_2$ be the corresponding coefficients
for $\M_1$ and $\M_2$. Then $E$ satisfies \eqref{E_g1g2_Can1_R42}, and $E_1$, $E_2$ satisfy \eqref{E-g_Can_Princ_R31}. 
Comparing these formulas we see that $E$ is the  \emph{geometric mean} of $E_1$ and $E_2$:
\begin{equation}\label{Mp_R42-M1p1_M2p2_R31_E_E12}
  E=\sqrt{E_1E_2}\;.
\end{equation}

The curvatures $(K,\varkappa)$ of $\M$ satisfy \eqref{K_kappa_g1g2_Can1_R42}, and the curvatures $\nu_1$ and $\nu_2$
of $\M_1$ and $\M_2$ satisfy respectively \eqref{nu_g_Can_Princ_R31}. As we have already seen, from here follow formulas 
\eqref{K_kappa_nu1nu2_Can1_R42} and \eqref{nu1_nu2_K_kappa_Can1_R42}, giving the relation between the pairs
$(K,\varkappa)$ and $(\nu_1,\nu_2)$.

\smallskip
\textbf{Acknowledgements}

The first author is partially supported by the National Science Fund, Ministry of Education
and Science of Bulgaria under contract DN 12/2.


\end{document}